\documentclass[a4paper]{article}

\usepackage[all]{xy}
\usepackage{amssymb,amsmath}
\usepackage[only,mapsfrom,heavycircles]{stmaryrd}
\usepackage{latexsym}
\usepackage{graphicx}
\usepackage{mathrsfs}   
\usepackage{titlesec} 
\usepackage{theorem}
\usepackage{paralist}
\usepackage{tocloft}

\newtheorem{theorem}{Theorem}[section]
\newtheorem{lemma}[theorem]{Lemma}
\newtheorem{proposition}[theorem]{Proposition}
\newtheorem{definition}[theorem]{Definition}
\newtheorem{corollary}[theorem]{Corollary}
\newtheorem{question}[theorem]{Question}
\newtheorem{conjecture}[theorem]{Conjecture}

{\theorembodyfont{\normalfont}

\newtheorem{remark}[theorem]{Remark}

\newtheorem{example}[theorem]{Example}

\newtheorem{notation}[theorem]{Notation}

}

\newenvironment{proof}[1][Proof]{\noindent \textbf{#1.~}}
{\hfill $\Box$ \\ }

\makeatletter
\@addtoreset{equation}{section}

\makeatother

\titleformat{\section}[hang]{\normalfont\large\bfseries}{}{0cm}%
{\thesection \  --\ }
\titleformat{\subsection}[hang]{\normalfont\bfseries}{}{0cm}%
{\thesubsection\ -- \,}
\newenvironment{opening}%
{\titleformat{\section}[hang]{\normalfont\large\bfseries}{}{0cm}{}
\titleformat{\subsection}[hang]{\normalfont\bfseries}{}{0cm}{}

}{}

\newenvironment{closing}%
{\titleformat{\section}[hang]{\normalfont\large\bfseries}{}{0cm}{}
\setlength{\itemsep}{0mm}
\small}{}

\makeatletter
\renewcommand\@maketitle{%
  \newpage
  \begin{center}%
  \let \footnote \thanks
    {\Large \bf \@title \par}%
    \vskip 1em%
    {\large
      \begin{tabular}[t]{c}%
        \@author
      \end{tabular}\par}%
  \end{center}%
  \par
  \vskip 1.5em}
\makeatother

\renewenvironment{abstract}
{\small \quotation
\noindent {\bf Abstract.}}{\endquotation}

\newcommand{\C}{\mathbf{C}}

\newcommand{\N}{\mathbf{N}}
\renewcommand{\P}{\mathbf{P}}
\newcommand{\A}{\mathbf{A}}
\newcommand{\F}{\mathbf{F}}

\renewcommand{\O}{\mathcal{O}}
\renewcommand{\L}{\mathcal L}
\renewcommand{\H}{\mathrm{H}}
\newcommand{\h}{\mathrm{h}}

\newcommand{\Pic}{\mathrm{Pic}}

\newcommand{\Sym}{\mathrm{Sym}}

\newcommand{\Hilb}{\mathrm{Hilb}}
\newcommand{\lineq}{\sim} 
\newcommand{\sg}{\le} 
\newcommand{\hil}{\mathcal} 
\renewcommand{\mod}{\mathscr} 
\newcommand{\sym}{\mathfrak{S}} 
\newcommand{\ring}[1]{#1^\circ}

\renewcommand{\epsilon}{\varepsilon}

\newcommand{\utau}{\underline{\tau}}
\newcommand{\ubtau}{\underline{\boldsymbol{\tau}}}
\newcommand{\seqN}{\underline{\N}}

\newcommand{\bdelta}{\boldsymbol{\delta}}
\newcommand{\tQ}{\tilde{Q}}
\newcommand{\bj}{\bar{\jmath}}

\newcommand{\fL}{\mathfrak{L}}
\newcommand{\fV}{\mathfrak{V}}

\newcommand{\fP}{\mathfrak{P}}
\newcommand{\fO}{\mathfrak{O}}
\newcommand{\fN}{\mathfrak{N}}

\newcommand{\Hi}{ {\rm Hilb}}
\newcommand{\chitop}{\chi_{\rm top}}
\newcommand{\sV}{\mathscr{V}}

\newcommand{\p}[1]{#1^{\scriptscriptstyle +}}
\newcommand{\m}[1]{#1^{\scriptscriptstyle -}}
\newcommand{\plm}[1]{#1^{\scriptscriptstyle \pm}}
\newcommand{\pG}{\vphantom{G}'G}
\newcommand{\cru}[1]{#1^{\mathrm{cr}}}
\newcommand{\reg}[1]{#1^{\mathrm{reg}}}
\newcommand{\loc}[1]{#1^{\mathrm{loc}}}
\newcommand{\tetra}[1]{#1_{\mathrm{tetra}}}
\newcommand{\kum}[1]{#1_{\mathrm{Kum}}}

\begin{document}

\title{Limits of pluri--tangent planes to quartic surfaces}
\author{Ciro Ciliberto and Thomas Dedieu}

\maketitle

\begin{abstract}
We describe, for various degenerations $S\to \Delta$ of quartic $K3$
surfaces over the complex unit disk
(e.g., to the union of four general planes, and to a general Kummer
surface),
the limits as $t\in \Delta^*$ tends to $0$ of the Severi
varieties $V_\delta(S_t)$, parametrizing irreducible $\delta$-nodal
plane sections of $S_t$.
We give applications of this to
\begin{inparaenum}[(i)]
\item
the counting of plane nodal curves through base points in special
position,
\item
the irreducibility of Severi varieties of a general quartic surface,
and
\item
the monodromy of the universal family of rational curves on quartic
$K3$ surfaces.
\end{inparaenum}
\end{abstract}

\begin{opening}

{\small
\setlength{\cftaftertoctitleskip}{0.2cm}
\renewcommand{\cfttoctitlefont}{\normalfont\large\bfseries}
\setlength{\cftbeforesecskip}{0.05cm}
\setcounter{tocdepth}{1}
\tableofcontents
}

\section*{Introduction}

Our objective in this paper is to study the following:
\begin{question}
\label{quest}
Let $f:S\to \Delta$ be a projective family of surfaces of degree $d$ in $\P^3$, with $S$ a smooth threefold,
and $\Delta$ the complex unit disc (usually called a \emph{degeneration} of the general  
 $S_t:=f^{-1}(t)$, for $t\neq 0$, which is a smooth surface, to the \emph{central fibre} $S_0$,
 which is in general supposed to be singular).
 What are the limits of tangent, bitangent, and tritangent planes 
to $S_t$, for $t\neq 0$,  as $t$ tends to $0$?
\end{question}

Similar questions make sense also for degenerations of plane curves,
and we refer to \cite[pp. 134--135]{harris-morrison} for a glimpse on
this subject.
For surfaces, our contribution is based on  foundational
investigations by Caporaso and Harris
\cite{caporaso-harris,caporaso-harris2}, 
and independently by Ran \cite{Ran1,Ran2,Ran3},
which were both aimed at the study of the so-called
\emph{Severi varieties}, i.e. the families of irreducible plane
nodal curves of a given degree.
We have the same kind of motivation for our study;
the link with Question \ref{quest} resides in the fact that nodal
plane sections of a surface $S_t$ in $\P^3$ are cut out by those
planes that are tangent to $S_t$.

Ultimately, our interest resides in the study of Severi varieties of 
nodal curves on $K3$ surfaces.
The first interesting instance of this is the one of plane
sections of smooth quartics in $\P^3$, the latter being primitive $K3$
surfaces of genus $3$.
For this reason, we concentrate here on the case $d=4$.
We consider a couple of interesting {degenerations} of such
surfaces to quite singular degree 4 surfaces, and we answer Question
\ref {quest} in these cases. 

The present paper is of an explorative nature,
and hopefully shows, in a way we believe to be useful and instructive,
how to apply some general techniques
for answering some specific  questions.
On the way, a few related problems will be raised, which we feel
can be attacked with the same techniques.
Some of them we solve (see below), and the other ones we plan to make
the object of future research.

\medskip
Coming to the technical core of the paper, we start from the following
key observation due to Caporaso and Harris, and Ran 
(see \S\ref{s:ZRR} for a complete statement).
Assume the central fibre $S_0$ is the transverse union of two smooth
surfaces, intersecting along a smooth curve $R$.
Then the limiting plane of a family of tangent planes to the general fibre
$S_t$, for $t\neq 0$, is:
\begin{inparaenum}[(i)]
\item \label{i:node}
either a plane that is tangent to $S_0$ at a smooth point, or
\item \label{i:tangency}
a tangent plane to $R$.
\end{inparaenum}
Furthermore, the limit has to be counted with multiplicity 2 in case
\eqref{i:tangency}.

Obviously, this is not enough to deal directly with all possible
degenerations of surfaces.
Typically, one overcomes this by applying a series of base changes and
blow--ups to $S\to \Delta$, thus producing a semistable model $\tilde
S\to \Delta$ of the initial family,
such that it is possible to provide a complete answer to Question
\ref{quest} for $S\to \Delta$ by applying a suitable extended version
of the above observation to $\tilde S\to \Delta$.
We say that $S\to \Delta$ is \emph{well behaved} when it is possible
to do so, and $\tilde S\to \Delta$ is then said to be a
\emph{good model} of $S\to \Delta$.

We give in \S\ref{s:ZRR} a rather restrictive criterion to ensure that
a given semistable model is a good model, which nevertheless 
provides the inspiration for constructing a good model for a given
family.
We conjecture that there are suitable assumptions, under which a
family is  well behaved.
We do not seek such a general statement here, but rather prove various
incarnations of this principle,
thus providing a
complete answer to Question \ref{quest} for the 
degenerations we consider.
Specifically, we obtain:
\begin{theorem}
\label{T:tetrahedron}
Let $f:S\to \Delta$ be a family of general quartic surfaces in $\P^ 3$ 
degenerating to a \emph{tetrahedron} $S_0$, i.e. the union of four independent planes. The singularities of $S$ consist in
four ordinary double points on each \emph{edge} of $S_0$.
The limits in $|\O_{S_0}(1)|$ of $\delta$-tangent planes to
$S_t$, for $t\neq 0$, are:\\
\begin{inparaenum}[\normalfont {($\delta=$}1{)}]
\item
the $24$ webs of planes passing through a singular point of $S$, plus
the $4$ webs of planes passing through a \emph{vertex} of $S_0$, the latter
counted with multiplicity $3$;\\
\item
the $240$ pencils of planes passing through two double points of the
total space $S$ that do not belong to an edge of $S_0$, plus
the $48$ pencils of planes passing through a vertex of $S_0$ and a
double point of $S$ that do not belong to a common
edge of $S_0$ (with multiplicity $3$),
plus the $6$ pencils of planes containing an edge of $S_0$ (with
multiplicity $16$);\\
\item
the $1024$ planes containing three double points of $S$ but no edge of
$S_0$, plus
the $192$ planes containing a vertex of $S_0$ and two double points of
$S$, but no edge of $S_0$ (with multiplicity $3$), plus
the $24$ planes containing an edge of $S_0$ and a double point of $S$
not on this edge (with multiplicity $16$), plus
 the $4$ \emph{faces} of $S_0$ (with multiplicity $304$).
\end{inparaenum}
\end{theorem}

\begin{theorem}
\label{T:kummer}
Let $f:S\to \Delta$ be a family of general quartic surfaces
degenerating to a general \emph{Kummer surface} $S_0$.
The limits in $|\O_{S_0}(1)|$ of $\delta$-tangent planes to 
$S_t$, for $t\neq 0$, are:\\
\begin{inparaenum}[\normalfont {($\delta=$}1{)}]
\item
the dual surface $\check{S}_0$ to the Kummer (which is itself a Kummer
surface), plus the $16$ webs of planes containing a node of $S_0$ (with
multiplicity $2$);\\
\item
the $120$ pencils of planes containing two nodes of $S_0$, each
counted with multiplicity $4$;\\
\item
the $16$ planes tangent to $S_0$ along a \emph{contact conic}  (with multiplicity $80$), plus
 the $240$ planes containing exactly three nodes of $S_0$ (with
multiplicity $8$).
\end{inparaenum}
\end{theorem}

We could also answer  Question \ref{quest} for
degenerations to a general union of two smooth quadrics,
as well as to a general union of a smooth cubic and a plane;
once the much more involved degeneration to a tetrahedron is understood,
this is an exercise.
We do not dwell on this here, and  we encourage the interested reader to
treat these cases on his own, and to look for the relations between
these various degenerations. However,  a mention to the degeneration
to a \emph{double quadric} is needed,
and we treat this in \S \ref {sec:pair of quadrics}.

Apparent in the statements of Theorems \ref{T:tetrahedron} and
\ref{T:kummer} is the strong enumerative flavour of Question
\ref{quest}, and actually we need information of this kind (see
Proposition \ref{p:deg-dual}) to prove that the two families under
consideration are well behaved. Still, we hope to find a direct proof
in the future.

As a matter of fact, Caporaso and Harris' main 
goal in \cite{caporaso-harris,caporaso-harris2} is the
computation of the degrees of Severi varieties  of irreducible nodal plane curves of a given degree, which
they achieve by providing a recursive formula.
Applying the same strategy, we are able to derive the following statement (see  \S \ref{S:triangle}):

\begin{theorem}
\label{T:triangle}
Let $a,b,c$ be three independent  lines in the projective
plane, and consider a degree $12$ divisor $Z$ cut out on 
$a+b+c$ by a general quartic curve. 
The sub--linear system $\mathcal V$  of $|\O_{\P^2}(4)|$
parametrizing curves containing $Z$ has dimension $3$.

For $1\leqslant \delta\leqslant 3$, we let
$\mathcal V_\delta$ be the Zariski closure in $\mathcal V$ of the
locally closed subset parametrizing irreducible $\delta$-nodal
curves. Then $\mathcal V_\delta$ has codimension $\delta$ in $\mathcal
V$, and degree $21$ for $\delta=1$, degree $132$ for $\delta=2$,
degree $304$ for $\delta=3$.
\end{theorem}

Remarkably, one first proves a weaker version of this (in
\S\ref{S:triangle}), which is required for the proof of Theorem
\ref{T:tetrahedron}, given in \S\ref{S:4planes}.
Then, Theorem \ref{T:triangle} is a corollary of Theorem
\ref{T:tetrahedron}.

\medskip
It has to be noted that Theorems \ref{T:tetrahedron} and
\ref{T:kummer} display a rather coarse picture of
the situation. Indeed, in describing the good models of the degenerations,
we interpret all limits of nodal curves  as elements of the limit $\fO(1)$
of $|\O_{S_t}(1)|$, for $t\neq 0$, inside the relative Hilbert scheme of curves in $S$.  
We call $\fO(1)$ the \emph{limit linear system} of $|\O_{S_t}(1)|$, for $t\neq 0$
(see \S\ref{s:resolving}), which in general is no longer a $\P^ 3$, but rather a 
degeneration of it. While in $|\O_{S_0}(1)|$, which is also a limit 
of $|\O_{S_t}(1)|$, for $t\neq 0$, there are in general elements which do not correspond to curves 
(think of the plane section of the tetrahedron with one of its faces), all elements in $\fO(1)$ 
do correspond to curves, and this is the right ambient to locate the limits of nodal curves. 
So, for instance, each face appearing with multiplicity $304$ in Theorem
\ref{T:tetrahedron} is much better understood once interpreted as the
contribution given by the $304$ curves in $\mathcal V_3$ appearing in Theorem \ref{T:triangle}.

It should also be stressed that the analysis of a semistable model 
of $S\to \Delta$
encodes information about several flat limits of the $S_t$'s in
$\P^3$, as $t\in \Delta^*$ tends to $0$
(each flat limit corresponds to an irreducible
component of the limit linear system $\fO(1)$), and
an answer to Question \ref{quest} for such a
semistable model would provide answers for all these flat limits at the
same time.
Thus, in studying Question \ref{quest} for degenerations of 
quartic surfaces to a tetrahedron, we study
simultaneously degenerations to certain rational quartic surfaces, 
e.g., to certain \emph{monoid} quartic surfaces that are
projective models of the faces of the tetrahedron,
and to sums of a \emph{self--dual} cubic surface plus a  suitable plane.
For degenerations to a Kummer, we see simultaneously degenerations to
double quadratic cones,  to sums of a smooth quadric and a double
plane (the latter corresponding to the projection of the Kummer from
one of its nodes), etc.

\medskip
Though we apply the general theory  (introduced  in \S \ref{S:gnl-rules})
to the specific case of degenerations of singular plane sections
of general quartics, it is clear that, with some more work, the same ideas can be applied to attack similar problems for
different situations,  e.g., degenerations of singular plane sections
of general surfaces of degree $d>4$, or even singular higher degree
sections of (general or not)  surfaces  of higher degree. 
For example, we obtain  Theorem \ref{T:triangle}
thinking of the curves in $\mathcal V$ as cut out by quartic
surfaces on a plane embedded in $\P^3$, and letting this plane
degenerate.  By the way, this is the first of a series of results
regarding no longer triangles, but general configurations of lines, which can be proved, we think,
by using the ideas in this paper. 
On the other hand, for general primitive $K3$ surfaces of any genus
$g\geqslant 2$,
there is a whole series of known enumerative results
\cite{YZ,beauville-counting,BL,KPMS}, yet leaving some open space for further questions,
which also can be attacked in the same way. 

\medskip
Another application of our analysis of  Question \ref{quest} is to the
irreducibility of families of singular curves on a given surface.
This was indeed Ran's main motivation in \cite{Ran1,Ran2,Ran3}, since he applied these ideas
to give an alternative proof to Harris' one
\cite{harris,harris-morrison} of the irreducibility of Severi
varieties of plane curves.
The analogous question for the family of irreducible $\delta$--nodal curves in $\vert \O_S(n)\vert$, for $S$ a general primitive 
$K3$ surface of genus $g\geqslant 3$ is widely open.

In  \cite {cd-universal} one proves that for any non negative
$\delta\leqslant g$, with
$3\leqslant g\leqslant 11$ and $g\neq 10$,
the \emph{universal Severi variety} 
 $\sV_g^{n,\delta}$,  parametrizing $\delta$--nodal members of $\vert
 \O_S(n)\vert$, with $S$ varying in the moduli space $\mod B_g$ of
 primitive $K3$  surfaces  of genus $g$ in $\P^ g$,  is irreducible
 for $n=1$.
One may conjecture that all universal Severi varieties
$\sV_g^{n,\delta}$ are irreducible (see \cite{dedieu}), and
we believe it is possible to obtain further results in this direction
using the general techniques presented in this paper.
For instance, the irreducibility of
$\sV_3^{1,\delta}$, $0< \delta\leqslant 3$, which is well known and easy to prove (see Proposition \ref{prop:irr}),
could also be deduced with the degeneration arguments developed here. 

Note the obvious surjective morphism
$p: \sV_g^{n,\delta}\to \mod B_g$. 
For $S\in \mod B_g$ general, 
one can consider   $V_g^{n,\delta}(S)$ the \emph{Severi variety} of
$\delta$--nodal  curves in $\vert \O_S(n)\vert$ 
(i.e. the fibre of $p$ over $S\in \mod B_g$), which has dimension
$g-\delta$ (see \cite{cd-universal, fkps}). 
Note that the irreducibility of $\mod V_g^{n,\delta}$ does not imply
the one of the Severi varieties $V^{n,\delta}(S)$ for a general $S\in
\mod B_g$;
by the way, this is certainly not true for $\delta=g$, since
$V^{n,g}(S)$ has dimension 0 and degree bigger than 1, see
\cite{beauville-counting, YZ}. Of course,
$V^{1,1}(S)$ is isomorphic to the \emph{dual variety} $\check S\subset
\check \P^ g$, hence it is irreducible.
Generally speaking, the smaller $\delta$ is
with respect to $g$, the easier it is to prove the irreducibility of
$V^{n,\delta }(S)$: partial results along this line can be found in
\cite{keilen} and \cite[Appendix A]{kemeny}. To the other extreme,
the curve $V^{1,g-1}(S)$ is not known to be irreducible for $S\in
\mod B_g$ general. In the simplest case $g=3$, this amounts to
proving the irreducibility of  $V^{1,2}(S)$ for a general quartic $S$ in
$\P^ 3$, which  is the nodal locus of $\check S$. This has been
commonly accepted as a known fact, but  we have not been able
to find any proof of this in the current literature.  We give one
with our methods (see Theorem \ref {thm:V2}). 

Finally, in \S \ref
{ssec:monod}, we give some information about the monodromy group of
the finite covering $\mod V^ {1,3}_3\to \mod B_3$, by showing
that it contains some  \emph {geometrically  interesting}
sugbroups. Note that a  
remarkable open question is whether the monodromy group of  $\mod
V^ {1,g}_g\to \mod B_g$ is the full symmetric group for all $g\geqslant 2$.

\medskip
The paper is organized as follows.
In \S \ref{S:gnl-rules}, we set up the machinery: we give general
definitions, introduce limit linear systems,  state our refined 
versions of Caporaso and Harris' and Ran's results, introduce limit
Severi varieties. 
In \S \ref{S:enumerate}, we state some known results for proper
reference, mostly about the degrees of the singular loci of the dual
to a projective variety.
In \S\S \ref{S:4planes} and \ref{S:Kummer-degen}, we give a
complete description of  limit Severi
varieties relative to general degenerations of quartic surfaces to 
tetrahedra and Kummer surfaces respectively; Theorems
\ref{T:tetrahedron} and \ref{T:kummer} are proved in
\S\ref{s:tetra-concl} and \S\ref{s:kummer-concl} respectively. 
In \S \ref {sec:pair of quadrics} we briefly treat other degenerations of
quartics.
Section \ref{S:Kummer-descr} contains some classical material concerning 
Kummer quartic surfaces, as well as a few 
results on the monodromy action on their nodes (probably known to the
experts but for which we could not find any proper reference): they
are required for our proof of  
Theorem \ref {thm:V2} and of the results in \S \ref {ssec:monod}.
Section \ref{S:triangle} contains the proof of a preliminary version
of Theorem \ref{T:triangle}; it is useful for \S \ref{S:4planes}, and
required for \S \ref{S:irreducibility}. Section \ref{S:irreducibility} contains Theorem \ref {thm:V2}  and the aforementioned results on the monodromy.

\subsection*{Acknowledgements}

We thank Erwan Brugall\'e  (who, together with G. Mikhalkin, 
was able to provide the enumerative numbers of
Theorems~\ref{T:tetrahedron} and \ref{T:triangle} using
tropical methods),
and Concettina Galati for numerous discussions, both enlightening and
motivating.

This project profited of various visits of the second author to the
first, which have been made possible by the research group GRIFGA, in
collaboration between CNRS and INdAM.

The first author is a member of GNSAGA of INdAM and was partly supported by the
project ``Geometry of Algebraic varieties'' funded by the Italian
MIUR. 
The second author is a member of projects CLASS and MACK of the French
ANR.

\end{opening}

\section{Conventions}\label{conv}

We will work over the field $\C$ of complex numbers.
We denote the linear equivalence on a variety $X$ by $\lineq_X$, or
simply by $\lineq$ when no confusion is likely.
Let $G$ be a group; we write $H\sg G$ when $H$ is a subgroup of $G$.

We use the classical notation for projective spaces: if $V$ is a
vector space, then $\P V$ is the space of lines in $V$,
and if $\cal E$ is a locally free sheaf on some variety $X$, we let
$\P(\mathcal{E})$ be $\mathbf{Proj}\,(\Sym\, \mathcal{E}^\vee)$.
We denote by $\check{\P}^n$ the projective space dual to $\P^n$,
and if $X$ is a closed subvariety of $\P^n$, we let $\check X$ be its
dual variety, i.e. the Zariski closure in $\check \P^n$ of the set 
of those
hyperplanes in $\P^n$ that are tangent to the smooth locus of $X$.

By a \emph{node}, we always mean an ordinary double point. 
Let $\delta \geqslant 0$ be an integer.
A \emph{nodal} (resp. \emph{$\delta$-nodal}) variety is a variety having
nodes as its only possible singularities (resp. precisely $\delta$
nodes and otherwise smooth).
Given a smooth surface $S$ together with an effective line bundle $L$
on it, 
we define the 
\emph{Severi variety} ${V}_{\delta}(S,L)$ as the Zariski closure in
the linear system $|L|$ of the locally closed subscheme
parametrizing irreducible $\delta$-nodal curves.

\medskip
We usually let $H$ be the line divisor class on $\P^2$;
when $\F_n=\P(\O_{\P^1}\oplus \O_{\P^1}(n))$ is a \emph{Hirzebruch
surface}, we let $F$ be the divisor class of its ruling over $\P^ 1$, we let  $E$ be an 
irreducible effective divisor with self-intersection $-n$ (which is unique if $n>0$),
and we let $H$ be the divisor class of $F+nE$. 

When convenient (and if there is no danger of confusion),
we will  adopt the following abuse of notation:
let $\epsilon:Y \to X$ be a birational morphism, and $C$  (resp. $D$) 
a divisor  (resp.~a  divisor class)
 on $X$; we use
the same symbol $C$  (resp. $D$) 
to denote the proper transform
$(\epsilon_*)^{-1}(C)$  (resp. the pull-back $\epsilon^*(D)$) on $Y$. 

For example, let $L$ be a line in $\P^2$, and $H$ the
 divisor class of $L$. We consider the blow-up 
$\epsilon_1:X_1\to \P^2$ at a point on $L$, and call $E_1$ the
exceptional divisor. 
The divisor class $H$ on $X_1$ is $\epsilon_1^* (H)$, and $L$ on $X_1$ is linearly equivalent to $H-E_1$.
Let then $\epsilon_2:X_2\to X_1$ be the blow-up
of $X_1$ at the point $L\cap E_1$, and $E_2$ be the exceptional
divisor. The divisor $E_1$ (resp. $L$) on $X_2$ is linearly equivalent to
$\epsilon_2^* (E_1)-E_2$ (resp. to $H-2E_1-E_2$).

\medskip
In figures depicting series of blow--ups, we indicate with a big black
dot those points that have been blown up.

\section{Limit linear systems and limit Severi varieties}
\label{S:gnl-rules}

In this section we explain the general theory upon which  this
paper relies. We build on foundational work by Caporaso and Harris
\cite{caporaso-harris,caporaso-harris2} and Ran \cite{Ran1,Ran2,Ran3},
as reinvestigated by Galati \cite{concettina1,concettina2}
(see also the detailed discussion in \cite{galati-knutsen}).

\subsection{Setting}
\label{s:setting}

In this paper we will consider flat, proper families of surfaces $f: S \to \Delta$, where $\Delta\subset \C$
is a disc centered at the origin. We will denote by $S_t$ the (schematic) fibre of $f$ over $t\in \Delta$.
We will usually consider the case in which the \emph{total space} $S$ is a smooth threefold, $f$ is smooth over $\Delta^ *=\Delta-\{0\}$, and 
$S_t$ is irreducible for $t\in \Delta^ *$. The \emph{central fibre} $S_0$ may be singular, but we will usually consider
the case in which $S_0$  is reduced and with \emph{local normal crossing} singularities. In this case the family is called \emph{semistable}.

Another family of surfaces $f': S'\to \Delta$ as above is said to be
a \emph{model of} $f: S\to \Delta$ if there is a commutative diagram
\[\xymatrix@C=25pt@R=20pt{
S' \ar[d]_{f'} \ar@{}[dr]|{\Box}
& \bar S' \ar[d] \ar[l] \ar@{-->}[r]^p
& \bar S \ar[r] \ar[d] & S \ar[d]^f \ar@{}[dl]|{\Box} \\
\Delta & \Delta \ar[l]^{t^{d'} \mapsfrom t} \ar@{=}[r]
& \Delta \ar[r]_{t\mapsto t^d} & \Delta
}\]
where the two squares marked with a '$\Box$' are Cartesian, and $p$ is
a birational map, which is an isomorphism over $\Delta^*$.
The family $f': S'\to \Delta$, if semistable, is a \emph{semistable
model of} $f: S\to \Delta$ if in addition $d'=1$
and $p$ is a morphism.
The \emph{semistable reduction theorem} of \cite {Kempf} asserts
that $f:S\to \Delta$ always has a semistable model.

\begin{example} [Families of surfaces in $\P^3$] \label{ex:pi3}
Consider a \emph{linear pencil} of degree $k$ surfaces in $\P^3$, 
generated by a \emph{general} surface $S_{\infty}$ and a \emph{special} one $S_0$.
This pencil gives rise to a flat, proper family $\varphi: \cal S\to$ $\P^1$, with $\cal S$ a hypersurface of type $(k,1)$
in $\P^3\times \P^1$, isomorphic to the blow--up of $\P^3$ along the
\emph{base locus} $S_0\cap S_\infty$ of the pencil,
 and $S_0,S_\infty$ as fibres over $0,\infty\in \P^1$, respectively.

We will usually consider the case in which $S_0$ is reduced, 
its various components may have isolated singularities, but meet
transversely along smooth curves contained in their respective smooth loci.
Thus $S_0$ has local normal crossing singularities, except for finitely many  isolated \emph{extra singularities}  belonging to one, and only one, component of $S_0$. 

We shall study the family $f: S \to \Delta$ obtained by
restricting $\cal S$ to a  disk ${\Delta} \subset \P^1$  centered at $0$, 
such that $\mathcal{S}_t$ is smooth for all $t\in {\Delta}^ *$, and
we will consider a semistable model  
of  $f: S \to \Delta$. To do so, we resolve the singularities
of $S$ which occur in the central fibre of $f$, at the points mapped by 
$\mathcal{S}_0 \to S_0 \subset \P^3$
to the intersection
points of $S_{\infty}$ with the double curves of $S_0$ (they are the
singular points of the curve $S_0\cap S_\infty$).
These are \emph{ordinary double points} of $S$, 
i.e. singularities analytically equivalent to the one at the origin of the hypersurface
$xy=zt$ in $\A^4$. 
Such a singularity is resolved by a single blow--up, which produces an
exceptional divisor  $F\cong \P^1\times\P^1$, and then it is possible to contract $F$ in the direction of either one of its rulings 
without introducing any singularity: the result is called a \emph{small resolution} of the ordinary double point.
If $S_0$ has no extra singularities, the small resolution process provides a semistable model. 
Otherwise we will have to deal with the extra singularities, which are in any case smooth points of the total space. We will do this when needed. 

Let $\tilde f: \tilde S\to \Delta$ be the semistable model thus obtained. One has $\tilde S_t\cong S_t$ for $t\in \Delta^ *$. 
If  $S_0$ has irreducible components $Q_1,\ldots,Q_r$, 
then $\tilde S_0$ consists of irreducible components $\tilde Q_1,\ldots,\tilde Q_r$
which are suitable blow--ups of $Q_1,\ldots,Q_r$, respectively. 
If $q$ is the number of ordinary double points of the original total space $S$,
we will denote by $E_1,\ldots,E_{q}$ the exceptional curves on $\tilde Q_1,\ldots,\tilde Q_r$ arising from 
the small resolution process. 
\end{example}

Going back to the general case, we will say that   $f: S \to \Delta$ is \emph{quasi--semistable} if 
$S_0$  is reduced, with \emph{local normal crossing} singularities, except for  finitely many  isolated \emph{extra singularities}  belonging to one, and only one, component of $S_0$, as in Example \ref {ex:pi3}. 

Assume then that
 $S_0$ has irreducible components $Q_1,\ldots,Q_r$, intersecting transversally along the double curves
$R_1,\ldots,R_p$, which are Cartier divisors on the corresponding components.

\begin{lemma}[Triple Point Formula, \cite {CCFM, Frie}]
\label{l:triple-point}
Assume  $f: S \to \Delta$ is quasi--semistable. Let $Q,Q'$ be irreducible components of $S_0$ intersecting along 
the \emph{double curve} $R$. Then
\begin{equation*}
\deg (N_{R\vert Q} )+\deg(N_{R\vert Q'})
+\mathrm{Card}\left\lbrace \begin{array}{c}
\text{triple points of $S_0$} \\
\text{along}\ R_s
\end{array} \right\rbrace=0,
\end{equation*}
where a \emph{triple point} is the intersection $R\cap Q''$ with a component $Q''$  of $S_0$ different from $Q,Q'$. 
\end{lemma}

\begin{remark} [See \cite {CCFM, Frie}]  \label {rem:TPF} There is a version of the  Triple Point Formula for the case in which the central fibre 
is not reduced, but its support 
has local normal crossings. Then, if the multiplicities of $Q,Q'$ are $m,m'$ respectively, 
one has
\begin{equation*}
m'\deg (N_{R\vert Q} )+m\deg(N_{R\vert Q'})
+\mathrm{Card}\left\lbrace \begin{array}{c}
\text{triple points of $S_0$} \\
\text{along}\ R_s
\end{array} \right\rbrace=0,
\end{equation*}
where each triple point $R\cap Q''$  has to be counted with
the multiplicity $m''$ of $Q''$ in $S_0$.  
\end{remark}

\subsection{Limit linear systems}
\label{s:resolving}

Let us consider a quasi--semistable family $f: S\to \Delta$ as in \S
\ref {s:setting}.
Suppose there is a fixed component free line bundle $\L$ on the total
space $S$, restricting to  a line bundle $\L_t$ on each fibre $S_t$,
$t\in  \Delta$. We assume $\L$ to be ample, with $\h^ 0(S_t,\L_t)$
constant for $t\in \Delta$. If $W$ is an effective divisor supported
on the central fibre $S_0$, we may consider the line bundle $\L(-W)$,
which is said to be obtained from $\L$ by \emph{twisting} by $W$. 
For $t\in \Delta^ *$, its restriction to $S_t$ is the same as $\L_t$,
but in general this is not the case for $S_0$;
any such a line bundle $\left. \L(-W) \right |_{S_0}$ is called a
\emph{limit line bundle} of $\L_t$ for $t\in \Delta^ *$.  
 
 \begin{remark}
\label{r:twist}
Since $\Pic(\Delta)$ is trivial,   the
divisor $S_0 \subset S$ is linearly equivalent to $0$.
So if  $W$ is a divisor supported on $S_0$, one has
$\L(-W) \cong  \L(mS_0-W)$ for all integers $m$. 
In particular if $W+W'=S_0$ then $\L(-W)\cong \L(W')$.
\end{remark}
 
Consider the subscheme ${\rm Hilb}(\L)$ of the \emph{relative Hilbert
  scheme} of curves 
of $S$ over $\Delta$, which is the Zariski closure of the set of all curves $C\in \vert \L_t\vert$, for $t\in \Delta^ *$.  
We assume  that ${\rm Hilb}(\L)$ is a component of the relative
Hilbert scheme, a condition  satisfied if $\Pic (S_t)$ has no
torsion,  
which will always be the case in our applications. 
One has a natural projection morphism $\varphi: {\rm Hilb}(\L)\to
\Delta$, which is a projective bundle over $\Delta^ *$; actually ${\rm
  Hilb}(\L)$ is isomorphic to $\fP:=\P(f_*(\L))$ over $\Delta^ *$.  
We call the fibre of $\varphi$ over $0$ the \emph{limit linear system}
of $\vert \L_t\vert$ as $t\in \Delta^*$ tends to $0$,
and we denote it by  $\fL$.  

\begin{remark}
\label{r:twist2} In general,
\emph{the limit linear system is not a linear system}.
One would be tempted to say that $\fL$ is nothing but $\vert
\L_0\vert$; this is the case if $S_0$ is irreducible, but it is 
in general no longer true when $S_0$ is reducible. In the latter case,
there may be non--zero sections of $\L_0$ whose zero--locus
contains some irreducible component of $S_0$, and accordingly points
of $\vert \L_0\vert$ which do not correspond to points in  the Hilbert
scheme of curves
(see, e.g., Example \ref{ex:gh} below).
\end{remark} 

In any event,  ${\rm Hilb}(\L)$ is a birational modification of
$\fP$, and $\fL$ is a suitable degeneration of the projective
space  $\vert \L_t\vert$, $t\in \Delta^ *$.  
One has:
\begin{lemma}
\label{l:lim-lin}
Let $\fP'\to \Delta$ be a flat and proper morphism, 
isomorphic to $\P(f_*(\L))$ over $\Delta^*$,
and such that $\fP'$ is a Zariski closed subset of the
relative Hilbert scheme of curves of $S$ over $\Delta$. 
Then $\fP'=\Hilb(\L)$.
\end{lemma}

\begin{proof}
The two Zariski closed subsets $\fP'$ and $\Hilb(\L)$
are irreducible, and coincide over $\Delta^*$.
\end{proof}

In passing from $\P(f_*(\L))$ to ${\rm Hilb}(\L)$,
one has to perform a series of blow--ups along
smooth centres contained in the central fibre, which correspond to
spaces of non--trivial sections of some (twisted)  line  bundles which
vanish on divisors contained in the central fibre. 
The exceptional divisors
one gets in this way give rise to components of $\fL$, and may be
identified with birational modifications of sublinear systems of
twisted linear systems restricted to $S_0$, 
as follows from Lemma \ref{l:tg-id} below.
We will see examples of
this later (the first one in Example \ref{ex:gh}).

\begin{lemma}
\label{l:tg-id}
\begin{inparaenum}[\normalfont (i)]
\item \label{pt:gnl-case}
Let $X$ be a connected variety, $\L$ a line bundle on
$X$, and $\sigma$ a non zero global section of $\L$ defining a
subscheme $Z$ of $X$.
Then the projectivized tangent space to $\P\H^0(X,\L)$ at $\left<
  \sigma \right>$ canonically identifies with
   the \emph{restricted linear system}  
\[
\P\hspace{.05cm}
\mathrm {Im}\big(\H^ 0(X,\L)\to \H^ 0(Z,\left.\L\right|_{Z})\big),
\]
also called the \emph{trace} of $\vert \L\vert$ on $Z$
(which in general is not the  complete linear system 
$\vert \L\otimes \O_Z \vert$).

\item \label{subspace}
More generally, let $\mathfrak l$ be a linear subspace of $\P\H^0(X,\L)$
with fixed locus scheme $F$ defined by the  system of equations
$\{\sigma=0\}_{\left<\sigma\right>\in \mathfrak{l}}$.
Then the projectivized normal bundle of $\mathfrak l$ in $\P\H^0(X,\L)$
canonically identifies with 
\[
\mathfrak{l}\times 
\P\hspace{.05cm}
\mathrm {Im}\big(\H^ 0(X,\L)\to \H^ 0(F,\left.\L\right|_{F})\big).
\]
\end{inparaenum}
\end{lemma}

\begin{proof}
Assertion \eqref{pt:gnl-case} comes from  the identification of
the tangent space of $\P\H^0(X,\L)$ at 
$\left< \sigma \right>$ with the 
cokernel of the injection
$
\H^0(X,\O_X)\to \H^0(X,\L)
$,
given by the multiplication by $\sigma$.
As for \eqref{subspace},  note that the normal bundle of
$\mathfrak l$ in $\P\H^0(X,\L)$ splits as a direct sum of copies of $\O_\mathfrak l(1)$, hence
the associated projective bundle is trivial. Then the proof is similar
to that of \eqref{pt:gnl-case}.
\end{proof}

\begin{example} [See \cite{gh-degen}] 
\label {ex:gh}
Consider a family of degree $k$ surfaces $f: S\to \Delta$ arising,
as in Example \ref{ex:pi3}, from a pencil
  generated by a general surface $S_\infty$ and by $S_0=F\cup P$,
  where $P$ is a plane and $F$ a general surface of degree $k-1$. 
One has a semistable model 
$\tilde f: \tilde S\to \Delta$ of this family,
as described in  Example \ref {ex:pi3},
with $\tilde S_0=F\cup \tilde P$, where $\tilde P\to P$ is the
blow--up of $P$ at the $k(k-1)$ intersection points of $S_\infty$ with
the smooth degree $k-1$ plane curve $R:=F\cap P$ (with exceptional
divisors $E_i$, for $1\leqslant i\leqslant k(k-1)$).  

We let $\L:=\O_{\tilde S}(1)$ be the pull--back by $\tilde S\to S$ of
$\O_S(1)$, obtained by pulling back $\O_{\P^3}(1)$ via the map
$S\to \P^3$.
The component $\Hilb(\L)$ of the Hilbert scheme is gotten from the
projective bundle $\P(f_*(\O_{\tilde S}(1)))$, by blowing up  
the point of the central fibre $\vert \O_{S_0}(1)\vert $ corresponding to the 1--dimensional space of non--zero sections vanishing on the plane $P$.  
The limit linear system $\fL$ is  the union of $\fL_1$, the blown--up
$\vert \O_{S_0}(1)\vert $, and of the exceptional divisor $\fL_2\cong
\P^ 3$, 
identified as the twisted linear system $\vert \O_{S_0}(1)\otimes
\O_{S_0}(-P)\vert$. The corresponding twisted line bundle restricts to
the trivial linear system on $F$, and to $\vert \O_{\tilde
  P}(k)\otimes \O_{\tilde P}(-\sum_{i=1}^ {k(k-1)}E_i)\vert $ on
$\tilde P$.

The components $\fL_1$ and $\fL_2$ of $\fL$  meet along the
exceptional divisor $\mathfrak{E}\cong \P^ 2$ of the morphism 
$\fL_1\to |\O_{S_0}(1)|$. 
Lemma \ref {l:tg-id} shows that the elements of 
$\mathfrak{E}\subset \fL_1$ 
identify as the points of $\vert \O_R(1)\vert\cong \vert \O_P(1)\vert$, 
whereas the plane $\mathfrak{E}\subset \fL_2$ is the set of elements  
$\Gamma\in \vert \O_{\tilde P}(k)\otimes \O_{\tilde P}(-\sum_{i=1}^
{k(k-1)}E_i)\vert$ containing the proper transform $\hat R \cong R$ of
$R$ on $\tilde P$. The corresponding element of $\vert \O_R(1)\vert$
is cut out on $\hat R$ by the further component of $\Gamma$, which
is the pull--back to $\tilde P$ of a line in $P$. \end{example}

\subsection{Severi varieties and their limits}
\label{s:Severi_1}

Let $f: S\to \Delta$ be a semistable family as in \S \ref {s:setting},
and $\L$ be a line bundle on $S$ as in \S \ref {s:resolving}.  
We fix  a non--negative integer $\delta$, and consider 
the locally closed subset  $\mathring{V}_{\delta}(S,\L)$ of  $\Hi(\L)$
formed by all curves $D\in \vert \L_t\vert$, for $t\in \Delta^ *$,
such that $D$ is  irreducible, nodal, and has exactly $\delta$ nodes.
We  define ${V}_{\delta}(S,\L)$ (resp. $\cru{V}_{\delta}(S,\L)$) as
the Zariski closure of 
$\mathring{V}_{\delta}(S,\L)$ in $\Hi(\L)$ (resp. in $\P(f_*(\mathcal
L))$). This is the \emph{relative Severi variety} (resp. the
\emph{crude relative Severi variety}).
We may write  $\mathring{V}_{\delta}$,
 ${V}_{\delta}$,  and $\cru V_\delta$, rather than
 $\mathring{V}_{\delta}(S,\L)$,  ${V}_{\delta}(S,\L)$,
 and $\cru{V}_{\delta}(S,\L)$, respectively.
  
We have a natural map $f_{\delta}: {V}_{\delta}\to \Delta$. If $t\in
\Delta^* $, the fibre $V_{\delta,t}$
of $f_{\delta}$ over $t$ is the  \emph{Severi variety}
${V}_{\delta}(S_t,\L_t)$ of $\delta$--nodal curves in the linear
system $\vert \L_t\vert$ on $S_t$, whose degree, independent on $t\in
\Delta^ *$,  we denote by $d_\delta(\L)$ (or simply by $d_\delta$).  
We let ${\fV}_{\delta}(S,\L)$ (or simply $\fV_\delta$) be
the central fibre of $f_{\delta}: {V}_{\delta}\to \Delta$;
it is the \emph{limit Severi variety} of ${V}_{\delta}(S_t,\L_t)$ as
$t\in \Delta^*$ tends to $0$.  This is a subscheme of  the
limit linear system $\fL$, which, as we said, has been studied by
various authors.
In particular, one can describe in a number of situations its
various irreducible components, with their multiplicities (see \S \ref
{s:ZRR} below). 
This is what we will do for several families of quartic
surfaces in $\P^ 3$. 

In a similar way, one defines the \emph{crude limit Severi variety}
$\cru{\fV}_{\delta}(S,\L)$ (or $\cru\fV_\delta$),
sitting in $\vert \L_0 \vert$. 

\begin{remark}\label{rem:exp} For $t\in \Delta^ *$, the \emph{expected dimension} of the Severi variety ${V}_{\delta}(S_t,\L_t)$ is
$\dim (\vert L_t\vert)-\delta$. We will always assume that the dimension of (all components of) ${V}_{\delta}(S_t,\L_t)$  equals the expected one for all 
$t\in \Delta^ *$. This is a strong assumption, which will be satisfied
in all our applications. 
\end{remark}

\begin{notation}
Let  $f: S\to \Delta$ be a family of degree $k$ surfaces in $\P^ 3$ as
in Example \ref {ex:pi3}, and let $\tilde f: \tilde S\to \Delta$ be a
semistable model of $f: S\to \Delta$.
We consider the line bundle $\O_S(1)$, defined as the pull--back of
$\O_{\P^3}(1)$ via the natural map $S\to \P^3$, and let 
$\O_{\tilde S}(1)$ be its pull--back  on $\tilde S$.
We denote by $\fV_{n,\delta}(\tilde S)$
(resp. $\fV_{n,\delta}(S)$), or simply $\fV_{n,\delta}$,
the limit Severi variety $\fV_{\delta}(\tilde S,\O_{\tilde S}(n))$
(resp. $\fV_{\delta}(S,\O_{S}(n))$).
Similar notation $\cru \fV_{n,\delta}(\tilde S)$
(resp. $\cru \fV_{n,\delta}(S)$), or $\cru \fV_{n,\delta}$, will be used for
the crude limit. 
\end{notation}

\subsection{Description of the limit Severi variety}
\label{s:ZRR}

Let again $f: S\to \Delta$ be a semistable family as in \S \ref
{s:setting}, and $\L$ a line bundle on $S$ as in \S \ref {s:resolving}.
The local machinery developed in \cite{concettina1,concettina2,
galati-knutsen} enables us to identify
the components of the limit Severi variety, with their multiplicities. As usual, we will suppose that  $S_0$ has irreducible components $Q_1,\ldots,Q_r$, intersecting transversally along the double curves
$R_1,\ldots,R_p$.  We will also assume that there are $q$ exceptional curves $E_1,\ldots, E_q$ on $S_0$, arising from a small resolution of an original family with singular total space, as discussed in \S \ref {s:setting}. 

\begin{notation}
Let $\seqN$ be the set of sequences $\utau=(\tau_m)_{m\geqslant 2}$ of
non--negative integers 
 with only finitely many non--vanishing terms. We
define two maps $\nu,\ \mu: \seqN \to \N$ as follows:
\begin{equation*}
\nu\left( \utau \right) =
\sum\nolimits_{m \geqslant 2} \tau_m \cdot (m-1),
\quad \text{and} \quad
\mu\left( \utau \right) =
\prod\nolimits_{m \geqslant 2} m^{\tau_m}.
\end{equation*}
Given a $p$-tuple
$\ubtau=(\utau_1,\ldots,\utau_p) \in  \seqN^p$,
we set 
\begin{equation*}
\nu(\ubtau) = \nu(\utau_1)+\cdots + \nu(\utau_p),
\quad \text{and} \quad
\mu(\ubtau) = \mu(\utau_1) \cdots \mu(\utau_p),
\end{equation*}
thus defining two maps $\nu,\ \mu: \seqN^p \to \N$.
Given $\bdelta=(\delta_1,\ldots,\delta_r) \in \N^r$, we set
\begin{equation*}
|\bdelta| := \delta_1 + \cdots + \delta_r.
\end{equation*}
Given a subset $I \subset \{1,\ldots,q\}$, $|I|$
will denote its cardinality.
\end{notation}

\begin{definition}
\label{d:zrr}
Consider a divisor
$W$ on $S$, supported on the central fibre $S_0$, i.e. a linear
combination of $Q_1,\ldots,Q_r$. Fix $\bdelta \in \N^r$, $\ubtau\in
\seqN^p$, and $I\subseteq \{1,\ldots,r\}$.
We let $\mathring{V}(W,\bdelta,I,\ubtau)$
be the Zariski locally closed subset in
$|\L(-W)\otimes \O_{S_0}|$ parametrizing curves $D$ such that:\\
\begin{inparaenum}[\normalfont (i)]
\item
\label{no-double-curve}
  $D$
neither contains any curve $R_l$, with $l\in \{1,\ldots, p\}$, 
nor passes through any triple point of $S_0$;\\
\item $D$ contains the exceptional divisor $E_i$, with multiplicity 1,
  if and only if $i\in I$, and has a node on it;\\
\item $D-\sum_{i\in I} E_i$ has  $\delta_s$ nodes on $Q_s$, for $s\in \{1,\ldots, r\}$, off the singular locus of $S_0$, and is 
otherwise smooth;\\
\item 
\label{c:tacnodes}
for every $l\in \{1,\ldots, p\}$ and $m \geqslant 2$, there are exactly
$\tau_{l,m}$ points on $R_l$, off the intersections with $\sum_{i\in I} E_i$, 
at which $D$ has an \emph{$m$-tacnode} (see below for the definition),
with reduced tangent cone equal to the tangent line of $R_l$ there.
\end{inparaenum}

We let ${V}(W,\bdelta,I, \ubtau)$ be the Zariski closure of
$\mathring{V}(W,\bdelta,I, \ubtau)$ in
$|\L(-W)\otimes \O_{X_0}|$.
\end{definition}

Recall that an \emph{$m$-tacnode} is an $A_{2m-1}$-double point,
i.e. a plane
curve singularity locally analytically isomorphic 
to the hypersurface of $\C^2$ defined by the equation 
$y^2=x^{2m}$ at the origin.
Condition \eqref{c:tacnodes} above requires that $D$ is a divisor having $\tau_{l,m}$ 
$m$--th order tangency points with the curve $R_l$, at 
points of $R_l$ which are not triple points of $S_0$.

\begin{notation}\label{not:1}
In practice, we shall not use the notation ${V}(W,\bdelta,I, \ubtau)$,
but rather a more expressive one like, e.g., 
$V(W,\delta_{Q_1}=2,E_1,\tau_{R_1,2}=1)$ for the variety
parametrizing curves in $|\L(-W)\otimes \O_{S_0}|$, with two nodes on $Q_1$, one
simple tacnode along $R_1$, and containing the exceptional curve
$E_1$.
\end{notation}

\begin{proposition}
[\cite{concettina1,concettina2,galati-knutsen}]
\label{p:zrr}
Let $W,\bdelta, I, \ubtau$ be as above,
and set $|\bdelta|+\vert I\vert+\nu(\ubtau)=\delta$.
Let $V$ be an irreducible component of $V(W,\bdelta,I,\ubtau)$.
If \\
\begin{inparaenum}[(i)]
\item
\label{c:hilb}
the linear system $|\L(-W)\otimes \O_{X_0}|$
has the same dimension
as $|\L_t|$ for $t\in \Delta^*$, and\\
\item
\label{c:versal}
$V$ has (the expected) codimension $\delta$ in 
$|\L(-W)\otimes \O_{X_0}|$,
\end{inparaenum}\\
then $V$ is an irreducible component of multiplicity $\mu(V):=\mu(\ubtau)$
of the limit Severi variety $\mathfrak V_\delta(S,\L)$.
\end{proposition}

\begin{remark}
\label{r:smoothness}
Same assumptions as in Proposition \ref{p:zrr}. If  there is
at most one tac\-node (i.e. all $\tau_{l,m}$ but 
possibly one vanish, and this is equal to $1$), 
the relative Severi variety $V_\delta$ is smooth at the general
point of $V$ (see \cite {concettina1, concettina2,
galati-knutsen}),
and thus $V$ belongs to only one
 irreducible component of $V_\delta$.
There are other cases in which such a smoothness property holds (see
\cite{caporaso-harris}).

If $V_\delta$  is smooth at the
general point  $D\in V$, the multiplicity of $V$ in the limit Severi
variety $\mathfrak V_\delta$  is the minimal integer $m$
such that there are local analytic \emph{$m$--multisections} of
$V_\delta \to \Delta$, i.e. analytic smooth curves in 
$V_\delta$, passing through $D$ and intersecting the general
fibre $V_{\delta, t}$, $t\in \Delta^ *$, at $m$ distinct points.
\end{remark}

Proposition \ref{p:zrr}  still does not provide a complete picture of 
the limit Severi variety.
For instance, curves passing
through a  triple point of $S_0$ could  play a role in this limit.
It would be desirable to know that one can always  obtain a semistable model of the original family,
where  \emph{every} irreducible component of the limit Severi variety
is realized 
as a family of curves of the kind stated in Definition \ref{d:zrr}. 

\begin{definition}
\label{def:reg-comps}
Let $f: S\to \Delta$ be a semistable family as in \S \ref {s:setting},
$\L$ a line bundle on $S$ as in \S \ref {s:resolving},
and $\delta$ a positive integer.
The \emph {regular part of the limit Severi variety}
$\fV_\delta(S,\L)$ is the cycle in the limit linear system $\fL
\subset \Hilb(\L)$
\begin{equation}
\label{reg-Severi}
\reg \fV_\delta(S,\L):=
\sum_W \sum_{|\bdelta|+|I|+\nu(\ubtau)=\delta} \mu(\ubtau) \cdot
\Biggl( \sum_{V \in \mathrm{Irr}^\delta\left( V(W,\bdelta,I,\ubtau)
    \right)}  V \Biggr)
\end{equation}
(sometimes simply denoted by $\reg \fV_\delta$), where:\\
\begin{inparaenum}[(i)]
\item \label{def:2} $W$ varies among all effective divisors on $S$
supported on the central fibre $S_0$, such that
$\h^0(\L_0(-W))=\h^0(\L_t)$ for $t\in \Delta^*$;\\
\item ${\rm Irr}^\delta(Z)$ denotes the set of all codimension
  $\delta$ irreducible components of a scheme $Z$. 
\end{inparaenum}
\end{definition}

Proposition \ref{p:zrr} asserts that the cycle
$Z(\fV_\delta)-\reg \fV_\delta$ is effective, with support disjoint
in codimension $1$ from that of $\reg \fV_\delta$
(here, $Z(\fV_\delta)$ is the cycle associated to $\fV_\delta$).
We call the irreducible components of the support of $\reg \fV_\delta$
the \emph {regular components} of the limit Severi variety.

Let  $\tilde f: {\tilde S} \to \Delta$ be a semistable model of
$f: S\to \Delta$, and $\tilde \L$ the pull--back on $\tilde S$ of
$\L$.
There is a natural map $\Hilb(\tilde\L)\to \Hilb(\L)$, which induces a
morphism $\phi: \tilde\fL\to |\L_0|$.

\begin{definition}
\label{def:good-model}
The semistable model $\tilde f: {\tilde S} \to \Delta$ is a 
$\delta$--\emph{good model} of $f: S\to \Delta$
(or simply \emph{good model}, if it is clear which $\delta$ we are
referring at), if the following equality of cycles holds
\[
\phi_* \bigl( \reg\fV_\delta(\tilde S, \tilde \L) \bigr)=
\cru\fV_\delta(S,\L).
\]
\end{definition}

Note that the cycle $\cru\fV_\delta(S,\L)-
\phi_* \bigl( \reg\fV_\delta(\tilde S, \tilde \L) \bigr)$
is effective.
The family $f:S\to \Delta$ is said to be $\delta$--\emph{well behaved}
(or simply \emph{well behaved}) if it has a $\delta$-good model.
A semistable model $\tilde f: \tilde S\to \Delta$
of $f: S\to \Delta$ as above
is said to be  $\delta$--\emph{absolutely good} if $\fV_\delta(\tilde S, \tilde \L)=
\reg \fV_\delta(\tilde S, \tilde \L)$ as cycles in the relative
Hilbert scheme.
It is then a $\delta$--good model both of itself, and
of $f: S\to \Delta$.

Theorems \ref{T:tetrahedron} and \ref{T:kummer} will be proved by
showing that the corresponding families of quartic surfaces are well
behaved.

\begin{remark}
\label{r:other-lim}
Suppose that $f: S\to \Delta$ is $\delta$--well behaved, with
$\delta$--good model $\tilde f:{\tilde S} \to \Delta$. 
It is possible that some components in
$\reg\fV_\delta(\tilde S,\tilde \L)$ are
contracted by ${\rm Hilb}(\tilde \L)\to |\L_0|$ to varieties of
smaller dimension, and therefore that their push--forwards are zero.
Hence these components of $\fV_\delta(\tilde S)$ are 
\emph{not visible} in $\cru\fV_\delta(S)$.
They are however usually visible in the crude limit Severi variety of
another model $f': S'\to \Delta$, obtained from $\tilde S$ via an
appropriate twist of $\L$.
The central fibre $S'_0$ is then a flat limit of $S_t$, as
$t\in\Delta^*$ tends to $0$, different from $S_0$.
\end{remark}

\begin{conjecture}\label{conj:goodness}
Let $f:S\to \Delta$ be a semistable family of surfaces, endowed with a
line bundle $\L$ as above, and $\delta$ a positive integer. Then:\\
{\bf (Weak version)} Under suitable assumptions (to be discovered), $f: S\to \Delta$ is
$\delta$--well behaved.\\
{\bf (Strong version)} Under suitable assumptions (to be discovered), $f: S\to \Delta$
has a $\delta$--absolutely good semistable model.
\end{conjecture}

The local computations in \cite{concettina2} provide a criterion for
absolute goodness:

\begin{proposition}\label{prop:criterion}
Assume there is a semistable model $\tilde f:{\tilde S} \to \Delta$ of
$f:S \to \Delta$, 
with a limit linear system $\tilde \fL$ free in codimension $\delta+1$
of curves of the following types: \\
\begin{inparaenum}[(i)]
\item
\label{cat:first} curves containing 
double curves  of $\tilde S_0$;\\
\item \label {cat:trip} curves passing through a triple point of
  $\tilde S_0$;\\
 \item
\label{cat:last}
non--reduced  curves.\\
\end{inparaenum}
If in addition, for $W,\bdelta,I,\ubtau$ as in Definition \ref{d:zrr},
every irreducible component of ${V}(W,\bdelta,I, \ubtau)$ has the
expected codimension in $|\L_0(-W)|$,
then $\tilde f:{\tilde S} \to \Delta$ is $\delta$--absolutely good,
which implies that $f:S \to \Delta$ is $\delta$--well behaved.
\end{proposition}

Unfortunately, in the cases we shall consider
conditions \eqref{cat:first}--\eqref{cat:last} in Proposition \ref {prop:criterion} are violated
(see Propositions \ref{prop:limlin-tetra} and
\ref{prop:limlin-kummer}), which indicates that further investigation
is needed to prove the above conjectures.
The components of the various ${V}(W,\bdelta,I, \ubtau)$ have
nevertheless the expected codimension, and we are able to prove that
our examples are well--behaved, using additional enumerative
information.

Absolute goodness seems to be a property hard to prove,
except when the dimension of the Severi varieties under
consideration is $0$, equal to the expected one
(and even in this case, we will need extra enumerative information for the proof).
We note in particular that the $\delta$--absolute goodness of
$\tilde f:{\tilde S} \to \Delta$ implies that it is a $\delta$--good
model of every model $f': S'\to \Delta$, obtained from $\tilde S$ via a
twist of $\tilde \L$ corresponding to an irreducible component of the
limit linear system $\tilde\fL$.

\subsection{An enumerative application}
\label{s:enumeration}

Among the applications of the theory described above,
there are the ones to enumerative problems, in particular to  the
computation of the degree $d_\delta$ of 
Severi varieties $V_\delta(S_t,\L_t)$, for the general member $S_t$
of a family $f: S\to \Delta$ as in \S \ref {s:setting},
with $\L$ a line bundle on $S$ as in \S \ref {s:resolving}.

Let $t\in \Delta^ *$ be  general,
and let $m_\delta$ be the dimension of $V_\delta(S_t,\L_t)$, which we
assume to be $m_\delta=\dim (\vert \L_t\vert)-\delta$.
Then $d_\delta$ is the number of points in common of
$V_\delta(S_t,\L_t)$ with $m_\delta$
\emph{sufficiently general}  hyperplanes of $\vert \L_t\vert$. Given
$x\in S_t$,
\[
{H}_{x} := \left\{ [D] \in |\L_t|\ \text{s.t.}\ x\in D \right\}
\]
is a plane in $|\L_t|$.  It is well known, and easy to check
(we leave this to the reader), that if $x_1,\ldots, x_{m_\delta}$ are \emph{general points}
of $S_t$, then $H_{x_1},\ldots, H_{x_{m_\delta}}$ are \emph{sufficiently general}  planes of $\vert \L_t\vert$
with respect to $V_\delta(S_t,\L_t)$.  Thus \emph{$d_\delta$ is the number of $\delta$--nodal curves in $\vert \L_t\vert$
passing through $m_\delta$  general points of $S_t$.}

\begin{definition}\label{def:enne} 
In the above setting, let $V$ be an irreducible component of the limit
Severi variety $\fV_\delta(S,\L)$, endowed with its reduced
structure. 
We let $Q_1,\ldots,Q_r$ be the irreducible components of $S_0$,
and ${\bf n}=(n_1,\ldots,n_r) \in \N^r$ be such that 
$|{\bf n}|:=n_1+\cdots+n_r=m_{\delta}$.
Fix a collection $Z$ of $n_1,\ldots,n_r$ general
points on $Q_1,\ldots,Q_r$ respectively.
The \emph{$\bf n$--degree} of $V$ is the number
$\deg_{\bf n}(V)$
of points in $V$ corresponding
to curves passing through the points in $Z$.  
\end{definition}

Note that in case $m_{\delta}=0$, the above definition is somehow
pointless: in this case, $\deg_{\bf n}(V)$ is simply the number of
points in $V$.
By contrast, when $V$ has positive dimension, 
it is  possible that $\deg_{\bf n} (V)$ be zero for various $\bf n$'s. 
This is  related to the phenomenon described in Remark
\ref{r:other-lim} above.  We will see examples of this  below. 

By flatness, the following result is clear:
\begin{proposition}
\label{prop:deg}
Let $\tilde f:\tilde S \to \Delta$ be a semistable model, and
name $P_1,\ldots,P_{\tilde r}$ the irreducible components of $\tilde
S_0$, in such a way that $P_1,\ldots,P_r$ are the proper transforms of
$Q_1,\ldots,Q_r$ respectively.\\
\begin{inparaenum}[(i)]
\item
For every ${\bf \tilde n}=(n_1,\ldots,n_{ r},0,\ldots,0) \in
\N^{\tilde r}$ such that ${|{\bf \tilde n}|=m_{\delta}}$, 
one has
\begin{equation}
\label{degree}
d_\delta \geqslant
\sum_{V \in \mathrm{Irr}(\reg\fV_\delta(\tilde S,\tilde \L))} 
\mu(V)\cdot \deg_{\bf \tilde n} (V) 
\end{equation}
(recall the definition of $\mu(V)$ in Proposition \ref{p:zrr}).\\
\item If equality holds in \eqref{degree} for every $\bf \tilde n$ as above, then 
 $\tilde f:\tilde S \to \Delta$ is a $\delta$--good model
of $f: S\to \Delta$ endowed with $\L$.
\end{inparaenum}
\end{proposition}

\section{Auxiliary results}
\label{S:enumerate}

In this section we collect a few results which we will use later.

First of all, for a general surface $S$ of degree $k$ in $\P^ 3$,
we know from classical projective geometry the degrees $d_{\delta,k}$
of the Severi varieties $V_\delta(S,\O_S(1))$,
for $1\leqslant \delta \leqslant 3$.
For $K3$ surfaces, this fits in a more general
framework of known  numbers
(see \cite{beauville-counting,BL,KPMS, YZ}). One has:

\begin{proposition}[\cite {salmon, vainsencher}] 
\label{p:deg-dual}
Let $S$ be a general degree $k$ hypersurface in $\P^3$. Then
\begin{align*}
d_{1,k}=& k(k-1)^2,\\
d_{2,k}=& \frac{1}{2}k(k-1)(k-2)(k^3-k^2+k-12),\\
d_{3,k}=& \frac{1}{6}k(k-2)(k^7-4k^6+7k^5-45k^4+114k^3-111k^2+548k-960).
\end{align*}
For $k=4$, these numbers are $36$, $480$, $3200$ respectively.
\end{proposition}

Note that $V_1(S,\O_S(1))$ identifies  with the \emph{dual surface}
$\check{S} \subset \check{\P}^3$. The following is an extension of the computation of $d_{1,k}$ 
for  surfaces with certain singularities.
This is well--known and the details can be left to the reader. 

\begin{proposition}
\label{p:deg-dual-sing}
Let $S$ be a degree $k$ hypersurface in $\P^3$, having
$\nu$ and $\kappa$ double points of type $A_1$ and $A_2$ respectively
as its only singularities.
Then 
\begin{equation*}
\deg (\check{S})=
k(k-1)^2-2\nu-3\kappa.
\end{equation*}
\end{proposition}

The following topological formula is well-known (see, e.g.,
\cite[Lemme VI.4]{beauville}).

\begin{lemma}
\label{l:pencil}
Let $p:S \to B$ be a surjective morphism of a smooth projective
surface onto a smooth curve.
One has
\[
\chitop(S)=\chitop(F_{\rm gen})\chitop(B)+
\sum_{b\in {\rm Disc}(p)}\bigl(\chitop(F_b)-\chitop(F_{\rm gen})\bigr),
\]
where $F_{\rm gen}$ and $F_b$ respectively denote the fibres of $p$
over the generic  point of $B$ and a closed point $b\in B$, and
${\rm Disc}(p)$ is the set of points above which $p$ is not smooth.
\end{lemma}

As a side remark, note that it is possible to give a proof of the
 Proposition \ref {p:deg-dual-sing} based on Lemma \ref {l:pencil}.
 This can be left to the reader.

Propositions \ref{p:deg-dual} and \ref{p:deg-dual-sing} are sort of
Pl\"ucker formulae for surfaces in $\P^3$. The next proposition  
provides analogous formulae for curves in a
projective space of any dimension.

\begin{proposition}
\label{p:dJ}
Let $C \subset \P^N$ be an irreducible, non--degenerate curve of degree $d$ and of genus
$g$, the normalization morphism of which is unramified.
Let $\tau \leqslant N$ be a non-negative integer, and assume $2\tau <
d$.
Then the Zariski closure of the locally closed subset of $\check{\P}^N$ parametrizing 
\emph{$\tau$--tangent} hyperplanes  to $C$ (i.e. planes tangent to $C$ at $\tau$ distinct points) has degree equal to the coefficient of
$u^{\tau}v^{d-2\tau}$ in 
\[
(1+4u+v)^g (1+2u+v)^{d-\tau-g}.
\]
\end{proposition}

\begin{proof}
Let $\nu:\bar{C} \to C$ be the normalization of $C$, and let $\mathfrak g$
be the $g^N_\mu$ on $\bar C$ defined as the pull--back on $\bar C$ of the 
hyperplane linear series on $C$. Since $\nu$ is unramified, the degree of
the subvariety of $\check{\P}^N$ parametrizing 
$\tau$-tangent hyperplanes to $C$ is equal to the number of divisors having $\tau$
double points in a general sublinear series $g^\tau_\mu$ of $\mathfrak g$. This number is 
computed by a particular instance of de Jonqui\`eres'
formula, see \cite[p. 359]{acgh}.
\end{proof}

The last result we shall need is:

\begin{lemma}
\label{l:tangency}
Consider a smooth, irreducible curve $R$, contained in a smooth
surface $S$ in $\P^3$.
Let $\check{R}_S$ be the irreducible curve in $\check \P^3$
parametrizing planes tangent to $S$ along $R$.
Then the dual varieties $\check{S}$ and $\check{R}$ both contain
$\check{R}_S$, and do not intersect transversely at its general point. 
\end{lemma}

\begin{proof}
Clearly  $\check{R}_S$ is contained in  $\check{S}\cap \check{R}$.
If either  $\check{S}$ or $\check{R}$ are singular at the general
point of $\check{R}_S$, there is nothing to prove. Assume that
$\check{S}$ and $\check{R}$ are both smooth at the general point of
$\check{R}_S$. 
We have to show that they are tangent there. Let $x\in R$ be
general. Let $H$ be the tangent plane to $S$ at $x$. Then $H\in
\check{R}_S$ is the general point. Now, the biduality theorem (see ,
e.g., \cite [Example 16.20]{harris2}) says that the tangent plane to
$\check{S}$ and of $\check{R}$ at $H$ both coincide with the set of
planes in $\P^ 3$ containing $x$, hence the assertion.
\end{proof}

\section{Degeneration to a tetrahedron}
\label{S:4planes}

We consider a family $f:S \to \Delta$ of surfaces in $\P^3$,
induced (as in Example \ref {ex:pi3} and in \S \ref {s:resolving})
by a pencil generated by a general quartic surface $S_{\infty}$ and a
\emph{tetrahedron} $S_0$
(i.e. $S_0$ is the union of four independent planes,
called the \emph{faces} of the tetrahedron), together with the
pull-back $\O_S(1)$ of $\O_{\P^3}(1)$.
We will prove that it is $\delta$--well behaved for
$1\leqslant \delta\leqslant 3$ by constructing a suitable good
model.

The plan is as follows.
We construct the good model in \S\ref{ssec:model}, and complete its
description in \S\ref{s:identifications}.
We then construct the corresponding limit linear system: the core of
this is \S\ref{s:4planes-limlin}; the paragraphs \ref{s:planes},
\ref{s:rat4ics}, and \ref{s:cubics}, are devoted to the study of the
geometry of the exceptional components of the limit linear system
(alternatively, of the geometry of the corresponding flat limits of
the smooth quartic surfaces $S_t$, $t\in\Delta^*$);
eventually, we complete the description in \S\ref{s:good}.
We then identify the limit Severi varieties in \S\ref{s:tetra-concl}.

\subsection{A good model}\label{ssec:model}

The outline of the construction is as follows:\\
{\it
 \begin{inparaenum}[(I)]
\item \label{Tetra:sm-resol}
we first make a small resolution of the singularities of $S$ as in
Example \ref {ex:pi3};\\ 
\item \label{Tetra:bs-change}
then we  perform a degree $6$ base change;\\
\item \label{Tetra:granchio-resol}
next we resolve the singularities  of the total space arisen with the
base change, thus obtaining a new semistable family
$\pi:X\to\Delta$;\\ 
\item \label{Tetra:flop}
finally  we will \emph{flop} certain double curves in the central fibre
$X_0$, thus obtaining a new semistable family $\varpi:\bar{X}\to
\Delta$.
\end{inparaenum}}\\
The central fibre of the intermediate family $\pi:{X}\to \Delta$ is
pictured in Figure \ref{f:granchio} (p.\ \pageref{f:granchio};
we provide a cylindrical projection of a 
real picture of $X_0$, the dual graph of which is topologically an
$\mathbf{S}^2$ sphere),
and the flops are described in Figure
\ref{f:flop} (p.\ \pageref{f:flop}). 
The reason why we need to make the degree 6 base change is, intuitively, the following: 
a degree $3$ base change is needed to understand the contribution to
the limit Severi variety of curves passing through a \emph{vertex} 
(i.e.~a triple point) of  the tetrahedron,
while an additional degree $2$ base change enables one to understand
the contributions due to the \emph{edges} (i.e.~the double lines) of
the tetrahedron. 

\subsubsection{Steps (\ref{Tetra:sm-resol}) and (\ref{Tetra:bs-change})}
The singularities of the initial total space $S$ consist of four
ordinary double points on each edge of $S_0$. We consider 
(cf. Example \ref {ex:pi3})
the small resolution $\tilde{S} \to S$ obtained by arranging for every
edge the four $(-1)$--curves two by two on the two adjacent faces.
We call $\tilde{f}:\tilde{S} \to \Delta$ the new family.

Let $p_1,\ldots,p_4$ be the triple points of $\tilde{S}_0$. For each
$i\in \{1,\ldots,4\}$, we let $P_i$ be the irreducible component of
$\tilde{S}_0$ 
which is opposite to the vertex $p_i$: it is a plane blown-up at six
points. 
For distinct $i,j\in \{1,\ldots,4\}$, we let $\p{E}_{ij}$ and
$\m{E}_{ij}$ be the two $(-1)$-curves contained in $P_i$ and meeting
$P_j$.  
We call $\p{z}_{ij}$ and $\m{z}_{ij}$ the two points cut out on $P_i$
by $\p{E}_{ji}$ and $\m{E}_{ji}$ respectively.

Let now
$\bar{f}:\bar{S} \to \Delta$
be the family obtained from $\tilde{f}:\tilde{S} \to \Delta$  by the
base change $t\in\Delta \mapsto t^6\in\Delta$. The  central fibre
$\bar S_0$ is isomorphic to $\tilde S_0$, so we will keep the above
notation for it.

\subsubsection{Step (\ref{Tetra:granchio-resol})} 

As a first step in the desingularization of $\bar{S}$,
we perform the following sequence of
operations  for all $i\in \{1,\ldots,4\} $. 
The total space $\bar{S}$ around $p_i$ is locally analytically
isomorphic to the hypersurface of $\mathbf C^ 4$ defined by the
equation $ xyz = t^6$ at the origin.
We blow-up $\bar{S}$ at $p_i$. The blown--up total space locally sits in $\bf C^ 4\times \P ^ 3$.
Let $(\xi:\eta:\zeta:\vartheta)$ be the homogeneous coordinates  in $\bf P^ 3$. Then 
the new total space is locally defined in $\bf C^ 4\times \P ^ 3$ by the equations
\begin{equation}
\label{eq1-granchio}
\xi^4\eta\zeta=\vartheta^6x^3, \quad
\xi\eta^4\zeta=\vartheta^6y^3, \quad
\xi\eta\zeta^4=\vartheta^6z^3, \quad \text{and} \quad
\xi\eta\zeta=\vartheta^3t^3.
\end{equation}
The equation of the exceptional divisor  (in the exceptional $\bf P^3$ of the blow--up of $\bf C^ 4$) is $\xi\eta\zeta=0$, hence
this is the union of three planes meeting transversely at a point $p'_i$  in $\bf P^ 3$.
For $i,j$ distinct in $\{1,\ldots,4\}$, we call $A_j^i$ the exceptional planes  meeting the proper transform of  $P_j$
(which, according to our conventions, we still denote by $P_j$, 
see \S \ref {conv}).

The equation of the new family around the point $p'_i$ given by 
$\bigcap_{j\neq i}A_j^i$ is $\xi\eta\zeta=t^3$
(which sits  in the affine chart $\vartheta=1$).
Next we blow-up the points $p'_i$, for $i\in \{1,\ldots,4\}$.  
The new exceptional divisor $T^ i$ at each point $p'_i$ is isomorphic
to the cubic surface with equation $\xi\eta\zeta=t^3$ 
in the $\P^3$ with coordinates $(\xi:\eta:\zeta:t)$. Note that $T^ i$ has
three $A_2$--double points, at the vertices of the triangle $t=0,\,\,
\xi\eta\zeta=0$.

Next we have to get rid of the singularities of the total space along the double curves of the central fibre. 
First we take care of the curves $C_{hk}:=P_h\cap P_k$, for $h,k$ distinct in $\{1,\ldots,4\}$.
The model we  constructed so far is defined
along such a curve by an equation of the type
$\xi\eta=\vartheta^6z^3$, 
(as it follows, e.g., from the third equation in \eqref{eq1-granchio} by setting $\zeta=1$). The curve
$C_{hk}$  is defined by $\xi=\eta=\vartheta=0$. If $i\in \{1,\ldots,4\}-\{h,k\}$,  the
intersection point $p_{hki}:=C_{hk} \cap A_h^i\cap A_k^i$ is cut out on $C_{hk}$ by 
the hyperplane with equation $z=0$.  Away from the $p_{hki}$'s, with
$i\in \{1,\ldots,4\}-\{h,k\}$,  the points of $C_{hk}$
are double points of type  $A_5$  for the total space.
We blow--up along this curve: this introduces new homogeneous 
coordinates $(\xi_1:\eta_1:\vartheta_1)$,
with new equations for the blow--up
\begin{equation*}
\xi_1^5\eta_1=\vartheta_1^6\xi^4z^3, \quad
\xi_1\eta_1^5=\vartheta_1^6\eta^4z^3, \quad \text{and} \quad
\xi_1\eta_1=\vartheta_1^2\vartheta^4z^3.
\end{equation*}
The exceptional divisor is defined by $\xi_1\eta_1=0$, and is the
transverse union of two ruled surfaces: we call $W_{hk}'$ the one that
meets $P_h$, and $W_{kh}'$ the other. 
The affine chart we are interested in is $\vartheta_1=1$, where the
equation is $\xi_1\eta_1=\vartheta^4z^3$.
We then blow--up along the curve $\xi_1=\eta_1=\vartheta=0$, which
gives in a similar 
way the new equation $\xi_2\eta_2=\vartheta^2z^3$ with the new
coordinates 
$(\xi_2:\eta_2:\vartheta_2)$. The exceptional divisor consists of
two ruled 
surfaces, and we call $W_{hk}''$ (resp. $W_{kh}''$) the one that meets
$W_{hk}'$ (resp. $W_{kh}'$).
Finally, by blowing-up along the curve $\xi_2=\eta_2=\vartheta=0$,
we obtain a new 
equation $\xi_3\eta_3=\vartheta_3^2z^3$, with new coordinates
$(\xi_3:\eta_3:\vartheta_3)$. The 
exceptional divisor is a ruled surface, with two  $A_2$--double
points at its intersection points with the curves $C^ i_{hk}:=A_h^i\cap A_k^ i$, with $i\in \{1,\ldots,4\}-\{h,k\}$.
We call it either $W_{hk}$ or $W_{kh}$, with no ambiguity.

The final step of our desingularization process consists in
blowing--up along the twelve curves
$C^i_{hk}$, with pairwise distinct $h,k,i\in \{1,\ldots, 4\}$. 
The total space is given along each of these curves
by an equation of the type  $\xi\eta=\vartheta^3t^3$ in the variables
$(\xi,\eta,\theta,t)$, obtained from the last equation in \eqref
{eq1-granchio} by setting $\zeta=1$.  The curve $C^i_{hk}$ is defined
by the local equations $\xi=\eta=t=0$, which
 shows that they consist  of $A_2$--double points for the total
space.  They also contain an $A_2$--double point of $W_{hk}$ and
$T^i$ respectively. A computation similar to the above shows that the  blow--up along these curves resolves all
singularities in a single move. The exceptional divisor over $C^i_{hk}$ 
is the union of two transverse ruled surfaces: we call $V_{hk}^i$ the
one that meets $A_h^i$, and $V_{kh}^i$ the other.

At this point, we have a semistable family $\pi: X\to \Delta$, whose
central fibre is depicted in Figure \ref{f:granchio}:
for each double curve we indicate  its 
self--intersections in the two components of the central fibre it
belongs to. This is obtained by applying the Triple Point
Formula (see Lemma \ref{l:triple-point}).

\begin{figure}
\begin{center}
\includegraphics[width=16cm]{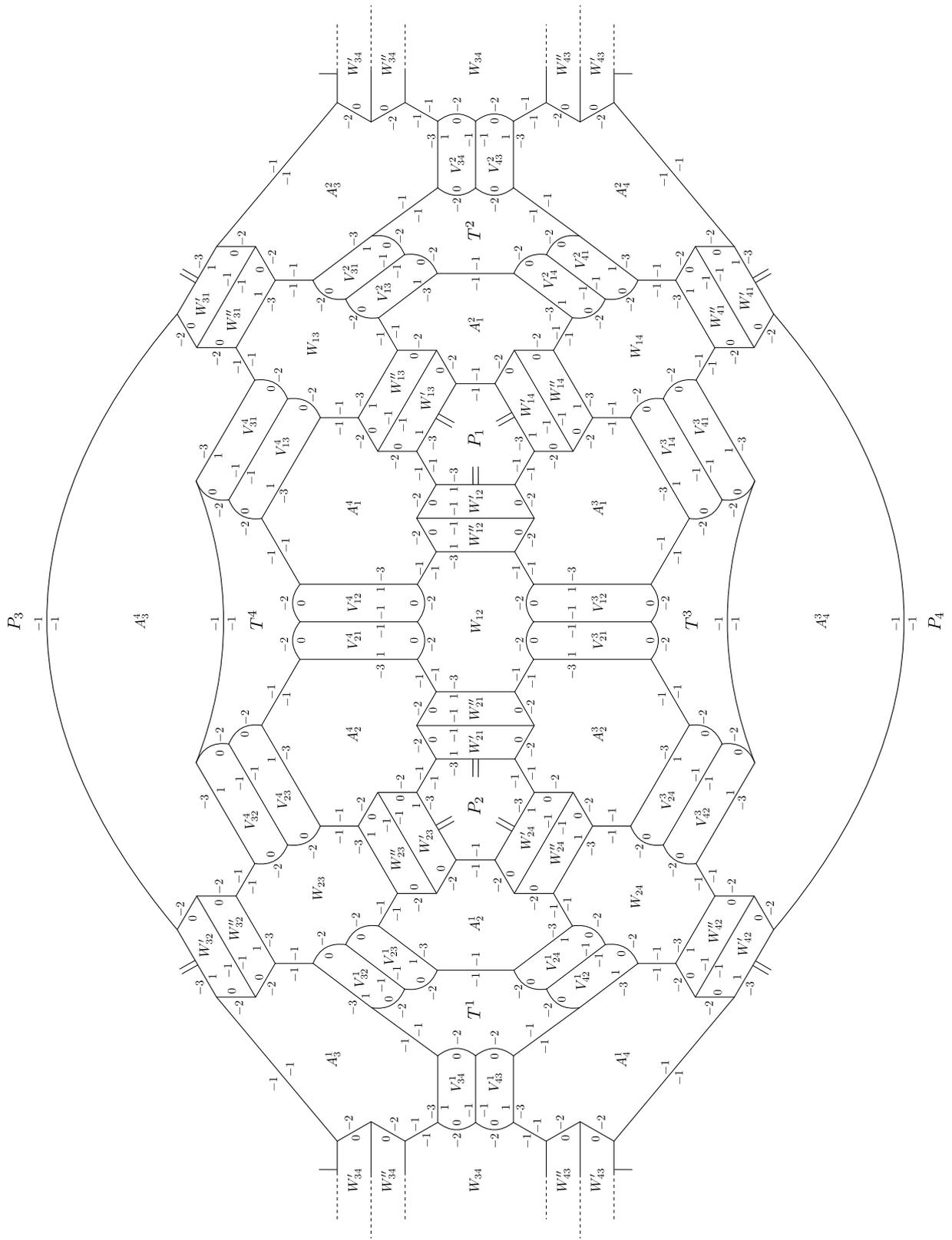}
\end{center}
\caption{Planisphere of the model $X_0$ of the degeneration into
four planes}
\label{f:granchio}
\end{figure}

\subsubsection{Step (\ref{Tetra:flop})}
For our purposes,
we need to further blow-up the total space along the twelve curves
$\Gamma^ i_{hk}:=V_{hk}^i\cap V_{kh}^i$.
This has the drawback of introducing components with multiplicity two
in the central fibre, namely the corresponding exceptional divisors.
To circumvent this, we will  flop these curves as follows.

\begin{figure}
\begin{center}
\includegraphics[width=12cm]{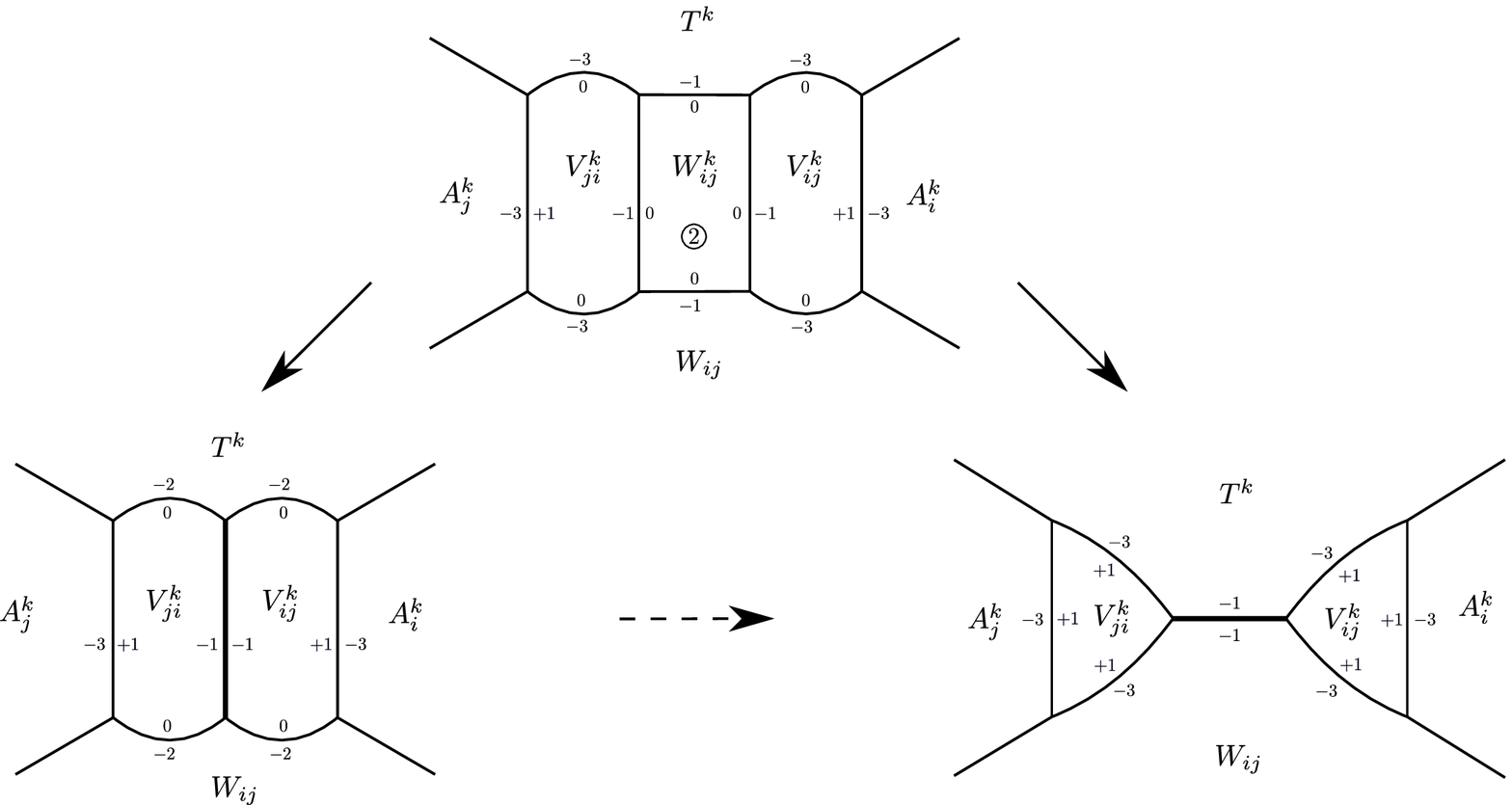}
\end{center}
\caption{One elementary flop of the birational transformation 
  $X \dashrightarrow \bar{X}$
}
\label{f:flop}
\end{figure}

Let $\hat{\pi}:\hat{X} \to \Delta$ be the family obtained
by  blowing-up $X$ along the $\Gamma^ i_{hk}$'s. We call $W_{hk}^i$ (or,
unambiguously, $W_{kh}^i$) the corresponding exceptional divisors:
they appear with multiplicity two in the central
fibre $\hat{X}_0$.  By applying the Triple Point Formula as in Remark \ref {rem:TPF}, 
one checks that the  surfaces $W_{kh}^i$ are all
isomorphic to $\P^1\times \P^1$. Moreover, 
it is possible to contract
$W_{hk}^i$ in the direction of the ruling cut out by $V_{hk}^i$ and
$V_{kh}^i$, as indicated on Figure \ref{f:flop}.
We call $\hat{X} \to \bar{X}$ the contraction of the twelve divisors $W_{hk}^i$ in
this way, and $\varpi: \bar{X} \to \Delta$ the corresponding semistable family of
surfaces. 

Even though $\bar{X} \dashrightarrow X$ is only a
birational map, we have a birational morphism
$\bar{X} \to \bar{S}$ over $\Delta$.

\subsection{Identification of the components of the
  central fibre} 
\label{s:identifications}

Summarizing, the irreducible components of the central fibre
$\bar X_0$ are the following:

\medskip\noindent
(i)
 The 4  surfaces $P_i$,  with  $1\leqslant i
  \leqslant 4$.

\smallskip \noindent
Each $P_i$ is a plane  blown--up at 6+3 points,
and $H$ (i.e. the pull-back of a general line in the plane, recall our
conventions in \S\ref{conv}) is the restriction class of
$\O_{\bar{X}}(1)$ on $P_i$.
For $j,k \in \{1,\ldots, 4\}-\{i\}$, we set
\[  L_{ij}:=P_i\cap W_{ij}'
\quad {\rm and} \quad 
G^ k_i:= P_i\cap A_i^k,\]
as indicated in Figure \ref{fig:P}. 
In addition to the three $(-1)$--curves $G_i^k$, we have on $P_i$ the
six exceptional curves $\p{E}_{ij}, \m{E}_{ij}$, for all $j\in
\{1,\ldots, 4\}-\{i\}$, 
with $\p{E}_{ij}, \m{E}_{ij}$ intersecting $L_{ij}$ at one
point. 
Moreover, for $j\in \{1,\ldots,4\}-\{i\}$, we have on $L_{ij}$
the two points $\plm{z}_{ji}$ defined as the strict transform of the
intersection $\plm{E}_{ji}\cap L_{ij}$ in $\tilde S$. We will denote
by $Z_i$ the $0$--dimensional scheme   
of length 6 given by $\sum_{j\neq i} (\p{z}_{ji}+ \m{z}_{ji})$.
We let $\mathcal{I}_{Z_i}\subset \O_{P_i}$ be its defining sheaf of
ideals.

\begin{figure}
\begin{center}
\begin{minipage}{5.96cm}
\includegraphics[width=5.96cm]{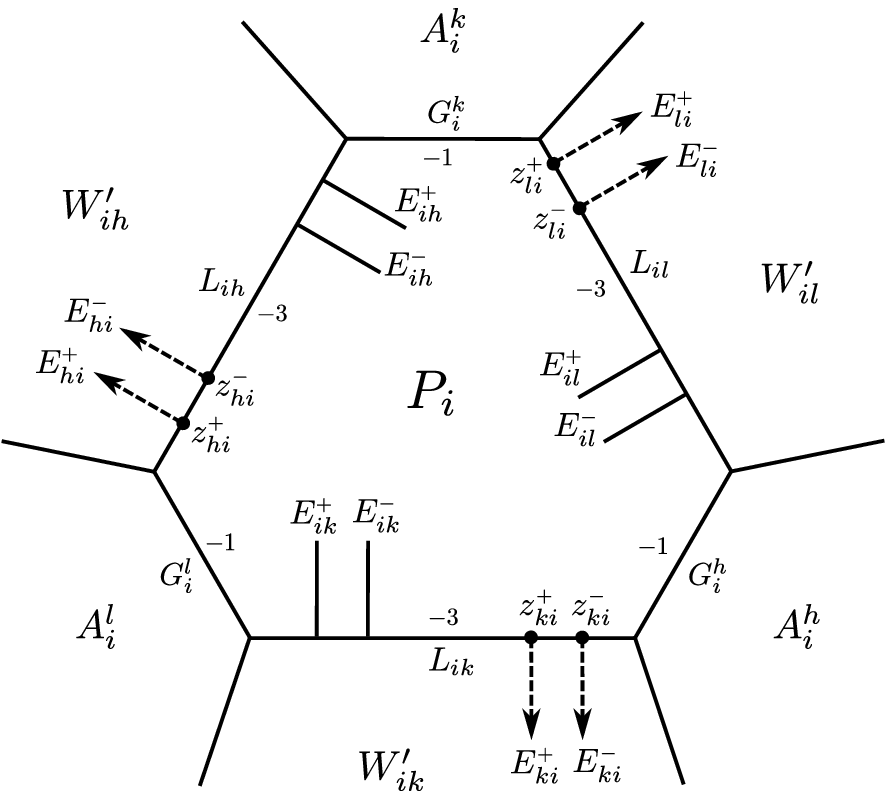}
\caption{Notations for $P_i\subset \bar{X}_0$}
\label{fig:P}
\end{minipage}
\hspace{1.5cm}
\begin{minipage}{7.11cm}
\includegraphics[width=7.11cm]{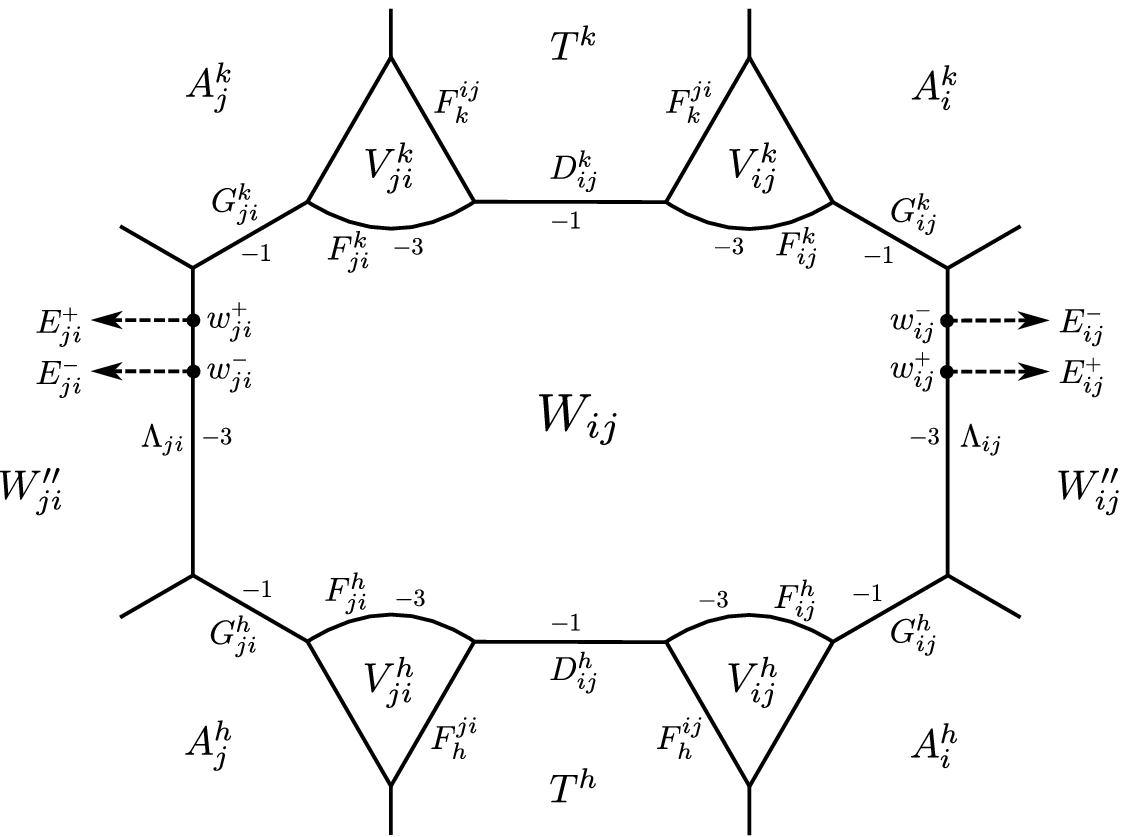}
\caption{Notations for $W_{ij} \subset \bar{X}_0$}
\label{fig:W}
\end{minipage}
\end{center}
\end{figure}

\medskip\noindent
(ii)
 The  24 surfaces $W'_{ij}, W''_{ij}$, with
  $i,j\in \{1,\ldots,4\}$ distinct.

\smallskip \noindent
Each of them  is isomorphic to ${\bf F}_1$.  We denote by $\vert
F\vert$ the ruling. Note that the divisor class $F$ corresponds to the
restriction of $\O_{\bar{X}}(1)$.

\medskip\noindent
(iii)
 The 6 surfaces  $W_{ij}$, with $i,j\in
  \{1,\ldots,4\}$ distinct.

\smallskip \noindent
For each $k\in \{1,\ldots, 4\}-\{i,j\}$, we set
\[
\Lambda_{ij}:=W_{ij}''\cap {W_{ij}}, 
\quad G_{ij}^k:= W_{ij}\cap  A_i^k,
\quad F_{ij}^k=W_{ij}\cap V_{ij}^k ,
\quad D_{ij}^k=W_{ij}\cap T^ k,
\]
and define similarly $\Lambda_{ji}$, $G_{ji}^k$, $F_{ji}^k$
($D_{ij}^k$ may be called $D_{ji}^k$ without ambiguity). This is
indicated in Figure \ref{fig:W}.

A good way of thinking to the surfaces $W_{ij}$ is to
consider them as (non--minimal) rational ruled surfaces, 
for which the two curves $\Lambda_{ji}$ and $\Lambda_{ij}$ are
sections which do not meet, and the two rational chains
\[
G_{ji}^k+ F_{ji}^k+2 D_{ij}^k+F_{ij}^k+G_{ij}^k,
\quad  k\in \{1,\ldots, 4\} - \{i,j\},
\]
are two disjoint reducible fibres of the ruling $\vert F\vert$. 
One has furthermore $\O_{W_{ij}}(F)=\O_{\bar{X}}(1)\otimes
\O_{W_{ij}}$. 

The surface $W_{ij}$ has the length 12 anticanonical cycle
\begin{equation}
\label{Wantican}
\Lambda_{ji}
+G_{ji}^k+ F_{ji}^k+ D_{ij}^k+F_{ij}^k+G_{ij}^k
+\Lambda_{ij}
+G_{ij}^h+ F_{ij}^h+ D_{ij}^h+F_{ji}^h+G_{ji}^h
\end{equation}
cut out by $\bar X_0-W_{ij}$, where we fixed $k$ and $h$ such that
$\{i,j,k,h\}=\{1,\ldots,4\}$.
It therefore identifies with a plane blown--up as
indicated in Figure \ref{f:WT}:
consider a general triangle $L_1,L_2,L_3$ in $\P^ 2$, with vertices
$a_1,a_2,a_3$, where $a_1$ is opposite to $L_1$, etc.;
then blow--up the three vertices $a_s$, and call $E_s$ the
corresponding exceptional divisors;
eventually blow--up the six points $L_r \cap E_s$, $r \neq s$, and
call $E_{rs}$ the corresponding exceptional divisors.
The obtained surface has the anticanonical cycle
\begin{equation}
\label{plane-antican}
L_1
+E_{13}+E_1+E_{23}+L_2+E_{21}
+E_1
+E_{31}+L_3+E_{32}+E_2+E_{12},
\end{equation}
which we identify term-by-term and in this order with the anticanonical
cycle \eqref{Wantican} of $W_{ij}$.

\begin{figure}
\begin{center}
\includegraphics[width=250pt]{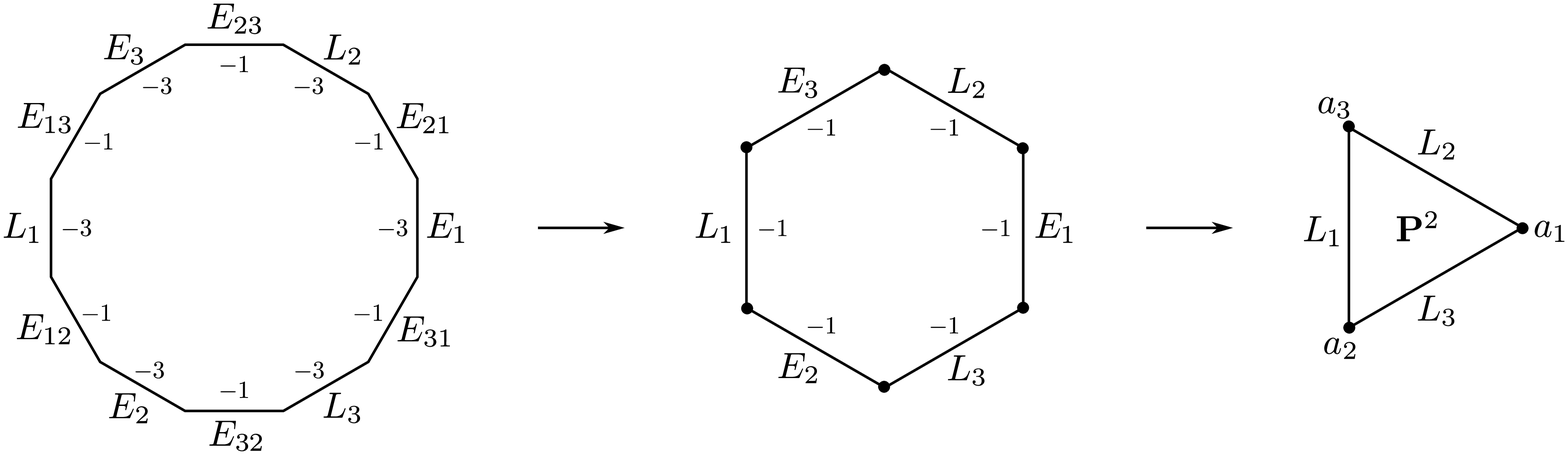}
\end{center}
\caption{$W_{ij}$ and $T^k$ as blown--up planes}
\label{f:WT}
\end{figure}

We let $H$ be, as usual, (the  transform of) a general line in the
plane
\begin{equation}
\label{acca}
H\lineq_{W_{ij}}\; \Lambda_{ji}+\sum\nolimits_{k\not\in
  \{i,j\}}(2G_{ji}^k+F_{ji}^k+D_{ji}^k).
\end{equation}
The ruling $\vert F\vert$  is  the strict transform of  the pencil
of lines through the point $a_1$, hence
\begin{equation}
\label{effe}
\vert F\vert =\bigl\vert H-(\Lambda_{ij}+G_{ij}^k+G_{ij}^h)\bigr\vert, 
\quad \text{with} \quad \{1,\ldots, 4\}=\{i,j,k,h\}.
\end{equation}

\medskip\noindent
(iv)
  The 4 surfaces $T^ k$, with $1\leqslant
  k\leqslant 4$.

\smallskip \noindent
Here we set 
\[
\Gamma^k_i= T^k\cap A^k_i, \ \text{for}\ i\in \{1,\ldots,4\}-\{k\},
\quad \text{and}\quad F_k^{ij}= T^ k\cap V^ k_{ij}, 
\ \text{for}\ i,j\in \{1,\ldots,4\}-\{k\} \ {\rm distinct}.
\]
Also recall that $D_{ij}^k=T^k\cap W_{ij}$ for $i,j\in
\{1,\ldots,4\}-\{k\}$ distinct.
This is indicated in Figure \ref{fig:T}.

\begin{figure}
\begin{center}
\includegraphics[height=5.5cm]{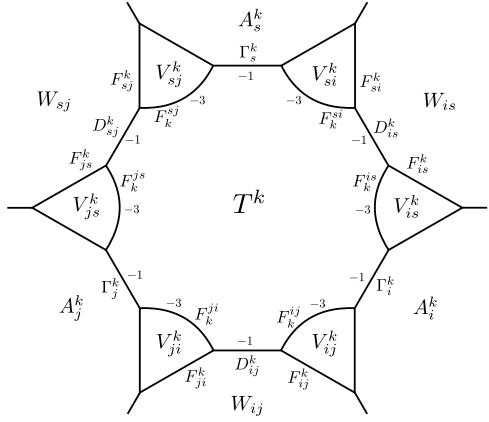}
\end{center}
\caption{Notations for $T^k \subset \bar{X}_0$}
\label{fig:T}
\end{figure}

Each $T^ k$  identifies with a plane blown--up as
indicated in Figure \ref{f:WT}, as in the case of the $W_{ij}$'s:
it has the length 12 anticanonical cycle
\begin{equation}
\label{Tantican}
F_k^{js}+D_{sj}^k+F_k^{sj}
+ \Gamma_s^k
+F_k^{si} +D^k_{is}+F_k^{is}
+\Gamma^k_i
+F_k^{ij}+D_{ij}^k+F_k^{ji}
+\Gamma^k_j
\end{equation}
(where we fixed indices $s,i,j$ such that
$\{s,i,j,k\}=\{1,\ldots,4\}$)
cut out by $\bar X_0-T^k$ on $T^ k$,
which we identify term-by-term and in this order with the anticanonical
cycle \eqref{plane-antican}.
This yields
\begin{equation}
\label{acca-T}
H \lineq_{T^k}\; F_k^{js}+(2D_{sj}^k+F_k^{sj}+\Gamma_s^k)
+(2\Gamma_j^k+F_k^{ji}+D_{ij}^k).
\end{equation}

We have on $T^k$ the proper transform of 
a  pencil of (bitangent) conics that meet the curves 
$\Gamma_s^k$   and  $D_{ij}^k$ in one point respectively, 
and do not meet any other curve in the anticanonical cycle
\eqref{Tantican}:
we call this pencil $|\Phi^k_s|$, and we have
\begin{equation*}
\bigl|\Phi^k_s\bigr|=
\bigl\vert 2H-(F_k^{sj}+D_{sj}^k+2\Gamma_s^k)
-(F_k^{ji}+\Gamma_j^k+2D_{ij}^k)
\bigr\vert.
\end{equation*}

The restriction of $\O_{\bar{X}}(1)$ on $T^k$ is
trivial.

\medskip\noindent
(v)
  The 12 surfaces $A^k_i$, with $i,k\in \{1,\ldots,
  4\}$ distinct.

\smallskip \noindent
Each of them  identifies with a 
blown--up plane as indicated in Figure \ref{f:A}.
It is equipped with the ruling $|H-\Gamma_i^k|$, 
the members of which meet the curves $G_i^k$ and $\Gamma_i^k$ at one
point respectively, and do not meet any other curve in the length 8
anticanonical cycle cut out by $\bar{X}_0-A_i^k$ on $A_i^k$.
The restriction of $\O_{\bar{X}}(1)$ on $A_i^k$ is trivial.

\begin{figure}
\begin{center}
\includegraphics[width=11cm]{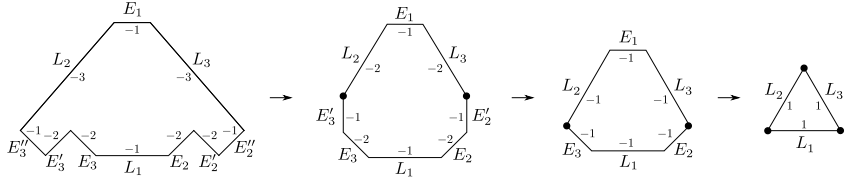}
\end{center}
\caption{$A_j^k$ as a blown-up plane}
\label{f:A}
\end{figure}

\medskip\noindent
(vi)
  The 24 surfaces $V^ k_{ij}$ with $i,j,k\in
  \{1,\ldots 4\}$ distinct.

\smallskip \noindent
These are all copies of $\P^ 2$, on which the restriction of
$\O_{\bar{X}}(1)$ is trivial.

\subsection{The limit linear system, I: construction}
\label{s:4planes-limlin}

According with the general principles stated in \S \ref{s:resolving},
we shall now describe  the limit linear system 
of $|\O_{\bar X_t}(1)|$ as $t\in \Delta^*$ tends to $0$.
This will suffice for the proof,  
presented in \S \ref {s:good}, that $\varpi:\bar{X} \to \Delta$ is a 
$\delta$--good model for  $1\leqslant \delta\leqslant 3$. 

We start with  ${\fP} := \P(\varpi_*(\O_{\bar X}(1)))$, 
which is a $\P^ 3$--bundle over $\Delta$, whose fibre 
at $t\in \Delta$ is $\vert \O_{\bar X_t}(1)\vert$. 
We set $\L=\O_{\bar{X}}(1)$, and $\vert \O_{\bar X_t}(1)\vert=|\L_t|$; 
note that $\vert \L_0 \vert\cong \vert \O_{S_0}(1)\vert$. 
We will often use the same notation to denote a divisor (or a divisor
class)  on the central fibre and its restriction to a component of the
central fibre, if this does not cause any confusion.
 
We will  proceed as follows: \\
{\it \begin{inparaenum}[(I)]
\item \label{Tetra:faces}
we first blow--up $\fP$ at the  points $\pi_i$ corresponding to the  
irreducible components $P_i$ of $S_0$, for $i\in \{1,\ldots,4\}$  (the
new central fibre then consists of  $ \vert \O_{S_0}(1)\vert \cong \bf P^
3$ blown--up at four independent points, plus the four exceptional
$\P^ 3$'s); \\ 
\item \label{Tetra:edges}
next, we blow--up the total space along the proper transforms $\ell_{ij}$ of the six lines of $ \vert \O_{S_0}(1)\vert$  joining two distinct points $\pi_i, \pi_j$, with
$i,j\in \{1,\ldots,4\}$, corresponding to pencils of planes  with base locus an edge of $S_0$ 
(the new central fibre is the proper transform of the previous one,
plus the six exceptional
$\P\bigl(\O_{\P^1}\oplus \O_{\P^1}(1)^{\oplus 2}\bigr)$'s);\\
\item \label{Tetra:vertices}
finally, we further blow--up along the proper transforms of the planes
$\Pi_k$  corresponding to the webs of 
planes passing through the  vertices $p_k$ of $S_0$, for $k\in
\{1,\ldots,4\}$ (this adds four more exceptional divisors  to the
central fibre, for a total of fifteen irreducible components).
\end{inparaenum}} \\
In other words, we successively blow--up $\fP$ along all the cells of
the tetrahedron dual to $S_0$ in $\fP_0$, by increasing order of
dimension. 

Each of these blow--ups will be interpreted in terms of suitable
twisted   linear systems as indicated in Remark \ref{r:twist2}. 
It will then become apparent that every point in the central fibre of
the obtained birational modification of $\fP$ corresponds to a curve
in $\bar{X}_0$ (see \S\ref{s:good}), and hence that this modification
is indeed the limit linear system $\fL$.

\subsubsection{Step (\ref{Tetra:faces})}

\newcounter{flag}
\setcounter{flag}{\arabic{enumi}}
%
%

In $\H^0(\bar X_0, \O_{\bar X_0}(1))$ there is for each 
$i\in \{1,\ldots,4\}$ 
the 1--dimensional subspace of sections vanishing on  $P_i$, which
corresponds to the sections of $\H^0(S_0, \O_{S_0}(1))$ vanishing on
the plane $P_i$. As indicated in Remark \ref{r:twist2},
in order to construct the limit linear system, we have to blow up the
corresponding points $\pi_i\in \vert \L_0\vert$.  
Let  ${\fP}' \to {\fP}$  be this blow--up, and call $\tilde\fL_i$,
$1\leqslant i\leqslant 4$, the exceptional divisors.
Each $\tilde\fL_i$ is a $\P^3$, and can be interpreted as the trace of the  linear
system
$\bigl| \L_0(-P_i) \bigr|$ on $X_0$
(see Lemma \ref {l:tg-id} and
Example \ref {ex:gh}).
However, any section of
$\H^0(\bar X_0,\L_0(-P_i))$
still vanishes on components of $\bar X_0$ different from $P_i$. By subtracting all of them
with the appropriate multiplicities (this computation is tedious but
not difficult  and can be left to the reader), one sees that
$\tilde\fL_i$ can be identified as the linear system
$\fL_i:=\bigl| \L_0(-M_i)\bigr |$,
where 

\begin{multline}
\label{faces-twist}
M_i:=
6P_i 
+\sum_{j\neq i}(5W_{ij}'+4W_{ij}''+3W_{ij}+2W_{ji}''+W_{ji}')+\\
+\sum_{k\neq i}\Biggl(
2T^k+4A_i^k
+\sum_{j\not\in\{i,k\}}\left(3V_{ij}^k +2V_{ji}^k +A_j^k \right)
+ \sum_{\{j<\bj\}\cap \{i,k\}=\emptyset}\left(
  V_{j \bj}^k+ V_{\bj j}^k \right)
\Biggr).
\end{multline}

With the notation introduced in \S \ref {s:identifications}, one has:

\begin{lemma} \label{lem:comput} 
The restriction class of $\L_0(-M_i)$ to the irreducible components of $\bar X_0$ is as follows:\\
\begin{inparaenum}[(i)]
\item \label{lem:comput:P_i} on $P_i$, we find $4H-\sum_{j\neq i}
  (\p{E}_{ij}+\m{E}_{ij})$; \\
\item \label{i} on $P_j$, $j\neq i$, we find
  $\p{E}_{ji}+\m{E}_{ji}$;\\
\item \label{lem:comput:W}
for each $j\neq i$, we find $2F$ on each of the surfaces $W_{ij}'$,
$W_{ij}''$, $W_{ij}$, $W_{ji}''$, $W_{ji}'$.\\
\item on the remaining components the restriction is trivial.
\end{inparaenum}
\end{lemma}

\begin{proof} This is a tedious but standard  computation. 
As a typical sample we prove \eqref{lem:comput:W}, and 
leave the remaining cases to the reader. Set
$\{h,k\}=\{1,\ldots,4\}-\{i,j\}$. Then, recalling \eqref{acca} and
\eqref{effe}, we see that the restriction of $\L_0(-M_i)$ to $W_{ij}$
is the line bundle determined by the divisor class 
\begin{multline*}
 F +\Bigl( W_{ji}'' -W_{ij}''
+\sum_{k\not\in\{i,j\}}\bigl( 2A_j^k +V_{ji}^k +T^k -A_i^k \bigr)
\Bigr) \biggr|_{W_{ij}} \\
\begin{aligned}
&\lineq F + \Lambda_{ji}-\Lambda_{ij}
+(2G_{ji}^k+F_{ji}^k+D_{ij}^k-G_{ij}^k)
+(2G_{ji}^h+F_{ji}^h+D_{ij}^h-G_{ij}^h)\\
&= F +\bigl(\Lambda_{ji}+(2G_{ji}^k+F_{ji}^k+D_{ij}^k)
+(2G_{ji}^h+F_{ji}^h+D_{ij}^h) \bigr)
-(\Lambda_{ij}+G_{ij}^k+G_{ij}^h)
= 2F.\\
\end{aligned}
\end{multline*}
\end{proof}

From this, we deduce that $\fL_i$ identifies with its restriction to
$P_i$:

\begin{proposition}\label{prop:descrL_i}
There is a natural isomorphism 
\begin{equation} \label{4ics-3angle}
\fL_i\cong
\Bigl|
\O_{P_i}\Bigl( 4H-\sum\nolimits_{j\neq i}(\p{E}_{ij}+\m{E}_{ij})
\Bigr)
\otimes {\cal I}_{Z_i}
\Bigr|.
\end{equation}
\end{proposition}

\begin{proof}
For each $j\neq i$, the restriction of $\fL_i$ to $P_j$ has 
$\p{E}_{ji}+\m{E}_{ji}$ as its only member.
This implies that its restriction to $W_{ji}'$ has only one member as
well, which is the sum of the two curves in $|F|$ intersecting
$\p{E}_{ji}$ and $\m{E}_{ji}$ respectively.
On $W_{ji}''$, we then only have the sum of the two curves in $|F|$
intersecting the two curves on $W_{ji}'$ respectively, and so on on
$W_{ij}$, $W_{ij}''$, and $W_{ij}'$.
Now the two curves on $W_{ij}'$ impose the two base points
$\p{z}_{ji}$ and $\m{z}_{ji}$ to the restriction of $\fL_i$ to $P_i$.
The right hand side in \eqref{4ics-3angle}
being $3$--dimensional, this ends the proof with
\eqref{lem:comput:P_i} of Lemma \ref{lem:comput}.
\end{proof}

\subsubsection{Step (\ref{Tetra:edges})}\label {step(2)}
Next, we consider the blow--up ${\fP}'' \to {\fP}'$ along the proper 
transforms $\ell_{ij}$ of the six lines of $\vert \L_0\vert$  joining
two distinct points $\pi_i, \pi_j$, with $i,j\in \{1,\ldots,4\}$,
corresponding to the pencils of planes in 
$|\O_{S_0}(1)|$ respectively containing the lines $P_i \cap P_{j}$.
The exceptional divisors are isomorphic to $\P(\O_{\P^ 1}\oplus
\O_{\P^ 1}(1)^ {\oplus 2} )$; we call them $\tilde{\fL}_{ij}$, 
$1\leqslant i<j\leqslant 4$.
Arguing as in Step (\ref{Tetra:faces}) and  leaving the details to the
reader,  we see  that  
$\tilde{\fL}_{ij}$ is in a natural way a birational modification (see \S \ref {s:rat4ics} below) of the complete linear system 
$\fL_{ij}:=\left|\L_0(-M_{ij})\right|$, where 
\begin{multline}
\label{spigolo-twist}
M_{ij}:=
3W_{ij}+(2W_{ji}''+W_{ji}')+(2W_{ij}''+W_{ij}')+ \\
+\sum_{k\not\in\{i,j\}}\Biggl(
2T^k
+\sum_{s\neq k} A_s^k
+2\left(V_{ji}^k+V_{ij}^k\right)
+\sum_{\substack{s\in\{i,j\} \\ r\not\in\{i,j,k\}}}
\left(V_{sr}^k+V_{rs}^k\right)
\Biggr).
\end{multline}

We will denote by $k<h$ the two indices in $\{1,\ldots,4\}-\{i,j\}$,
and go on using the notations introduced in \S\ref{s:identifications}.

\begin{lemma} \label{lem:comput2} The restriction class of
$\L_0(-M_{ij})$ to the irreducible components of $\bar{X}_0$ is as
follows:\\
\begin{inparaenum}[(i)]
\item \label {i0} on $P_k$ (resp. $P_h$) we find $H-G_k^h$ 
(resp. $H-G_h^k$);\\
\item \label {i00} on each of the surfaces $W'_{kh}$, $W''_{kh}$,
  $W_{kh}$, $W''_{hk}$, and $W'_{hk}$, we find $F$;\\
\item \label{i1} on $A_h^k$ (resp. $A_k^h$) we find  $H-\Gamma_h^k$ 
(resp. $H-\Gamma_k^h$);\\
\item \label{ii1} on $T^k$ (resp. $T^h$), we find $\Phi^k_h$
(resp. $\Phi^h_k$);\\
\item \label{iii1} on $P_i$ (resp. $P_j$), we find
$\p{E}_{ij}+\m{E}_{ij}$ 
(resp. $\p{E}_{ji}+\m{E}_{ji}$);\\
\item \label{iv1} on $W_{ij}',W_{ij}'',W_{ji}'',W_{ji}'$, we find  $2F$;\\
\item \label{v1} on $W_{ij}$, with $H$ as in \eqref{acca}, we find 
\begin{equation*}
4H -2\left(\Lambda_{ij}+G_{ij}^k+G_{ij}^h \right)
-\left( F_{ji}^k+G_{ji}^k+D_{ij}^k \right)
-\left( F_{ji}^h+G_{ji}^h+D_{ij}^h \right)
-D_{ij}^k-D_{ij}^h;
\end{equation*}
\item \label{vi1} on the remaining components  the restriction is trivial.
\end{inparaenum}
\end{lemma}

\begin{proof} As for Lemma \ref {lem:comput}, this is a tedious but not difficult computation. 
Again we make a sample verification, proving \eqref {v1} above. The restriction class is
\begin{multline*}
F +\Bigl( W_{ji}'' +
\sum_{l=k,h}\left( 2A_j^l +V_{ji}^l +T^l +V_{ij}^l +2A_i^l \right)
+W_{ij}''\vert _{W_{ij}} \Bigr) 
\biggr|_{W_{ij}} \\
\begin{aligned}
& \lineq F+
\Lambda_{ji} + \sum_{l=k,h}\left( 
2G_{ji}^l +F_{ji}^l+D_{ij}^l +F_{ij}^l+2G_{ij}^l
\right)+\Lambda_{ij} \\
\end{aligned}
\end{multline*}
which, by taking into account the identification of Figure \ref{f:WT},
i.e. with \eqref{acca} and \eqref{effe},
is easily seen to be equivalent to the required class. 
\end{proof}

Let $\plm{w}_{ij} \in W_{ij}$ be the two points cut out on
$W_{ij}$ by the two connected chains of curves in 
$|F|_{W_{ij}'} \times |F|_{W_{ij}''}$ meeting $\plm{E}_{ij}$
respectively. We let $\plm{w}_{ji} \in W_{ij}$ be the
two points defined in a similar fashion by starting with $\plm{E}_{ji}$.
Define the $0$-cycle 
$
Z_{ij}=\p{w}_{ij}+\m{w}_{ij}+\p{w}_{ji}+\m{w}_{ji}
$
on $W_{ij}$, and let 
$\mathcal{I}_{Z_{ij}} \subset \O_{W_{ij}}$ be its defining sheaf of ideals.

\begin{proposition}\label{prop:descrL_ij}
There is a natural isomorphism between $\fL_{ij}$ and its
restriction to $W_{ij}$, which is the $3$--dimensional linear system
\begin{equation} \label{eq:systema}
\Bigl| \O_{W_{ij}}\Bigl(4H -2\left(\Lambda_{ij}+G_{ij}^k+G_{ij}^h \right)
-\left( F_{ji}^k+G_{ji}^k+D_{ij}^k \right)
-\left( F_{ji}^h+G_{ji}^h+D_{ij}^h \right)
-D_{ij}^k-D_{ij}^h \Bigr)
\otimes \mathcal{I}_{Z_{ij}} \Bigr|,
\end{equation}
where we set $\{1,\ldots,4\}=\{i,j,h,k\}$, and 
$H$ as in \eqref{acca}.
\end{proposition}

\begin{proof} Consider a triangle $L_1,L_2,L_3$ in $\P^ 2$, with vertices
  $a_1,a_2,a_3$, where $a_1$ is opposite to $L_1$, etc. Consider the linear
  system $\mathcal W$ of quartics with a double point at $a_1$, two
  simple  base points  infinitely near to $a_1$ not on $L_2$ and $L_3$, two
  base points at $a_2$ and $a_3$ with two infinitely near base points
  along $L_3$ and $L_2$ respectively, two more base points along $L_1$.
There is a birational transformation of $W_{ij}$ to the plane (see
Figure \ref{f:WT}) mapping \eqref{eq:systema} to a linear
system of type $\mathcal W$. One sees that  two independent
conditions are needed to impose to the curves of $\mathcal W$ to
contain the three lines $L_1,L_2,L_3$ and the residual system consists
of the pencil of lines through $a_1$. This proves the dimensionality
assertion  (see \S \ref {s:rat4ics} below for a more detailed
discussion). 

Consider then the restriction of $\fL_{ij}$ to the chain of surfaces
\begin{equation*}
P_j+W_{ji}'+W_{ji}''+W_{ij}+W_{ij}''+W_{ij}'+P_i.
\end{equation*}
By taking into account \eqref{iii1}, \eqref {iv1}, and \eqref {v1}, of 
Lemma \ref {lem:comput2}, we 
see that each divisor $C$ of this system determines, and is
determined, by its restriction $C'$ on $W_{ij}$, since $C$ consists of
$C'$ plus four rational tails matching it. 

The remaining components of $\bar X_0$ on
which $\fL_{ij}$ is non--trivial, all sit in the chain 
\begin{equation} \label{chain_kl}
T^k +A^k_h+P_h+W_{lk}' +W_{lk}'' +W_{hk} +W_{kh}'' +W_{kh}' 
+P_k +A^h_k +T^h.
\end{equation}
The restrictions of  $\fL_{ij}$ to each
irreducible component of this chain is a base point free pencil of
rational curves, 
hence $\fL_{ij}$ restricts on \eqref{chain_kl} to the $1$--dimensional
system of  connected chains of rational curves in these
pencils: we call it $\fN^{kh}$.
Given a curve in $\fL_{ij}$, it cuts $T^k$ and $T^ h$ in one point
each, and there is a unique chain of rational curves in $\fN^{kh}$
matching these two points.
\end{proof}

\subsubsection{Step (\ref{Tetra:vertices})}
\label{step:vertices}
Finally, we consider the blow-up 
${\fP}''' \to {\fP}''$
along the proper transforms of the three planes that are strict
transforms of the webs of planes in $|\O_{S_0}(1)|$ containing a
vertex $p_k$, with $1\leqslant k\leqslant 4$. 
For each $k$, the exceptional divisor $\tilde{\fL}^k$
is a birational modification 
(see \S\ref{s:cubics} below)
of the complete linear system 
$\fL^k:=\bigl|\L_0(-M^k)\bigr|$, where 
\begin{equation*}
M^k:=
2T^k
+\sum_{s\neq k} A_s^k
+\sum_{\{s<r\} \not\ni k}\left( V_{sr}^k+V_{rs}^k \right).
\end{equation*}

\begin{lemma} \label{lem:comput3} The restriction class of
  $\L_0(-M^ k)$ to the irreducible components of $\bar{X}_0$ is as
  follows:\\  
\begin{inparaenum}[(i)]
\item \label {i2} on $P_i$, $i\neq k$, we find $H-G^ k_i$;\\
\item \label {ii2} on $A^ k_i$, $i\neq k$, we find $H-\Gamma_i^k$; \\
\item \label{iii2} on $P_k$, as well as on the chains
  $W_{ik}'+W_{ik}''+W_{ik}+W_{ki}''+W_{ki}'$, $i\neq k$, 
we find the restriction class of $\L_0$;\\
\item \label{iii3} on $T^k$,
we find 
\begin{equation*}
3H-(F_k^{sj}+D^k_{sj}+2\Gamma_s^k)
-(F_k^{ji}+D_{ij}^k+2\Gamma_j^k)
-(F_k^{is}+D_{is}^k+2\Gamma_i^k),
\end{equation*}
with $\{s,i,j,k\}=\{1,\ldots,4\}$, and $H$ as in \eqref{acca-T};\\
\item \label{iii4}  on the remaining components  it is trivial. 
\end{inparaenum}
\end{lemma}

\begin{proof}
We limit ourselves to a brief outline of how things work for
$T^k$. The restriction class is 
\begin{equation*}
\Bigl( 
\sum_{r\neq k} A_r^k
+\sum_{\{r<r'\} \not\ni k}\left( V_{rr'}^k+2W_{rr'}+V_{r'r}^k \right)
\Bigr) 
\biggr|_{T^k}
\end{equation*}
which is seen to be equal to the required class with the
identification of Figures \ref{f:WT} and \ref{fig:T},
i.e. with $H$ as in \eqref{acca-T}.
\end{proof}

\begin{proposition}\label{prop:descr2}
There is a natural isomorphism between $\fL^k$ and its
restriction to $T^k$, which is the $3$--dimensional linear system
\begin{equation*}
\bigl| 3H-(F_k^{sj}+D^k_{sj}+2\Gamma_s^k)
-(F_k^{ji}+D_{ij}^k+2\Gamma_j^k)
-(F_k^{is}+D_{is}^k+2\Gamma_i^k) \bigr|,
\end{equation*}
where we set $\{s,i,j,k\}=\{1,\ldots,4\}$, and $H$ as in
\eqref{acca-T}. 
\end{proposition}

\begin{proof}
This is similar (in fact, easier) to the proof of Proposition
\ref{prop:descrL_ij}, so we will be sketchy here. The dimensionality
assertion will be discussed 
in \S\ref{s:cubics} below.

For each $i\neq k$, the restriction of $\fL^k$ to each irreducible
component of the chain 
\begin{equation}
\label{chain^k_i}
A^k_i+P_i+W_{ik}'+W_{ik}''+W_{ik}+W_{ki}''+W_{ki}'
\end{equation}
is a base point free pencil of rational curves, 
and $\fL^k$ restricts on \eqref{chain^k_i} to the $1$--dimensional
system of  connected chains of rational curves in these
pencils, that we will call $\fN^{k}_i$.

Now the general member of $\fL^k$ consists of a curve in 
$\fL^k\bigr|_{T^k}$, which uniquely determines three chains of rational
curves in $\fN^{k}_i$, $i\neq k$, which in turn determine a unique
line in $|\O_{P^k}(H)|$.
\end{proof}

\subsection{The linear systems $\fL_i$.}
\label{s:planes}

Let  $a,b,c$ be three independent  lines   in $\P^ 2$, and consider a $0$--dimensional scheme $Z$ cut out on $a+b+c$ by
a general quartic curve.  Consider the linear system $\mathcal P$ of plane quartics containing $Z$.
This is a linear system of dimension 3. Indeed containing the union of the three lines $a,b,c$ is one condition for
the curves in $\mathcal P$ and the residual system is the $2$--dimensional complete linear system of all  lines in the plane. 

Proposition \ref{prop:descrL_i} shows that $\fL_i$ can be identified
with a system of type $\mathcal P$.
We denote by $\sigma_i: P_i\dasharrow \P^ 3$ (or simply by $\sigma$)
the rational map determined by $\fL_i$ and by $Y$ its image, which is
the same as the image of the plane via the rational map determined by
the linear system $\mathcal P$.

\begin{proposition}\label{prop:monoid} The map $\sigma: P_i\dasharrow Y$ is birational, and $Y$ is a \emph{monoid} quartic surface,
with a triple point $p$ with tangent cone consisting of a triple  of  independent planes through $p$, and with no other singularity.
\end{proposition}

\begin{proof} The triple point $p\in Y$ is the image of the curve $C=\sum_{i=j}^  3(2D_i^j+L_{ij})$ (alternatively, of the sides of the triangle $a,b,c$). By subtracting $C$ to $\fL_i$ one gets a homaloidal net, mapping to the net of lines in the plane. This proves the assertion.
\end{proof}

\begin{remark} \label{rem:monoid}
The image of $\bar{X}$ by the complete linear system
$|\L(-M_i)|$ provides a model $f':S'\to \Delta$ of the initial family
$f:S\to \Delta$, such that the corresponding flat limit of $S'_t\cong
S_{t^6}$ with $t\neq 0$, is  $S'_0=Y$ the quartic monoid image of the face
$P_i$ of the tetrahedron via $\sigma$. 
The map $\bar{X}_0 \to S'_0$ contracts all other irreducible
components of $\bar{X}_0$ to the triple point of the monoid.
\end{remark}

\begin{remark}
 Theorem \ref {T:triangle} says that
the degree of the dual surface of the monoid $Y$ is 21.
\end{remark}

The strict transform of $\tilde \fL_i$ in $\fP_0'''$ (which we still
denote by $\tilde \fL_i$, see \S\ref{conv}) can  be identified
as a blow--up of $\fL_i \cong \mathcal{P}$:
first blow--up the three points corresponding to the three
non--reduced curves $2a+b+c$, $2b+a+c$, $2c+a+b$.  
Then blow--up the proper transforms of the three pencils of
lines with centres at $A, B, C$ plus the fixed part $a+b+c$.  
We will interpret this geometrically in \S\ref{s:good}, using Lemma
\ref{l:tg-id}.

\subsection{The linear systems $\fL_{ij}$.}
\label{s:rat4ics}

Next, we need to study some of the geometric properties of the
linear systems $\fL_{ij}$ as in  Proposition \ref{prop:descrL_ij}.
Consider the rational map $\varphi_{ij}: W_{ij}\dasharrow  \P^ 3$ (or
simply $\varphi$) determined by ${\fL}_{ij}$. Alternatively, one may
consider the rational map, with the same image $W$  (up to projective
transformations), determined by the planar linear system $\mathcal W$
of quartics considered in the proof of Proposition \ref {prop:descrL_ij}.
 
\begin{proposition}
\label{p:W-quartic}
The map $\varphi$ is birational onto its image, which is  a quartic
surface $W \subset \P^3$, with a double line $D$, and two triple
points  on $D$.
\end{proposition}

\begin{proof}
 First we get rid of the four base points 
in  $Z_{ij}$ by blowing them up and taking the proper transform $\bar
{\fL}_{ij}$ of the system. Let $u: \bar {W} \to W_{ij}$
be this blow--up, and let $\plm{I}_{ij}$
(resp. $\plm{I}_{ji}$) 
be the two $(-1)$--curves that meet $\Lambda_{ij}$
(resp. $\Lambda_{ji}$).

The strict transform $\bar {\fL}_{ij}:= u^ *({\fL}_{ij})-  \left(
\p{I}_{ij}+\m{I}_{ij}+\p{I}_{ji}+\m{I}_{ij} \right)$,  has
self--intersection $4$. Set, as usual, $\{1,\ldots,4\}=\{i,j,h,k\}$
and consider the  curves  
\begin{equation}
\label{C_ij}
C_{ji}:=\Lambda_{ji}+(2G_{ji}^k+F_{ji}^k)+(2G_{ji}^h+F_{ji}^h)
\quad \text{and} \quad
C_{ij}:=\Lambda_{ij}+(2G_{ij}^k+F_{ij}^k)+(2G_{ij}^h+F_{ij}^h).
\end{equation}
One has 
\[\bar {\fL}_{ij}\cdot C_{\bf s}=0, \quad p_a(C_{\bf s})=0,\quad
C_{\bf s}^ 2=-3, \quad
\text{for} \quad {\bf s}\in \{ (ij),(ji) \}
.\] 
 By mapping $\bar W$ to  $W_{ij}$, and this to the plane as in Figure
 \ref{f:WT} with \eqref{Wantican} and \eqref{plane-antican}
 identified, one sees that $C_{ji}$ goes to the 
 line $L_1$ and $C_{ij}$ to the union of the two lines $L_2,L_3$.  
The considerations in the proof of Proposition \ref {prop:descrL_ij} show
that $\bar {\fL}_{ij}$ has no base points on $C_{ji}\cup C_{ij}$ (i.e.,
$\mathcal W$ has only the prescribed base points along the triangle
$L_1+L_2+L_3$).  
On the other hand, the same considerations show that the base points
of $\bar {\fL}_{ij}$ may only lie on $C_{ji}\cup C_{ij}$. This shows that
$\bar {\fL}_{ij}$ is base points free, and the associated morphism
$\bar \varphi:\bar{W} \to \P^3$ contracts $C_{ji}$ and $C_{ij}$ to points
$c_1$ and $c_2$ respectively.  

The points $c_1$ and $c_2$ are distinct, since subtracting the line
$L_1$ from the planar linear system $\mathcal W$ does not force
subtracting the whole triangle $L_1+L_2+L_3$ to the system. By
subtracting $C_{ji}$ from $\bar {\fL}_{ij}$, the residual linear system is
a linear 
system of rational curves with self--intersection 1, mapping $W_{ij}$
birationally to the plane. Indeed, this residual linear system
corresponds to the residual linear system of $L_1$ with respect to
$\mathcal W$, which is the linear system of plane cubics, with a
double point at $a_1$, two simple  base points  infinitely near to $a_1$
not on $L_2$ and $L_3$, two base points at $a_2$ and $a_3$, and this
is a homaloidal system. This shows that $c_1$ is a triple point of
$W$ and that $\bar \varphi$ is birational. The same for $c_2$. Finally 
$\bar \varphi$ maps (the proper transforms of) $D_{ij}^k$ and
$D_{ij}^h$ both to the unique line $D$ containing $c_1$ and $c_2$. 
\end{proof} 

\begin{remark}
The subpencil of $\fL_{ij}$ corresponding to planes in $\P^3$ that
contain the line $D$ corresponds to the subpencil of curves in
$\mathcal W$ with the triangle $L_1+L_2+L_3$ as its fixed part, plus
the pencil of lines through $a_1$.  
In this subpencil we have two special curves, namely $L_1+2L_2+L_3$
and $L_1+L_2+2L_3$. This  shows that the tangent cone  
to $W$ at the general point of $D$ is fixed, formed by two planes.
\end{remark}

\begin{remark}
The image of $\bar{X}$ via the complete linear system
$|\L(-M_{ij})|$ provides a model $f':S'\to \Delta$ of the initial family
$f:S\to \Delta$, such that the corresponding flat limit of $S'_t\cong
S_{t^6}$ with $t\neq 0$, is  $S'_0=W$ the image of $W_{ij}$ via
$\varphi$. 
The map $\bar{X}_0 \to S'_0$ contracts the chain \eqref{chain_kl} to
the double line of $W$, 
and the two connected components of $\bar{X}_0-W_{ij}$ minus the chain
\eqref{chain_kl} (cf. Figure \ref{f:granchio}) to the two triple
points of $W$ respectively.
\end{remark}

\begin{corollary}\label{cor:excep} The exceptional divisor 
$\tilde{\fL}_{ij}$ of ${\fP}'' \to {\fP}'$ is naturally isomorphic to
the blow--up of the complete linear system 
$\bar{\fL}_{ij} \cong \left|\O_{W}(1)\right|$ 
along its subpencil corresponding to planes in $\P^3$ containing the
line $D$.
\end{corollary}

\begin{proof} 
This is a  reformulation of the description of $\fP'' \to \fP'$
(cf. Step \eqref{Tetra:edges} in \S\ref{s:4planes-limlin} above),
taking into account Propositions \ref{prop:descrL_ij} and
\ref{p:W-quartic}.
\end{proof}

The divisor $\tilde{\fL}_{ij} \subset \fP_0''$ is a 
$\P(\O_{\P^ 1}(1)^ {\oplus 2}\oplus \O_{\P^ 1})$, and its structure of
$\P^2$--bundle over $\P^1$ is the minimal resolution of
indeterminacies of the rational map $\fL_{ij}\dasharrow
|\O_D(1)|$, which
sends a general divisor $C\in  \fL_{ij}$ to its intersection point
with $D$.
The next Proposition provides an identification of the general
fibres of $\tilde{\fL}_{ij}$ over $|\O_D(1)|=\P^1$ as certain linear
systems.

\begin{proposition}
\label{p:projectW}
The projection of $W$ from a general point of $D$ is a double
cover of the plane, branched over a sextic $B$ which is the union 
\[
B=B_0+B_1+B_2
\]
of a quartic $B_0$ with a node $p$, and of its tangent cone $B_1+B_2$
at $p$, such that the two branches of $B_0$ at $p$ both have a
flex there 
(see Figure \ref{f:diramazione}; the intersection $B_i\cap B_0$ is concentrated at
the double point $p$, for $1\leqslant i\leqslant 2$).
\end{proposition}

\begin{figure}
\begin{center}
\includegraphics[width=14cm]{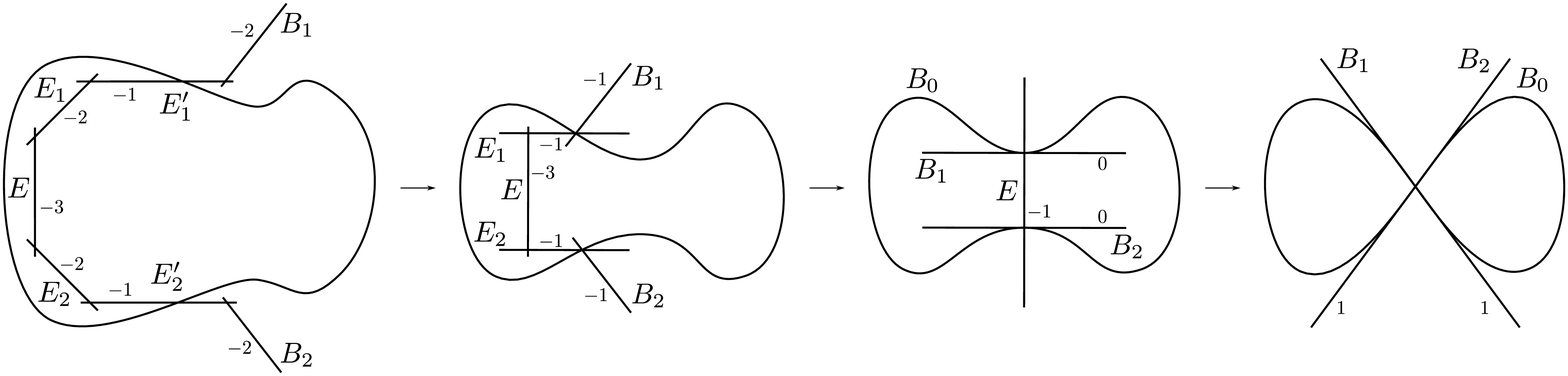}
\end{center}
\caption{Desingularization of the branch curve of the projection of
  $W_{ij}$} 
\label{f:diramazione}
\end{figure}

\begin{proof}
Let us consider a double cover of the plane as in the
statement. 
It is singular. Following \cite[\S 4]{calabri-ferraro}, 
we may obtain a resolution of singularities as a double cover of a
blown--up plane with non--singular branch curve. 
We will then observe that it identifies with
$\bar{W}$ blown--up at two general \emph{conjugate} points on
$D_{ij}^k$ and $D_{ij}^h$ respectively
(here \emph{conjugate} means that the two points are mapped to the
same point $x$ of $D$ by $\bar \varphi$).
We will denote by  $\tilde{W}$ the surface $\bar W$
blown--up at two such points, and by $I_x',I_x''$ the two exceptional
divisors. 

First note that our double plane is rational, because it
has a pencil of rational curves, namely the pull-back of
the pencil of lines passing through  $p$ (eventually this will correspond to the pencil of conics cut out on $W$ by the planes through $D$).

In order  to resolve the singularities of the branch curve (see
Figure~\ref{f:diramazione}), 
we first blow--up $p$, pull--back the double cover
and normalize it. Since $p$ has multiplicity $4$, which is even, 
the  exceptional divisor $E$ of the blow--up does not belong to the branch
curve of the new double cover, which is the proper transform  $B$ (still denoted by $B$ according to our general convention).
Next we blow--up the two double points of $B$ which lie on $E$, and
repeat the process.  Again, the two
exceptional divisors $E_1,E_2$ do not belong to the  branch
curve. 
Finally we blow--up the two double points of $B$ (which lie one on
$E_1$ one on $E_2$, off $E$), and repeat the process.
Once more, the two exceptional divisors $E'_1,E'_2$ do not belong to
the  branch curve which is the union of $B_0$, $B_1$ and $B_2$
(which denote here the proper transforms 
of the curves with the same names on the plane). 
This curve is smooth, so the corresponding double
cover is smooth.

The final double cover has the following configuration of
negative curves:
$B_1$ (resp. $B_2$) is contained in the branch divisor, so over it we
find a $(-1)$-curve; 
$E_1'$ (resp. $E_2'$) meets the branch divisor at two points, so its
pull-back is a $(-2)$-curve;
$E_1$ (resp. $E_2$) does not meet the branch divisor, so its pull-back
is the sum of two disjoint $(-2)$-curves;
similarly, the pull-back of $E$ is the sum of two disjoint
$(-3)$-curves.
In addition, there are four lines through $p$ tangent to
$B_0$  and distinct from $B_1$ and $B_2$. After the
resolution, they are curves with self-intersection $0$ and meet
the branch divisor at exactly one point with multiplicity $2$. The
pull-back of any such a curve is the transverse union of two
$(-1)$-curves, each of which meets transversely one component of the
pull-back of $E$.

This configuration is precisely the one we have on $\tilde{W}$,
after the contraction of the four $(-1)$-curves
$G_{ji}^k$, $G_{ij}^k$, $G_{ji}^h$, and $G_{ij}^h$.
Moreover, the pull-back of the line class of $\P^2$ is the pull--back
to $\tilde W$ of 
$\fL_{ij}(-(I_x'+I_x''))$. 
\end{proof}

\begin{corollary}\label{cor:projectW}
In the general fibre of the generic $\P^ 2$ bundle structure of
$\tilde \fL_{ij}$, the Severi variety of $1$--nodal
(resp. $2$--nodal) irreducible curves is an irreducible curve of degree
$10$ (resp. the union of $16$ distinct points).
\end{corollary}

\begin{proof}
This follows from the fact that the above mentioned Severi
varieties are respectively the dual curve $\check B_0$ of a plane
quartic as in Proposition~\ref{p:projectW},
and the set of ordinary double points of $\check B_0$.
One computes the degrees using Pl\"ucker formulae.
\end{proof}

\subsection{The linear systems $\fL^k$}
\label{s:cubics}

Here we study  some  geometric properties of the
linear systems $\fL^ k$ appearing in the third step of   \S \ref {s:4planes-limlin}.

Consider a triangle $L_1,L_2,L_3$ in $\P^ 2$, with vertices
$a_1,a_2,a_3$, where $a_1$ is opposite to $L_1$, etc. 
Consider the linear system $\mathcal T$
of cubics through $a_1,a_2,a_3$ and tangent there to $L_3, L_1, L_2$
respectively. By Proposition \ref{prop:descr2}, there is a birational
transformation of $T^ k$ to the plane (see Figure \ref {f:WT}) mapping
$\fL^ k$  to $\mathcal T$.  We consider the rational map $\phi_k: T^k
\dasharrow  \P^ 3$ (or simply $\phi$) determined by the linear system
${\fL}^k$, or,  alternatively,  the rational map, with the same image
$T$  (up to projective transformations), determined by the planar
linear system $\mathcal T$. The usual notation is $\{1,\ldots,
4\}=\{i,j,s,k\}$.

\begin{proposition}
\label{p:cubics}
The map $\phi: T^ k\to T\subset \P^ 3$ is a birational morphism, and
$T$ is a cubic surface  with three double points of type $A_2$
as its only singularities.  The minimal resolution of $T$ is the
blow--down of $T^ k$ contracting the $(-1)$--curves 
$D^k_{ij}, D^k_{is},D^k_{js}$.  
This cubic contains exactly three lines, each of them containing two
of the double points.
\end{proposition}

\begin{proof} The linear system $\mathcal T$ is a system of plane cubics with six simple base points,
whose general member is clearly irreducible. This  implies that 
$\phi: T^ k\to T\subset \P^ 3$ is a birational morphism and $T$ is a cubic
surface.  The linear system $\fL^ k$ contracts the three chains of rational curves
\[
C_1=F_k^{js}+2D^k_{sj}+F_k^{sj}, 
\quad C_2= F_k^{si}+2D^k_{is}+F_k^{is},
\quad C_3=F_k^{ij}+2D^k_{ij}+F_k^{ji},\] 
which map in the plane to the sides of the triangle $L_1,L_2,L_3$.
By contracting the  $(-1)$--curves $D^k_{ij}$, $D^k_{is}$, $D^k_{js}$,
the three curves $C_1,C_2,C_3$ are mapped to three  $(-2)$--cycles 
contracted by $\phi$ to double points of type $A_2$.

The rest follows from the classification of cubic
hypersurfaces in $\P^3$ (see, e.g., \cite{bruce-wall}). The three
lines on $T$ are the images via $\phi$ of the three 
exceptional divisors $\Gamma^ k_i, \Gamma^ k_j, \Gamma^ k_s,$.
\end{proof}

\begin{remark}
We now see that the image of $\bar{X}$ by the complete linear system
$|\L(-M^k)|$ provides a model $f':S'\to \Delta$ of the initial family
$f:S\to \Delta$, such that the corresponding flat limit of $S'_t\cong
S_{t^6}$ with $t\neq 0$, is  $S'_0=T+P$,
where $T$ is the image of $T^k$ via $\phi$,
and $P$ is the plane in $\P^3$ through the three lines contained in
$T$, image of $P_k$ by the map associated to $\fL^k$. The three other
faces of the initial tetrahedron $S_0$ are contracted to the three
lines in $T$ respectively.
\end{remark}

\begin{proposition}
\label{p:cubics-gauss}
The dual surface $\check{T}\subset \check{\P}^3$ to $T$ is itself a
cubic hypersurface with three double points of type $A_2$ as its only
singularities. 
Indeed, the Gauss map $\gamma_T$ fits into the commutative diagram
\[\xymatrix@=15pt{
& T^k \ar[dl]_{\phi} \ar[dr]^{\check{\phi}} & \\
T \ar@{-->}[rr]_{\gamma_T} && \check{T}
}\]
where $\check{\phi}$ is the morphism associated to the 
linear system
\begin{equation*}
\bigl| 3H-(F_k^{sj}+\Gamma_s^k+2D^k_{sj})
-(F_k^{ji}+\Gamma_j^k+2D_{ij}^k)
-(F_k^{is}+\Gamma_i^k+2D_{is}^k) \bigr|,
\end{equation*}
which is mapped to the linear system $\mathcal T'$  
of cubics  through $a_1,a_2,a_3$ and tangent there to $L_2, L_3, L_1$
respectively,
by the birational map $T^ k\dasharrow \P^2$ identifying $\fL^k$ with
$\mathcal T$.
\end{proposition}

\begin{proof}
The dual hypersurface $\check{T}$ has degree $3$ by Proposition
\ref{p:deg-dual-sing}. Let $p$ be a double point of $T$. The tangent cone to $T$ at $p$
is a rank 2 quadric, with vertex a line $L_p$. A local computation shows that the limits of all tangent planes to $T$ at smooth points 
tending to $p$ are planes through $L_p$. This means that $\gamma_T$ is not well defined on the minimal
resolution of $T$, which is the blow--down of 
$T^ k$ contracting the $(-1)$--curves $D^k_{ij}, D^k_{is},D^k_{js}$,
its indeterminacy points being exactly
the three points images of these curves. The same local computation also shows that $\gamma_T$ is well defined on $T^ k$, hence $\gamma_T$ fits in the diagram as stated. 

In $\mathcal T$ there are the three curves $2L_1+L_3, 2L_2+L_1,
2L_3+L_2$, which implies that  
for any given line $\ell \subset T$ there is a plane $\Pi_\ell$ in
$\P^3$ tangent to $T$ at the general point of  $\ell$
(actually one has $\Pi_\ell \cap T = 3\ell$).
Then $\gamma_T$ contracts  each  of the three lines contained in $T$
to three different points, 
equivalently $\check{\phi}$ contracts to three different points the
three curves $\Gamma^ k_i, \Gamma^ k_j, \Gamma^ k_s$.
Being $\check T$ a (weak) Del Pezzo surface, this implies that 
$\check{\phi}$ must contract the three
chains of rational curves
$F_k^{sj}+2\Gamma_s^k+F_k^{si}$,
$F_k^{is}+2\Gamma_i^k+F_k^{ij}$,
and $F_k^{ji}+2\Gamma^k_j+F_j^{js}$,
because they have $0$ intersection with
the anticanonical system, and the rest of the assertion follows.
\end{proof}

Recalling the description of $\fP'''\to \fP$, one can 
realize $\tilde{\fL}^k$ as a birational modification of 
$\fL^k\cong |\O_T(1)|$:
first blow--up the point corresponding to the plane containing the
three lines of $T$,
then blow--up the strict transforms of the three lines in 
$|\O_T(1)|$ corresponding to the three pencils of planes respectively
containing the three lines of $T$.
Notice that $\tilde{\fL}^k$ has a structure of $\P^1$--bundle on the
blow--up of  $\P^ 2$ at three non--coplanar points, as required.

Alternatively, we have in $\mathcal T$ the four curves $C_0=L_1+L_2+L_3,
C_1=2L_1+L_3, C_2=2L_2+L_1,
C_3=2L_3+L_2$, corresponding to four independent
points $c_0\ldots, c_3$ of $\mathcal T$.  
Then $\tilde{\fL}^k$ is the blow--up of $\mathcal T$ at $c_0$,
further blown--up along the proper transforms of the lines
$\left<c_0,c_1\right>$, $\left<c_0,c_2\right>$, and
$\left<c_0,c_3\right>$.
Via the map $\tilde{\fL}^k \to \mathcal{T}$, the projection of the
$\P^ 1$--bundle structure corresponds to the projection of
$\mathcal T$ from $c_0$ to the plane spanned by $c_1,c_2,c_3$. 

This will be interpreted using Lemma \ref{l:tg-id} in \S\ref{s:good}
below.

\subsection{The limit linear system, II: description}
\label{s:good}

We are now ready to prove:
\begin{proposition}
\label{prop:limlin-tetra}
The limit linear system of $|\L_t|=|\O_{\bar X_t}(1)|$ as $t\in
\Delta^*$ tends to $0$ is $\fP'''_0$.
\end{proposition}

\begin{proof}
The identification of $\fP'''$ as $\Hilb(\L)$ will follow from the
fact that every point in $\fP'''_0$ corresponds to a curve in
$\bar{X}_0$ (see Lemma \ref{l:lim-lin}).
Having the results of \S\S\ref{s:4planes-limlin}---\ref{s:cubics} at
hand, we are thus left with the task of describing how the various
components of the limit linear system intersect each other.
We carry this out by analyzing, with Lemma \ref{l:tg-id}, the
birational modifications operated on 
the components $\fP_0$, $\tilde\fL_i$, $\tilde\fL_{ij}$, and
$\tilde\fL^k$, during the various steps of the construction of
$\fP'''$ (see \S\ref{s:4planes-limlin}).

\medskip
\eqref{Tetra:faces}
In $\fP'_0$, the strict transform of $\fP_0$
(which we shall go on calling $\fP_0$, according to the
conventions set in \S\ref{conv}) is the blow up of 
$|\L_0| \cong |\O_{S_0}(1)|$ at the
four points corresponding to the faces of $S_0$.
For each $i\in\{1,\ldots,4\}$, the corresponding
exceptional plane is the intersection $\fP_0\cap \tilde\fL_i$,
and it identifies with the subsystem of $\fL_i$ consisting of curves
\[
L+\sum_{\substack{j\neq i \\ \{i,j,k,h\}=\{1,\ldots,4\}}}
(L_{ij}+G_i^h+G_i^k),
\quad L\in |\O_{P_i}(H)|,
\]
together with six rational tails respectively joining $\plm{E}_{ji}$
to $\plm{z}_{ji}$, $j\neq i$.

\medskip
\eqref{Tetra:edges}
For each $\{i\neq j\}\subset \{1,\ldots,4\}$, the intersection
$\fP_0\cap \tilde\fL_{ij} \subset \fP_0''$ identifies as the
exceptional $\P^1\times \P^1$ of both the blow--up of $\fP_0\subset
\fP_0'$ along the line $\ell_{ij}$, and the blow--up $\tilde\fL_{ij}\to
\bar{\fL}_{ij}$ described in Corollary \ref{cor:excep}.
As a consequence, it parametrizes the curves
\begin{equation}
\label{D+conic}
C+\Phi+D_{ij}^k+D_{ij}^h+C_{ji}+C_{ij},
\quad \{i,j,k,h\}=\{1,\ldots,4\},\ 
C_{ji}\ \text{and}\ C_{ij}\ \text{as in \eqref{C_ij}},
\end{equation}
where $C$ is a chain in $\fN^{kh}$, and $\Phi \in |F|_{W_{ij}}$ is the
proper transform by $\varphi_{ij}$ of a conic through the two
triple points of $W$ (cf. Proposition \ref{p:W-quartic}),
together with four rational tails respectively
joining $\plm{E}_{ij}$ and $\plm{E}_{ji}$ to $\plm{w}_{ij}$ and
$\plm{w}_{ji}$. 
The two components $C$ and $\Phi$ are independent one from another,
and respectively move in a $1$--dimensional linear system.

The intersection $\tilde\fL_{ij}\cap \tilde\fL_i \subset \fP_0''$ is a
$\P^2$. In $\tilde\fL_{ij}$, it identifies as the proper transform via
$\tilde{\fL}_{ij}\to \bar{\fL}_{ij}$ of the linear system of
curves 
\begin{equation}
\label{W-triple-point}
C_{ij}+C,
\quad C \in |\L_0(-M_{ij})\otimes \O_{W_{ij}}(-C_{ij})|,\quad
\text{and}\ C_{ij}\ \text{as in \eqref{C_ij}},
\end{equation}
while in $\tilde\fL_i$ it is the exceptional divisor of the blow--up
of $\tilde\fL_i \subset \fP_0'$ at the point corresponding to the
curve 
\[
2(L_{ij}+G_i^h+G_i^k)
+(L_{ih}+G_i^j+G_i^k)
+(L_{ik}+G_i^j+G_i^h),
\quad \{i,j,k,h\}=\{1,\ldots,4\}.
\]
It follows that it parametrizes sums of a curve as in
\eqref{W-triple-point}, 
plus the special member of $\fN^{kh}$ consisting of double curves of 
$\bar{X}_0$ 
and joining the two points
$D^k_{ij}\cap F_{ij}^k$ and $D^h_{ij}\cap F_{ij}^h$.

\medskip
\eqref{Tetra:vertices}
For each $k\in \{1,\ldots,4\}$, the intersection 
$\Pi_k=\tilde\fL^k\cap \fP_0$ is a $\P^2$ blown up at three non
colinear points.
Seen in $\fP_0$, it identifies as the blow--up of the web of planes in
$|\L_0|\cong |\O_{S_0}(1)|$ passing through the vertex $k$ of $S_0$,
at the three points 
corresponding to the faces of $S_0$ containing this very vertex.
In $\tilde\fL^k$ on the other hand, it is the strict transform of the
exceptional $\P^2$ of the blow--up $\tilde{\mathcal{T}}\to
\mathcal{T}\cong \fL^k$ at the point $[a+b+c]$.
It therefore parametrizes the curves
\begin{equation}
\label{antican-triangle+line}
L+\sum_{i\neq k} \Bigl( \Gamma_i^k+
\sum_{j\not\in \{i,k\}}(F_k^{ij}+D_{ij}^k) \Bigr),
\end{equation}
where $L$ is a line in $P_k$, together with three rational tails
joining respectively $L\cap L_{ki}$ to $\Gamma_i^k$, $i\neq k$.

For $i\neq k$, $\tilde\fL^k\cap \tilde\fL_i$ is a $\P^1\times
\P^1$, identified as the exceptional divisor of both the blow--up
$\tilde\fL^k\to \tilde{\mathcal{T}}$ along the strict transform of the
line parametrizing planes in $\mathcal{T} \cong |\O_T(1)|$ containing
the line
$\phi^k(\Gamma_i^k)$, and the blow--up of $\tilde\fL_i\subset \fP_0''$
along the strict transform of the line parametrizing curves 
\begin{equation}
\label{triple-line}
L+G_i^k+\sum_{\substack{j\neq i \\ \{i,j,k,h\}=\{1,\ldots,4\}}}
(L_{ij}+G_i^h+G_i^k),
\quad L\in |\O_{P_i}(H-G_i^k)|.
\end{equation}
It therefore parametrizes sums of
\begin{equation}
\label{fL^k_i}
\Phi
+\Bigl(\Gamma_i^k+ \sum_{j\not\in
  \{i,k\}}\bigl(F_k^{ij}+2D_{ij}^k+F_k^{ji}\bigr) 
+C\Bigr)
\end{equation}
(where $\Phi\in \fN^k_i$, and the second summand is a member of 
$\fL^k\bigr|_{T^k}$),
plus the fixed part
$
L_{ki}+\p{E}_{ki}+\m{E}_{ki}+\sum_{j\not\in \{i,k\}}(G_k^j+\Phi_j)
$,
where $\Phi_j$ is the special member of $\fN_j^k$ consisting of double
curves of $\bar{X}_0$ and joining the two points $G_k^j\cap L_{k\bj}$
on $P_k$
and $F_k^{\bj j}\cap \Gamma^k_{\bj}$ on $T^k$,
for each $j\not\in \{i,k\}$, with $\bj$ such that
$\{i,k,j,\bj\}=\{1,\ldots,4\}$. 
The two curves $\Phi$ and $C$ are independent one from another, and
respectively move in a $1$--dimensional linear system.

For each $j\not\in \{k,i\}$, $\tilde\fL^k\cap \tilde\fL_{ij}$ is an
$\F_1$, and identifies as the blow--up of the plane in $\fL^k$
corresponding to divisors in $|\O_T(1)|$ passing through the double
point $\phi^k(\Gamma_i^k)\cap \phi^k(\Gamma_j^k)$, at the point
$\bigl[\sum_{i\neq k}\phi^k(\Gamma_i^k)\bigr]$;
it also identifies as the exceptional divisor of the blow--up of
$\tilde\fL_{ij}\subset \fP_0''$ along the $\P^1$ corresponding to the
curves as in \eqref{D+conic}, with $\Phi$ the only member of
$|F|_{W_{ij}}$ containing $D_{ij}^k$.
We only need to identify the curves parametrized by the exceptional
curve of this $\F_1$; they are as in \eqref{antican-triangle+line},
with $L$ corresponding to a line in the pencil $|\O_{P_k}(H-G_k^s)|$,
$s\not\in \{i,j,k\}$.

\smallskip
In conclusion, $\fP'''$ is an irreducible Zariski closed
subset of the relative Hilbert scheme of $\bar X$ over $\Delta$,
and this proves the assertion.
\end{proof}

\subsection{The limit Severi varieties}
\label{s:tetra-concl}

We shall now identify the regular parts of the limit Severi varieties
$\fV_{1,\delta}(\bar X)=
\fV_{\delta}(\bar X,\L)$ for $1\leqslant \delta\leqslant 3$ (see Definition \ref {def:reg-comps}).
To formulate the subsequent statements, we use Notation \ref {not:1} and
the notion of 
$\bf n$--degree introduced in \S\ref{s:enumeration}.

We will be interested in those $\bf n$ that correspond to a
choice of $3-\delta$ general base points on the faces $P_i$ of
$S_0$, with $1\le i\le 4$.  These choices can be 
identified  with $4$--tuples
${\bf n}=(n_1,n_2,n_3,n_4)\in \N^4$ with
$\vert {\bf n}\vert=3-\delta$
(by choosing $n_i$ general points on $P_i$).
The vector $\bf n$ is non--zero only if $1\leqslant \delta\leqslant 2$.
For $\delta=1$ (resp. for $\delta=2$), to give $\bf n$ is equivalent
to give two indices $i,j\in \{1,\ldots,4\}^2$ (resp.  an $i\in
\{1,\ldots,4\}$): 
we let ${\bf n}_{i,j}$ (resp. ${\bf n}_i$)
be the $4$--tuple corresponding to the choice of  
general base points on $P_i$ and $P_{j}$ respectively if $i\neq j$,
and of two general base points on $P_i$ if $i=j$ 
(resp. a general base point on $P_i$).

\begin{proposition} [Limits of $1$-nodal curves]
\label{p:1-4planes}
The regular components of the limit Severi variety \linebreak
 $\fV_{1,1}(\bar X)$
are the following
(they all appear with multiplicity $1$):\\
\begin{inparaenum}[\normalfont (i)]
\item\label{4planes-sing1}   the proper transforms of the 24 planes 
$V(E)\subset \vert \O_{S_0}(1)\vert$,
where $E$ is any one of the $(-1)$--curves $\plm{E}_{ij}$,
for $1\leqslant i,j \leqslant 4$ and $i \neq j$. 
The ${\bf n}_{hk}$--degree is 1 if $h\neq k$;
when $h=k$, it is $1$ if $h\not\in \{i,j\}$, and $0$ otherwise; \\ 
\item\label{4planes-triple1} the proper transforms of  the four degree
  $3$ surfaces $V(M^k,\delta_{T^k}=1)\subset {\fL}^k$,
 $1\leqslant k\leqslant 4$. 
The ${\bf n}_{ij}$--degree is $3$ if $i\neq j$; when $i=j$, it is $3$
if $k=i$, and $0$ otherwise; \\ 
\item  \label{4planes-planes} the proper transforms of the four degree
$21$ surfaces $V(M_i,\delta_{P_i}=1)\subset {\fL}_i$, 
$1\leqslant  i\leqslant 4$. 
The ${\bf n}_{hk}$--degree is $21$ if $h=k=i$, and 0
otherwise; \\ 
\item  \label{4planesij} the proper transforms of  the six 
surfaces in $V(M_{ij},\delta_{W_{ij}}=1) \subset {\fL}_{ij}$,
$1\leqslant i<j \leqslant 4$. 
They have ${\bf n}_{hk}$--degree $0$ for every
$h,k \in \{1,\ldots,4\}^2$. 
\end{inparaenum}
\end{proposition}

\begin{proof}
This follows from \eqref{reg-Severi},
and from Propositions \ref {p:cubics-gauss}
and \ref{p:1-triangle}.
Proposition \ref{p:1-triangle}
tells us that $V(M_i,\delta_{P_i}=1)$ has degree at least $21$ in
${\fL}_i$ for $1\leqslant  i\leqslant 4$;
the computations in Remark \ref{rem:dual} \eqref{tetra1-count} below yield
that it cannot be strictly larger than $21$
(see also the proof of Corollary \ref{cor:thmDi}),
which proves Theorem \ref{T:triangle} for $\delta=1$.
The ${\bf n}_{hk}$--degree computation is  straightforward. 
\end{proof}

\begin{remark}\label{rem:dual} 
\begin{inparaenum}[\bf (a)]
\item \label{tetra1-count}
The degree of the dual of a smooth surface of degree 4 in $\P^ 3$ is 36. 
It is instructive to identify, in the above setting,  the $36$ limiting curves passing through
two general  points on the proper transform of $S_0$ in $\bar{X}$.
This requires the ${\bf n}_{hk}$--degree information in Proposition
\ref {p:1-4planes}. 
If we choose the two points on different planes,  24 of the 36 limiting
curves through them  
come from \eqref {4planes-sing1}, and 4 more, each with multiplicity
3, come from \eqref {4planes-triple1}. 
If the two points are chosen in the same plane, then we have 12
contributions from  \eqref {4planes-sing1}, only one contribution, with
multiplicity 3, from \eqref {4planes-triple1}, and 21 more
contributions form \eqref {4planes-planes}. No contribution ever comes
from \eqref {4planesij} if we choose points on the faces of the
tetrahedron.

\item \label{r:2-tetra}
We have here an illustration of Remark \ref{r:other-lim}:
the components $V(M_i,\delta_{P_i}=1)$ are mapped to
points in $\vert \O_{S_0}(1) \vert$, hence they do not appear in
the crude limit $\cru \fV_{1,1}(S)$ (see Corollary \ref{cor:thmBi}
below);
they are however visible in the crude limit Severi variety of the
degeneration to the quartic monoid corresponding to the face $P_i$.
In a similar fashion, to see the component
$V(M_{ij},\delta_{W_{ij}}=1)$ one should consider the flat limit
of the $S_t$, $t\in \Delta^*$, given by the surface $W$ described in
Proposition \ref{p:W-quartic}.
\end{inparaenum}
\end{remark}

\begin {corollary} [Theorem \ref{T:tetrahedron} for $\delta=1$]
 \label{cor:thmBi}
Consider a family $f:S\to \Delta$ of general quartic surfaces
in $\P^ 3$ degenerating to a tetrahedron $S_0$. The singularities of
the total space $S$ consist in 24 ordinary double points,  
four on each edge of $S_0$ (see \S \ref {s:setting}).
It is $1$--well behaved, with good model $\varpi:\bar X\to \Delta$.
The limit in $|\O_{S_0}(1)|$ of the dual surfaces
$\check S_t$, $t\in \Delta^*$ (which is the crude limit Severi variety
$\cru \fV_{1,1}(S)$), consists in the union of the $24$ webs of planes
passing through a singular point of $S$, and 
of the $4$ webs of planes passing through a vertex of $S_0$, each
counted with multiplicity $3$.
\end{corollary}

\begin{proof} The only components of $\reg\fV_{1,1}(\bar X)$ which are not
  contracted to lower dimensional varieties by the morphism 
$\fP'''\to \fP$
are the ones in  \eqref {4planes-sing1} and in \eqref
{4planes-triple1} of Proposition \ref{p:1-4planes}. 
Their push--forward in $\fP_0\cong |\O_{S_0}(1)|$ has total degree
$36$. The assertion follows.
\end{proof}

\begin{corollary}
\label{cor:thmDi}
Consider a family $f':S'\to \Delta$ of general quartic surfaces
in $\P^ 3$, degenerating to a monoid quartic surface $Y$ with tangent
cone at its triple point $p$ consisting of a triple of independent planes
(see Remark \ref{rem:monoid}).
This family  is $1$--well behaved, with good model $\varpi: \bar X\to \Delta$.
The crude limit Severi variety
$\cru \fV_{1,1}(S')$ consists in the surface $\check Y$
(which has degree $21$), 
plus the plane $\check p$ counted with multiplicity $15$.
\end{corollary}

\begin{proof}
We have a morphism $\fP'''\to 
\P\bigl(\varpi_*\bigl(\L(-M_i)\bigr)\bigr)\cong
\P\bigl(f'_*\bigl(\O_{S'}(1)\bigr)\bigr)$.
The push--forward by this map of the regular components of 
$\fV_{1,1}(\bar X)$
are $\check Y$ for $V(M_i,\delta_{P_i}=1)$,
$3\cdot \check p$ for $V(M^i,\delta_{T^i}=1)$,
$\check p$ for each of the twelve $V(E)$ corresponding to a
$(-1)$--curve $\plm{E}_{hk}$ with $i\in \{h,k\}$,
and $0$ otherwise.
The degree of $V(M_i,\delta_{P_i}=1)$ in $\fL_i$ is at least $21$ by
Proposition \ref{p:1-triangle}, so the total degree of the
push--forward in $|\O_{S_0'}(1)|$ of the regular components of
$\fV_{1,1}(\bar X)$ is at least $36$. The assertion follows.
\end{proof}

\begin{proposition} [Limits of $2$-nodal curves]
\label{p:2-4planes}
The regular components of the limit Severi variety \linebreak
$\fV_{1,2}(\bar X)$
are the following
(they all appear with multiplicity $1$):\\
\begin{inparaenum}[\normalfont (i)]
\item\label{4planes-sing2}
$V(E,E')$ for each set of two curves 
$E,E'\in \{\plm{E}_{ij}, 1\leqslant i,j\leqslant 4, i\neq j\}$
that do not meet the same edge of the tetrahedron $S_0$. 
The ${\bf n}_h$--degree is $1$ if $P_h\subset S_0$ does not contain
the two edges
met by $E,E'$, and $0$ otherwise;\\
\item\label{4planes-vertex2}
$V(M^k,\delta_{T^k}=1,E)$ for $k=1,\ldots,4$ and 
$E\in \bigl\{\plm{E}_{ij}, 1\leqslant i,j\leqslant 4, i\neq j,
k\in\{i,j\}\bigr\}$, which is a degree $3$ curve in $\fL^k$.
The ${\bf n}_h$--degree is $3$ if $P_h$ does not contain both the edge 
met by $E$ and the vertex corresponding to $T^ k$, it is $0$ otherwise;\\
\item\label{4planes-edge2}
$V(M_{ij},\delta_{W_{ij}}=2)$ for $1\leqslant i<j \leqslant 4$, 
which has ${\bf n}_{h}$-degree $16$ for 
$h\not\in \{i,j\}$, and $0$ otherwise;\\
\item \label{4planes-0}
$V(M_i,\delta_{P_i}=2)$ for $1\leqslant i\leqslant  4$, which which
has ${\bf n}_{h}$-degree $132$ for  
$h=i$, and $0$ otherwise;\\
\item \label{4planes-1}
$V(M_{ij},\delta_{W_{ij}}=1,E)$ for $1\leqslant i<j \leqslant 4$,
and 
$E\in \bigl\{\plm{E}_{\bar{\imath}\bar{\jmath}},
\{\bar{\imath},\bar{\jmath}\}\cup \{i,j\}=\{1,\ldots,4\} \bigr\}$,
which is a curve of ${\bf n}_{h}$-degree $0$ for $1\leqslant
h\leqslant 4$.
\end{inparaenum}
\end{proposition}

\begin{proof}
It goes as the proof of Proposition
\ref{p:1-4planes}. 
Again, Proposition \ref{p:2-triangle} asserts that
$V(M_i,\delta_{P_i}=2)$ has degree at least $132$ in $\fL_i$, but it
follows from the computations in
Remark \ref{rem:480} \eqref{tetra2-count} below that it is 
exactly $132$, which  proves Theorem \ref{T:triangle} for
$\delta=2$.
\end{proof}

\begin{remark}\label{rem:480}
\begin{inparaenum}[\bf (a)]
\item \label{tetra2-count}
The degree of the Severi variety $V_2(\Sigma, \O_\Sigma(1))$
for a general quartic surface $\Sigma$ is $480$ (see Proposition \ref {p:deg-dual}). 
Hence if we fix a general point $x$ on one of the components $P_h$ of $S_0$ we should be able to see the $480$ 
points of the limit Severi variety $\fV_{1,2}$ through $x$. The ${\bf n}_h$--degree information in Proposition 
\ref {p:2-4planes} tells us this.

For each choice of two distinct edges of $S_0$ spanning a plane
distinct from $P_h$, and of two 
$(-1)$-curves $E$ and $E'$ meeting these edges, we have a curve
containing $x$ in each of the  
items of type \eqref {4planes-sing2} . This amounts to a total of $192$ 
such curves.

For each choice of a vertex and an edge of $S_0$, such that
they span a plane distinct from $P_{h}$,
there are $3$ curves containing $x$ in each of the four corresponding
items \eqref {4planes-vertex2}.
This amounts to a total of $108$ such curves.

For each choice of an edge of $S_0$ not contained in $P_{h}$, there
are $16$  curves containing $x$  in the corresponding item \eqref {4planes-0}. This gives a
contribution of $48$ curves.

Finally, there are $132$ plane quartics containing $x$ in the item \eqref  {4planes-0} for $i=h$.
Adding up, one finds  the right number $480$.

\item
Considerations similar to the ones in Remark \ref{rem:dual}
\eqref{r:2-tetra} 
 could be made here, but we do not dwell on this.
\end{inparaenum}
\end{remark}

\begin {corollary} [Theorem\,\ref{T:tetrahedron} for $\delta=2$]  
\label{cor:thmBii} 
Same setting as in Corollary \ref {cor:thmBi}. The family $f:S\to
\Delta$ is $2$--well behaved, with good model $\varpi:\bar X\to
\Delta$. 
The crude limit Severi variety $\cru \fV_{1,2}(S)$ consists of the image
in $\vert \O_{S_0}(1)\vert$  
of:\\
\begin{inparaenum}[\normalfont (i)]
\item the $240$ components in  \eqref{4planes-sing2} of Proposition \ref {p:2-4planes}, which map to as many lines in  $\vert \O_{S_0}(1)\vert$;\\
\item the $48$ components in \eqref {4planes-vertex2} of Proposition \ref {p:2-4planes}, each mapping $3:1$ to as many lines in  $\vert \O_{S_0}(1)\vert$;\\
\item the $6$ components in \eqref{4planes-edge2} of Proposition \ref
  {p:2-4planes}, respectively mapping $16:1$ to the dual lines of the
  edges of $S_0$.
\end{inparaenum}
\end{corollary}

\begin{proof} The components in question are the only ones  not
  contracted to points by the morphism $\fP'''_0\to
  \vert\O_{S_0}(1)\vert$, and their push--forward sum up to a degree
  $480$ curve.
 \end{proof}

\begin{corollary}
\label{cor:thmDii}
Same setting as in Corollary \ref{cor:thmDi}; the family $f':S'\to
\Delta$ is $2$--well behaved, with good model $\varpi: \bar X\to \Delta$.
The crude limit Severi variety
$\cru \fV_{1,2}(S')$ consists of the ordinary double curve of the surface
$\check Y$, which has degree $132$,
plus a sum (with multiplicities) of lines contained in the
dual plane $\check p$ of the vertex of $Y$.
\end{corollary}

\begin{proof}
It is similar to that of Corollary \ref{cor:thmDi}.
The lines of $\cru\fV_{1,2}(S')$ contained in $\check p$ are the
push--forward by $\fP'''_0\to |\O_Y(1)|$ of the regular components of
$\fV_{1,2}(\bar X)$ listed in Remark \ref{rem:480}
\eqref{tetra2-count}, with the exception of $V(M_i,\delta_{P_i}=2)$.
They sum up (with their respective multiplicities) to a degree $348$
curve, while $V(M_i,\delta_{P_i}=2)$ has degree at least $132$ in
$\fL_i$ by Proposition \ref{p:2-triangle}.
\end{proof}

\begin{proposition} [Limits of $3$-nodal curves]
\label{p:3-4planes}
The family $\varpi:\bar X\to \Delta$ is absolutely $3$--good,
and the limit Severi variety $\fV_{1,3}(\bar{X})$ is reduced,
consisting of:\\ 
\begin{inparaenum}[\normalfont (i)]
\item
the $1024$ points $V(E,E',E'')$, for $E,E',E'' \in
\{\plm{E}_{ij},1\leqslant i<j\leqslant 4\}$ such that the
span of the three corresponding double points of $S$ is not contained
in a face of $S_0$;\\
\item
the $192$ schemes $V(M^k,\delta_{T^k}=1,E,E')$,
for $1\leqslant k\leqslant 4$ and $E,E'\in
\{\plm{E}_{ij},1\leqslant i<j\leqslant 4\}$, such that
the two double points of $S$ corresponding to $E$ and $E'$ and 
the vertex with index $k$ span a plane which is not a
face of $S_0$.
They each consist of $3$ points;\\
\item
the $24$ schemes $V(M_{ij},\delta_{W_{ij}}=2,E)$, for $1\leqslant i<
j\leqslant 4$, and 
$E\in\{\plm{E}_{ij},1\leqslant i< j\leqslant 4\}$, 
such that the double point of $S$ corresponding to $E$ does not lie on
the edge $P_i\cap P_j$ of $S_0$, and that these two together do not
span a face of $S_0$.
They each consist of $16$ points;\\
\item \label{3nodal-plane4ics}
the $4$ schemes $V(M_i,\delta_{P_i}=3)$, each consisting of $304$
points.
\end{inparaenum}
\end{proposition}

\begin{proof}
The list in the statement enumerates all regular components of the limit Severi
variety $\fV_{1,3}(\bar{X})$ with their degrees
(as before, Corollary \ref{coro:D} only gives $304$ as a lower bound
for the degree of \eqref{3nodal-plane4ics}).
They therefore add up to a total of at least $3200$ points,
which implies, by Proposition \ref{p:3-4planes}, that
$\fV_{1,3}(\bar{X})$ has no component besides the regular ones, 
and that those in \eqref{3nodal-plane4ics} have degree exactly $304$.
Reducedness then follows from Remark \ref{rem:D}, \eqref{rem:D-red}.
\end{proof}

In conclusion, all the above degenerations of
quartic surfaces constructed from $\bar X\to \Delta$ with a twist of
$\L$ are $3$--well behaved, with $\bar X$ as a good model.
In particular:

\begin {corollary} [Theorem \ref{T:tetrahedron} for $\delta=3$]
\label{cor:thmBiii} 
Same setting as in Corollary \ref {cor:thmBi}. 
The limits in $|\O_{S_0}(1)|$ of $3$-tangent planes to 
$S_t$, for $t\in \Delta^ *$, consist of:\\
\begin{inparaenum}[\normalfont (i)]
\item the $1024$ planes (each with multiplicity 1) containing three double points of $S$ but no edge of
$S_0$;\\
\item the $192$ planes  (each with multiplicity $3$) containing a vertex of $S_0$ and two double points of
$S$, but no edge of $S_0$;\\
\item the $24$ planes (each with multiplicity $16$)  containing an edge of $S_0$ and a double point of $S$
on the opposite  edge;\\
\item the $4$ faces of $S_0$ (each with multiplicity $304$).\end{inparaenum}
\end{corollary}

\section{Other degenerations}\label{sec:pair of quadrics}

The degeneration of a general quartic we considered in \S \ref {S:4planes} is, in a sense, one of the most intricate.
There are \emph{milder} ones, e.g. to:\\   
\begin{inparaenum}[(i)]
\item \label {cp}a general union of a cubic and a plane;\\
\item  \label {qq} a general union of two quadrics (this is an
  incarnation of a well known degeneration of $K3$ surfaces  described
  in \cite {clm}).\\
\end{inparaenum}
Though we encourage the reader to  study  in detail the instructive
cases of   degenerations \eqref {cp} and \eqref {qq}, we will not
dwell on this here,
and only make the following observation about degeneration \eqref
{qq}. Let $f: S\to \Delta$ be such a degeneration, with central fibre
$S_0=Q_1\cup Q_2$, where $Q_1,Q_2$ are two general quadrics meeting
along a smooth quartic elliptic curve $R$.
Then the limit linear
system of $|\O_{S_t}(1)|$ as $t\in \Delta^*$ tends to $0$ is 
just $|\O_{S_0}(1)|$, so that 
$f:S\to \Delta$ endowed with $\O_S(1)$ is absolutely good.

On the other hand, there are also  degenerations to special
singular irreducible surfaces, as the one we will consider  in \S
\ref {S:Kummer-descr} below.
In the subsequent sub--section, we will consider for further purposes
another degeneration, the central fibre of which is still a (smooth)
$K3$ surface.

\subsection{Degeneration to a double quadric}\label{ssec:doublequad}

Let $Q\subset \P^ 3$ be a smooth quadric and let $B$ be a general curve of type $(4,4)$ on $Q$. We  consider the double cover $p: S_0\to Q$ branched along $B$. This is a $K3$ surface and  there is a smooth family $f: S\to \Delta$ with general fibre a general quartic surface
and central fibre $S_0$.  The pull--back to $S_0$ of plane sections of $Q$ which are bitangent to $B$ fill up a component $\mathfrak V$ of multiplicity 1 of the crude limit Severi variety $\mathfrak V^ {\rm {cr}}_2$. Note that  $\mathfrak V^ {\rm {cr}}_2$ naturally sits in $\vert \O_{S_0}(1)\vert\cong \check {\P}^ 3$ in this case, hence one can unambiguously talk  about its degree.
Although it makes sense to conjecture that $\mathfrak V$ is
irreducible,  we will only prove the
following weaker statement: 

\begin{proposition}\label {prop:irrcomp} The curve $\mathfrak V$ contains an irreducible component of degree at least 36. 
\end{proposition}

We point out the following immediate consequence, which will be needed in \S \ref  {ssec:irred} below:

\begin{corollary}\label {cor:irrcomp} If $X$ is a general quartic surface in $\P^ 3$, then the Severi variety $V_2(X, \O_X(1))$ (which naturally sits in  $\vert \O_{X}(1)\vert\cong \check {\P}^ 3$) has an irreducible component of degree at least 36.
\end{corollary}

To prove Proposition \ref {prop:irrcomp} we make a further
degeneration to the case in which $B$ splits as $B=D+H$, where $D$ is
a general curve of type $(3,3)$ on $Q$, and $H$ is a general curve of
type $(1,1)$, i.e. a general plane section of
$Q \subset \P^ 3$. Then
the limit of $\mathfrak V$ contains the curve $\mathfrak W:=\mathfrak
W_{D,H}$ in $\check {\P}^ 3$ parametrizing those planes in $\P^3$ tangent to
both $H$ and $D$ 
(i.e., ${\mathfrak W}$ is the intersection curve of
the dual surfaces $\check H$ and $\check D$).
Note that  $\check H$ is the quadric cone
circumscribed to the quadric $\check Q$ and
with vertex the point $\check P$ orthogonal to the plane $P$ cutting
out $H$ on $Q$, 
while $\check D$ is a surface scroll, the degree of which is $18$ by
Proposition \ref{p:dJ}, hence $\deg(\mathfrak W)=36$.
To prove Proposition \ref {prop:irrcomp}, it 
suffices to prove that: 

\begin{lemma}\label {lem:irrcomp} The curve $\mathfrak W$  is irreducible.
\end{lemma}

To show this, we need a preliminary information. Let us consider the
irreducible, locally closed subvariety $\mathcal U\subset \vert
\O_Q(4)\vert$ of dimension $18$, consisting of all curves $B=D+H$,
where $D$ is a smooth, irreducible curve of type $(3,3)$, and $H$ is a
plane section of $Q$ which is not tangent to $D$. Consider   $\mathcal
I\subset \mathcal U\times \check {\P}^ 3$  the Zariski closure of the
set of all pairs $(D+H,\Pi)$ such that the plane $\Pi$ is tangent to both
$D$ and $H$, i.e. $\check \Pi\in \check H\cap \check D$. We have the
projections $p_1: \mathcal I \to \mathcal U$ and  
$p_2: \mathcal I\to  \check {\P}^ 3$.
The curve $\mathfrak W$ is a general fibre of $p_1$.

\begin{lemma}\label {lem:irrcomp2} The variety $\mathcal I$ contains a unique irreducible component $\mathcal J$ of dimension 19 which dominates $\check {\P}^ 3$ via the map $p_2$.
\end{lemma}

\begin{proof} 
Let $\Pi$ be a general plane of $\P^ 3$.
Consider the conic $\Gamma:=\Pi\cap Q$, and fix distinct points
$q_1,\ldots, q_6$ on $\Gamma$.  There is a plane $P$ tangent to  
$\Gamma$ at  $q_1$, and a cubic surface $F$ passing through
$q_3,\ldots, q_6$ and tangent to $\Gamma$ at  $q_2$; moreover $P$
and $F$ can be chosen general enough for $D+H$ to belong to $\mathcal U$,
where $H=P\cap Q$ and $D=F\cap Q$.
Then $(D+H, \Pi) \in \mathcal I$, which proves
 that $p_2$ is dominant.
 
 Let  $\mathcal F_\Pi$ be the fibre of $p_2$ over $\Pi$. The above
 argument shows that there is a dominant map $\mathcal
 F_\Pi\dasharrow \Gamma^ 2\times \rm {Sym}^4(\Gamma)$ whose general
 fibre is an open subset of $\P^ 1\times \P^ 9$: precisely, if
 $((q_1,q_2), q_3+\ldots+q_6)\in  \Gamma^ 2\times {\rm
   {Sym}}^4(\Gamma)$ is a general point, the $\P^ 1$ is the linear
 system of plane sections of $Q$ tangent to $\Gamma$ at $q_1$, and the
 $\P^ 9$ is the linear subsystem of $\vert \O_Q(3)\vert$ consisting of
 curves  passing through $q_3,\ldots, q_6$ and tangent to $\Gamma$ at
 $q_2$. The existence and unicity of $\mathcal J$ follow. \end{proof} 
 
 Now we consider the commutative diagram
 \begin{equation}\label{eq:stein}
 \xymatrix{
{ \mathcal J}' \ar[r]^\nu \ar[dr] \ar[d]_{p'}
& {\mathcal J}  \ar[d]^{p_1}  \\
\mathcal U'  \ar[r]_{f} &  \hil U
}\end{equation}
 where $\nu$ is the normalization of $\mathcal J$, and $f\circ p'$ is the
Stein factorization of $p_1\circ \nu: \mathcal J'\to \hil U$. The
morphism $f: \mathcal U'\to \mathcal U$ is finite of degree $h$, equal
to the number of irreducible components of the general fibre of $p_1$,
which is  $\mathfrak W$. The irreduciblity of $\mathcal J$ implies that
the monodromy group of $f: \mathcal U'\to \mathcal U$ acts
transitively on the set of components of $\mathfrak W$. 
\medskip

\begin{proof}[Proof of Lemma \ref  {lem:irrcomp}]
We need to prove that $h=1$.
To do this, fix  a general $D\in \vert\O_Q(3)\vert$, and
consider the curve $\mathfrak W=\mathfrak W_{D,H}$, with $H$ general,
which  
consists of $h$ components. 
We can move $H$ to be a section of $Q$ by a general tangent plane $Z$.
Then the quadric cone $\check H$
degenerates to the tangent plane 
$T_{\check Q,z}$ to $\check Q$ at $z:=\check Z$, counted with
multiplicity $2$.

We claim that, for $z\in \check Q$ general, the intersection of $T_{\check Q,z}$ with $\check D$ is irreducible. Indeed, since $\check D$ is a scroll, a plane section of $\check D$ is reducible if and only if it contains a ruling, i.e. if and only if it is a tangent plane section of $\check D$. 
Since $\check D\neq \check Q$, the biduality theorem implies the claim.

The above assertion implies  $h\leqslant 2$. If equality holds,
the general curve $\mathfrak W$ consists of two curves
which, by transitivity of the monodromy action of $f$,
are both unisecant  to the lines of the
ruling of $\check D$. 

To see that this is impossible, let us degenerate $D$ as $D_1+D_2$, where $D_1$ is a general curve of type
$(2,1)$  and $D_2$ is a general curve of type $(1,2)$ on $Q$. Then $\check D$ accordingly degenerates  and its limit contains
as irreducible components $\check D_1$ and $\check D_2$, which  are
both scrolls of degree $4$
(though we will not use it, we note that
$D_1\cdot D_2=5$ and the (crude) limit of $\check D$ in the above
degeneration 
consists of the union of $\check D_1$, $\check D_2$, and of the five
planes dual to the points of $D_1\cap D_2$, each of the
latter counted with multiplicity $2$).
We denote by $\mathfrak D$ either one of the curves $D_1,D_2$.

Let again $H$ be a general plane section of $Q$. We claim that the intersection of $\check {\mathfrak D}$ with $\check H$ does not contain any unisecant curve to the lines of the ruling of $\check {\mathfrak D}$. This clearly implies that the general curve $\mathfrak W$ cannot split into two unisecant curves to the lines of the ruling of $\check D$, thus proving that $h=1$. 

To prove the  claim,  it suffices to do it for specific ${\mathfrak
  D}$, $Q$ and $H$. For ${\mathfrak D}$ we take the rational normal
cubic with affine parametric equations $x=t, y=t^ 2, z=t^ 3$, with
$t\in \C$.  For $Q$ we take the quadric with affine equation $x^ 2+y^
2-xz-y=0$, and for $H$ the intersection of $Q$ with the plane
$z=0$. Let $(p,q,r)$ be affine coordinates in the dual space, so that
$(p,q,r)$ corresponds to the plane with equation $px+qy+rz+1=0$ (i.e.,
we take as plane at infinity  in the dual space the orthogonal to the
origin).  Then the affine equation of $\check {\mathfrak D}$  is
gotten by eliminating $t$ in the system
\begin{equation}\label{eq:dual}
pt+qt^ 2+rt^ 3+1=0,\quad p+2qt+3rt^ 2=0,
\end{equation}
which defines the ruling $\rho_t$ of  $\check{\mathfrak D}$ orthogonal to the tangent line to ${\mathfrak D}$ at the point with coordinates $(t,t^ 2,t^ 3)$, $t\in \C$. The affine equation of $\check H$ is gotten by imposing that the system
\[
px+qy+qz+1=0,\quad  x^ 2+y^ 2-xz-y=0, \quad z=0,
\]
 has  one solution with multiplicity $2$;
the resulting equation is $p^ 2-4q-4=0$. Adding this  to 
\eqref {eq:dual} means intersecting  $\check H$ with $\rho_t$;
for $t\neq 0$, the
resulting  system can be written as 
  \[
  p^ 2t^ 2+8pt-4(t^ 2-3)=0,   \quad q=\frac {p^ 2}4-1,\quad  r=\frac {4- p^ 2}{6t}-\frac p{3t^ 2}.
  \] 
 For a general $t\in \C$, the first equation gives  two values of $p$
 and the remaining equations the corresponding values of $q$ and $r$,
 i.e., we get the coordinates $(p,q,r)$ of the two intersection points
 of $\check H$ and $\rho_t$. Now we note that the discriminant of   
 $p^ 2t^ 2+8pt-4(t^ 2-3)$ as a polynomial in $p$  is $16t^ 2(t^ 2+1)$,
 which has the two simple solutions $\pm i$. This implies that the
 projection on  $\mathfrak D\cong \mathbf P^ 1$ of the curve cut out
 by $\check H$ on $\check{\mathfrak D}$ has two simple ramification
 points. 
In particular $\check H \cap \check{\mathfrak D}$ is locally
 irreducible at these points, and
it cannot split as two unisecant curves to the lines of
 the ruling. This proves the claim and ends the proof of the Lemma. 
\end{proof}

\section{Kummer quartic surfaces in $\P^3$}
\label{S:Kummer-descr}

This section is devoted to the description of some properties of
quartic \emph{Kummer surfaces} in $\P^3$.
They are quartic surfaces with 16 ordinary double points $p_1,\ldots, p_{16}$ as their
only singularities. Alternatively a Kummer surface $X$ is the  image of the Jacobian $J(C)$ of  a smooth 
genus $2$ curve $C$, via the degree $2$ morphism  $\vartheta: J(C)\to
X\subset \P^ 3$  
determined by the complete linear system
$|2C|$, where $C\subset J(C)$ is the Abel--Jacobi embedding, so that $(J(C),C)$ is a principally polarised abelian surface (see, e.g.,  
\cite[Chap.10]{birkenhake-lange}). Since $\vartheta$ is composed with the $\pm$ involution on $J(C)$, the 16 nodes of $X$ are the images of the 16 points of order $2$ of $J(C)$.
By projecting from a node,  Kummer surfaces can be realised as  
double covers of the plane, branched along the union of six distinct
lines tangent to one single conic
(see, e.g., \cite[Chap.VIII, Exercises]{beauville}).
We refer to the classical book \cite{hudson} for a thorough
description of these surfaces
(see also \cite[Chap.10]{topics}).

\subsection{The $16_6$ configuration and self--duality}
\label{s:16_6}

An important feature of Kummer surfaces is that they carry
a so-called \emph{$16_6$--configuration} (see 
\cite{gonzalez}, as well as the above listed references).
Let $X$ be such a surface. There are
exactly $16$ distinct planes $\Pi_i$  tangent to $X$ along a \emph{contact conic} $\Gamma_i$,
for $1\leqslant i\leqslant 16$. The contact conics are the images  of the 16 symmetric 
theta divisors $C_1,\ldots, C_{16}$ on $J(C)$.  Each of them contains exactly $6$ nodes of $X$,  
coinciding with the branch points of the map $\left.\vartheta\right|_{C_i}: C_i\cong C\to \Gamma_i\cong \P^ 1$  determined by the 
canonical $g^ 1_2$ on $C$. 

Two  conics $\Gamma_i,\Gamma_j$, $i\neq j$,
intersect at exactly two points, which are double points of $X$:
they are the nodes corresponding to the two order $2$ points of $J(C)$
where $C_i$ and $C_j$ meet. 
Since the restriction map ${\rm Pic}^ 0(J(C))\to {\rm Pic}^ 0(C)$ is  an isomorphism,  
there is no pair of points of $J(C)$ contained in three different theta divisors. This implies that, given a pair of nodes of $X$,
there are exactly two contact conics containing both of them.
In other words, if we fix an $i\in \{1,\ldots,16\}$,
the map from $\{1,\ldots,16\}-\{i\}$ to the set of pairs of distinct
nodes of $X$ on $\Gamma_i$,
which maps $j$ to $\Gamma_i\cap \Gamma_j$, is bijective.
This yields  that each node of
$X$ is contained in exactly $6$ conics $\Gamma_i$. The configuration
of 16 nodes and 16 conics with the above described incidence property
is called a $16_6$--configuration. 

Let $\tilde{X}$ be the minimal smooth model of $X$, 
$E_1,\ldots,E_{16}$  the $(-2)$-curves over the nodes $p_1,\ldots,
p_{16}$ of $X$ respectively,
and $D_i$ the proper transform of the conic $\Gamma_i$, for
$1\leqslant i\leqslant 16$.
Since $\tilde X$ is a $K3$ surface and the $D_i$'s are rational
curves, the latter are $(-2)$-curves. The  $16_6$--configuration
can be described in terms of the existence of the two
sets 
\[\mathcal E =\{E_1,\ldots,E_{16}\} \quad {\rm and}\quad\mathcal D=  \{D_1,\ldots,D_{16}\} \]
of $16$ pairwise disjoint $(-2)$--curves, enjoying the further property
that each curve of a given set meets exactly six curves of the
other set, transversely  at a single point.

\begin{proposition}
\label{p:dual-kummer}
Let $X$ be a Kummer surface. Then its dual $\check{X} \subset \check{\P}^3$ is also
a Kummer surface.
\end{proposition}

\begin{proof} By Proposition \ref {p:deg-dual-sing}, we have  $\deg(\check{X})=4$. 
 Because of the singularities on $X$, the Gauss map 
$\gamma_X:X \dashrightarrow \check{X}$
is not a morphism.  However we get an elimination of indeterminacies
$$\xymatrix@=15pt{
& \tilde{X} \ar[dl]_f \ar[dr]^g & \\
X \ar@{-->}[rr]_{\gamma_X} && \check{X}
}$$
by considering the minimal smooth model $\tilde{X}$ of $X$.
The morphism $f$ is the contraction of the sixteen curves in $\mathcal E$, and $g$ maps each $E_i$ to a conic which is the 
dual of the tangent cone to $X$ at the node corresponding to $E_i$.
On the other hand, since $\gamma_X$ contracts each of the curves $\Gamma_1,\ldots, \Gamma_{16}$ to a point,
then $g$  contracts the curves in $\mathcal D$ to as many ordinary double points of $\check{X}$.
The assertion follows. 
\end{proof}

\subsection{The monodromy action on the nodes}
\label{s:monodromy}

Let $\ring{\mathcal K}$ be the locally closed subset of $\vert \O_{\P^ 3}(4)\vert$ whose points correspond to Kummer surfaces and let
$\pi: \mathcal X\to \ring{\mathcal K}$ be the  universal family:
over $x\in \ring{\hil K}$, we have the corresponding Kummer surface $X=\pi^{-1}(x)$.
We have a subscheme $\mathcal N\subset \mathcal X$ such that 
$p:=\left.\pi \right|_{\mathcal N}: \mathcal N\to \ring{\hil K}$ is a finite morphism of degree 16: 
the fibre $p^ {-1}(x)$ over $x\in \ring{\hil K}$ consists of the nodes of
$X$.  We denote by  $G_{16,6}\subset \mathfrak S_{16}$ the monodromy
group of $p: \mathcal N\to \ring{\hil K}$  

There is in addition another degree 16 finite covering $q: \mathcal
G\to \ring{\hil K}$: for $x\in \ring{\hil K}$, the fibre $q^ {-1}(x)$
consists of the set of the contact conics on $X$. Proposition \ref
{p:dual-kummer} implies that the monodromy group of this covering is
isomorphic to  $G_{16,6}$. Then we can consider the commutative
square 
\begin{equation}\label{eq:square}
\xymatrix@=15pt{
\mathcal J \ar[d]_{p'}  \ar[rr]^{q'} && \mathcal N \ar[d]^{p} \\
\mathcal G \ar[rr]_{q} && \ring{\hil K}  
}
\end{equation}
where $\mathcal J$ is the incidence correspondence between nodes and conics. Note that 
$p', q'$ are both finite of degree 6, with isomorphic monodromy groups (see again Proposition  \ref {p:dual-kummer}). 

Here, we collect some results on the  monodromy groups  of the
coverings appearing in \eqref {eq:square}. They are probably well
known to the experts, but we could not find any reference for them.

\begin{lemma}
\label{l:S_6} The monodromy group of $q': \mathcal J\to \mathcal N$
and of $p': \mathcal J\to \mathcal G$  is  the
full symmetric group $\mathfrak{S}_6$.
\end{lemma}

\begin{proof} It suffices to prove only one of the two assertions, e.g. the one about $p'$.
Let $X$ be a general Kummer surface and let $e$ be a node of $X$. As we noticed, by projecting 
from $e$, we realise $X$ as a double cover of $\P^ 2$ branched along $6$ lines tangent to a conic $E$, which is the image of the $(-2)$--curve over $e$. These 6 lines are  the images of the six contact conics through $e$, i.e. the fibre over $q'$. Since $X$ is general, these 6 tangent lines are general. The assertion follows. \end{proof}

\begin{corollary}\label{cor:geer}
The group $G_{16,6}$ acts transitively, 
so $\mathcal G$ and $\mathcal N$ are irreducible.
\end{corollary}

\begin{proof} It suffices to prove that the monodromy of $p: \mathcal N\to \ring{\hil K}$ is transitive.
This follows from Lemma \ref {l:S_6} and from the fact that any two nodes of a Kummer surface lie on some contact conic.
\end{proof}

It is also possible to deduce the transitivity of the monodromy action
of $p$ and $q$ from the irreducibility of
the Igusa quartic solid, which parametrizes quartic Kummer surfaces
with one marked node (see, e.g., \cite[Chap.10]{topics}).
The following is stronger:

\begin{proposition}
\label{p:2-transitivity} The group $G_{16,6}$ acts $2$--transitively.
\end{proposition}

\begin{proof} Again, it suffices to prove the assertion for $p: \mathcal N\to \ring{\hil K}$.  By Corollary \ref {cor:geer},
proving that the monodromy is 2--transitive is equivalent to showing
that the stabilizer of a point in the general fibre of $p$ acts
transitively on the remaining points of the fibre.
Let $X$ be a general Kummer surface and $e\in X$ a node. Consider the
projection from $e$, which realizes $X$ as a double cover of $\P^ 2$
branched along 6 lines tangent to a conic $E$. The 15 nodes on $X$
different from $e$ correspond to the pairwise intersections of the 6
lines. Moving the tangent lines to $E$ one leaves the node $e$ fixed,
while acting transitively on the others. \end{proof} 

Look now at the pull back $q^ *(\mathcal N)$. Of course $\mathcal J$ is a component of $q^ *(\mathcal N)$. We set 
$\mathcal W=q^ *(\mathcal N)-\mathcal J$, and the morphism $p':
\mathcal W\to \mathcal G$ which is finite of degree $10$. We let
$H_{16,6}\subseteq \mathfrak S_{10}$ be the monodromy of this covering.  

\begin{lemma}
\label{l:preparation} The group $H_{16,6}$ acts transitively, i.e.
$\mathcal W$ is irreducible. 
\end{lemma}

\begin{proof}
Let $a, b\in  X$ be two nodes not lying on the contact conic $\Gamma$. There is a contact conic
$\Gamma'$ that contains both $a$ and $b$; it
meets $\Gamma$ transversely in two points, distinct from $a$ and $b$, that
we shall call $c$ and $d$.
Now a monodromy transformation that fixes $\Gamma'$ and fixes $c$
and $d$ necessarily fixes $\Gamma$.
It therefore suffices to find a monodromy transformation fixing $\Gamma'$
which fixes $c$ and $d$, and sends $a$ to $b$. Such a transformation
exists by Lemma \ref {l:S_6}.
\end{proof}

\begin{proposition} \label{p:3-transitivity} Let $X$ be a general Kummer surface. Then:\\
\begin{inparaenum}[(i)]
\item \label{it:alpha} $G_{16,6}$ acts
transitively the set of  unordered triples of distinct nodes belonging to a contact conic;\\
\item \label{it:beta}  the action of $G_{16,6}$ on the set of unordered triples of distinct nodes not belonging to a contact conic has 
at most two orbits. \end{inparaenum}
\end{proposition}

To prove this, we need to consider  degenerations of  Kummer surfaces when the principally polarised
abelian surface $(J(C),C)$ becomes non--simple, e.g., when $C$ degenerates to the union of two elliptic 
curves $E_1,E_2$ transversally meeting at a point. In this case the linear system $\vert 2(E_1+E_2)\vert$ on
the abelian surface $A=E_1\times E_2$, is still base point free, but it determines a degree 4 morphism 
$\vartheta: A\to \mathbf Q\cong \P^ 1\times \P^ 1\subset \P^ 3$ (where
$\mathbf Q\subset \P^ 3$ is a smooth quadric), factoring through the
\emph {product Kummer surface} 
$\mathsf{X}=A/ \pm$, and a double cover $\mathsf{X}\to \mathbf Q$
branched along a curve of bidegree $(4,4)$ which is a 
union of $8$ lines; the lines in question are $L_{1a}= \P^ 1\times \{a\}$ (resp. $L_{2b}= \{b\} \times \P^ 1$) where $a$ (resp. $b$) ranges among the four branch points of the morphism $E_1\to (E_1/\pm)\cong \P^ 1$ (resp. $E_2\to (E_2/\pm)\cong \P^ 1$).
We call the former \emph{horizontal lines}, and the latter \emph{vertical
lines}. 
Each of them has four marked points: on a line $L_{1a}$
(resp. $L_{2b}$), these are the four points $L_{2b}\cap L_{1a}$ where
$b$ (resp. $a$) varies as above. 
One thus gets 16 points, which are the limits on  $\mathsf{X}$
of the 16 nodes of a general Kummer surface $X$. 
The limits on  $\mathsf{X}$ of the sixteen contact conics on a general
Kummer surface $X$ are the sixteen curves $L_{1a}+L_{2b}$.
On such a curve, the limits of the six double points on a contact
conic on a general Kummer surface are the six marked points on
$L_{1a}$ and $L_{2b}$ that are distinct from $L_{1a}\cap L_{2b}$.\medskip

\begin{proof}[Proof of Proposition \ref{p:3-transitivity}]
Part (i) follows from Lemma \ref{l:S_6}.
As for part (ii), consider three distinct nodes $a$, $a'$ and $a''$ 
(resp. $b$, $b'$ and $b''$) of $X$ that
do not lie on a common conic of the $16_6$ configuration on $X$. We
look at their limits $\mathsf{a}$, $\mathsf{a}'$ and $\mathsf{a}''$
(resp. $\mathsf{b}$, $\mathsf{b}'$ and $\mathsf{b}''$)
on the product Kummer surface $\mathsf{X}$; they are  in
one of the two configurations (a) and (b) described in 
Figure \ref{f:triples}.

\begin{figure}
\begin{center}
\includegraphics[height=2.5cm]{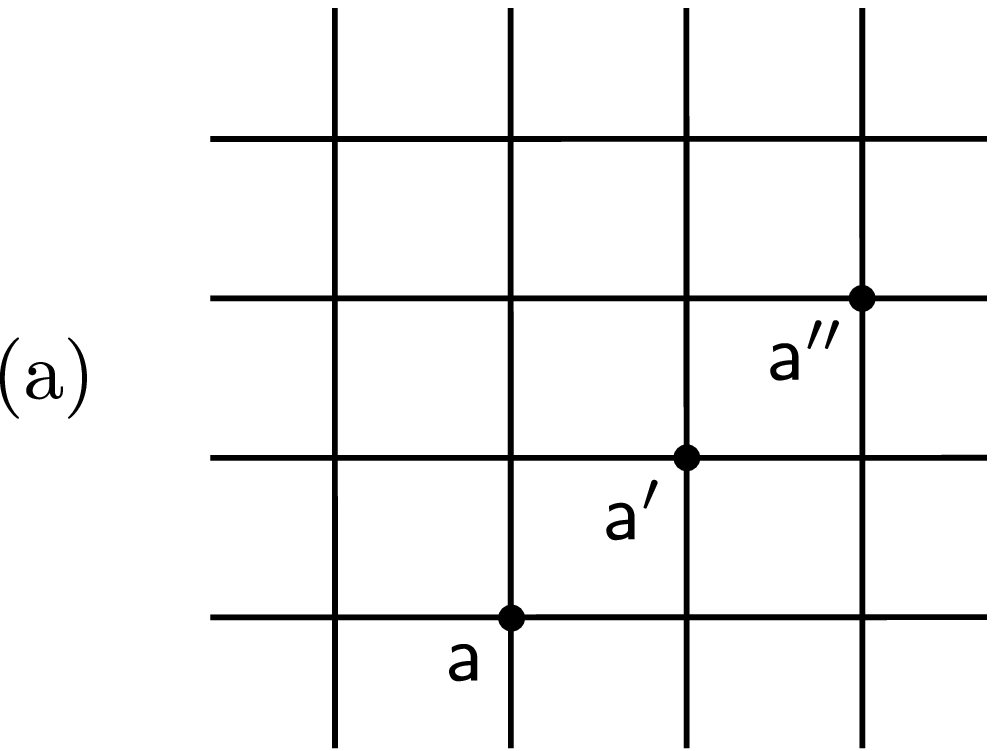}
\hspace{2cm}
\includegraphics[height=2.5cm]{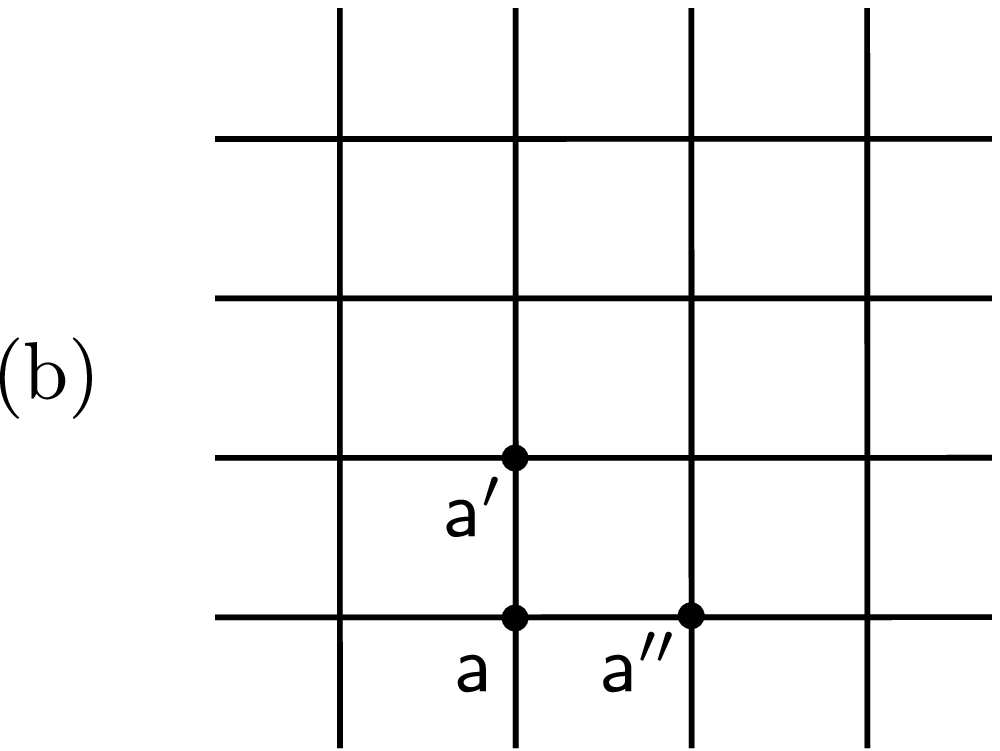}
\end{center}
\caption{Limits in a product Kummer surface
of three double points not on a double conic }
\label{f:triples}
\end{figure}

The result  follows from the fact that the monodromy of the family
of product  Kummer surfaces acts as the full
symmetric group $\mathfrak{S}_4$
on the two sets of vertical and horizontal lines respectively. Hence the triples in configuration (a) [resp. in configuration (b)]
are certainly in one and the same orbit. 
\end{proof}

\section{Degeneration to a Kummer surface}
\label{S:Kummer-degen}

We consider a family $f:S \to \Delta$ of surfaces in $\P^3$ induced
(as explained in \S\ref{s:setting}) by a pencil generated 
by a general quartic surface $S_{\infty}$ and a general  Kummer
surface $S_0$. We will  describe a related $\delta$-good model
$\varpi:\bar{X} \to \Delta$ for $1\le \delta\le 3$.

\subsection{The good model}

Our construction is as follows: \\
{\it \begin{inparaenum}[(I)] 
\item \label{Kum:bs-change}
we first perform a degree 2 base change on $f:S \to \Delta$;\\
\item \label{Kum:resolution}
then we resolve the singularities of the new family;\\
\item \label{Kum:blow-up} we blow--up the proper transforms of the sixteen contact conics
on $S_0$.\\
\end{inparaenum}}
The base change is useful to analyze the contribution of curves
passing through a node of $S_0$.

\subsubsection{Steps (\ref{Kum:bs-change}) 
and (\ref{Kum:resolution})}

The total space $S$ is smooth, analytically--locally given by the
equation 
\[ x^2+y^2+z^2=t \]
around each of the  double points of $S_0$.
We  perform a degree 2 base change on $f$, and call $\bar{f}:
\bar{S} \to \Delta$ the resulting family. The total space $\bar{S}$ has
$16$ ordinary double points at the preimages of the nodes of $S_0$.

We let $\epsilon_1: X \to \bar{S}$ be the resolution of these $16$
points, gotten by a simple  blow--up at each point. We have the new family $\pi: X\to \Delta$,
with $\pi=\bar f\circ \epsilon_1$.
The new 
central fibre $X_0$ consists of the minimal smooth model $\tilde{S}_0$ of
$S_0$, plus the  exceptional divisors $Q_1,\ldots,Q_{16}$. 
These are all isomorphic to a smooth quadric $\mathbf Q\cong \P^ 1\times \P^ 1\subset \P^ 3$.
 We let $E_1,\ldots,E_{16}$ be the exceptional divisors of 
$\tilde{S}_0 \to S_0$. Each $Q_i$ meets $\tilde{S}_0$ transversely
along the curve $E_i$, and two distinct $Q_i,Q_j$ do not meet. 

\subsubsection{Step (\ref{Kum:blow-up})}

As in \S \ref {s:16_6}, we  let $D_1,\ldots,D_{16}$ be the proper transforms of the $16$ contact
conics $\Gamma_1,\ldots, \Gamma_{16}$ on $S_0$: they are pairwise disjoint $(-2)$-curves in
$X_0$.
We consider the blow-up $\epsilon_2:\bar{X} \to X$
of $X$ along them.
The surface $\tilde{S}_0$ is isomorphic to its strict
transform on $\bar{X}_0$.
Let $W_1,\ldots,W_{16}$ be the exceptional divisors of $\epsilon_2$. 
Each $W_i$ meets $\tilde{S}_0$ transversely along the (strict transform
of the) curve $D_i$.
Note that, by the Triple Point Formula  \ref {l:triple-point}, one has
$\deg(N_{D_i\vert W_i})=-\deg(N_{D_i\vert \tilde S_0})-6=-4$, so that 
$W_i$ is an $\F_4$--Hirzebruch surface, and $D_i$ is the negative section on it.

We call $\tilde{Q}_1,\ldots,\tilde{Q}_{16}$ the strict transforms of
$Q_1,\ldots,Q_{16}$ respectively. They respectively meet $\tilde{S}_0$
transversely along (the strict transforms of) $E_1,\ldots,E_{16}$. For $1\leqslant i\leqslant 16$,
there are exactly six curves among the $D_j$'s that meet $E_i$: we
call them $D^i_{1},\ldots,D^i_{6}$.
The surface $\tilde{Q}_i$ is
the blow-up of $Q_i$ at the six intersection points of $E_i$ with
$D^i_{1},\ldots,D^i_{6}$: we call
$\pG^i_{1},\ldots,\pG^i_{6}$ respectively the six corresponding
$(-1)$-curves on $\tilde{Q}_i$. Accordingly, $\tilde{Q}_i$ meets
transversely six $W_j$'s, that we denote by $W^1_{i},\ldots,W^6_{i}$,
along $\pG^i_{1},\ldots,\pG^i_{6}$ respectively.
The surface  $\tilde{Q}_i$ is disjoint from the remaining $W_j$'s.

For $1\leqslant j\leqslant 16$,
we denote by $E^1_{j},\ldots,E^6_{j}$ the six $E_i$'s that meet $D_j$.
There are correspondingly six $\tilde{Q}_i$'s that meet $W_j$: we denote them by
 $\tilde{Q}^1_j,\ldots,\tilde{Q}^6_j$, and let
$G_j^1,\ldots,G_j^6$
be their respective intersection curves with $W_j$.
Note the equality of sets
\[
\left\{\pG^i_s,\ 1\leqslant i\leqslant 16,\ 1\leqslant s\leqslant 6\right\}
=\left\{G^s_j,\ 1\leqslant j\leqslant 16,\ 1\leqslant s\leqslant 6\right\}.
\]

We shall furthermore use the following notation (cf.
\S\ref{conv}). 
For $1\leqslant j\leqslant 16$, we let $F_{W_j}$ (or simply $F$) be the
divisor class of the  ruling on $W_j$, and $H_{W_j}$ (or simply $H$)
the divisor class $D_j+4F_{W_j}$.
Note that $G^s_j\lineq_{W_j} F$
and $\pG^s_i\lineq_{W^s_i} F$, for $1\leqslant i,j\leqslant 16$ and
$1\leqslant s\leqslant 6$. 
We write $H_0$ for the pull-back to $\tilde{S}_0$ of the
plane section class of $S_0 \subset \P^3$.
For $1\leqslant i \leqslant 16$, we let $L_i'$ and $L_i''$ be the two
rulings of $Q_i$, and $H_{Q_i}$ (or simply $H$) be the divisor class
$L_i'+L_i''$;
we use the same symbols for their respective
pull-backs in $\tilde{Q}_i$.
When designing one of these surfaces by $\tQ_j^s$, we use the 
obvious notation ${L_j^s}'$ and ${L_j^s}''$.

\subsection{The limit linear system}
\label{s:kummer-ls}

We shall now describe the limit linear system
of $|\O_{\bar X_t}(1)|$ as $t\in \Delta^*$ tends to $0$,
and from this we will see that $\bar X$ is a good model of $S$ over
$\Delta$.
We start with  ${\fP} =\P( \varpi_*(\O_{\bar X}(1)))$, 
which is a $\P^ 3$--bundle over $\Delta$, whose fibre 
at $t\in \Delta$ is $\vert \O_{\bar X_t}(1)\vert$;
we set $\L=\O_{\bar{X}}(1)$, and $\vert \O_{\bar X_t}(1)\vert=|\L_t|$
for $t\in\Delta$. Note that $|\L_0|\cong |\O_{S_0}(1)|$.

We will proceed as follows: \\
\emph{\begin{inparaenum} [(I)]
\item \label{Kum:coplanar}
we first blow--up $\fP$ at the points of $\fP_0\cong |\L_0|$
 corresponding to planes in $\P^3$ containing at least three distinct
 nodes of $S_0$  
(i.e. either planes containing exactly three nodes, or planes in the
$16_6$ configuration);\\ 
\item \label{Kum:pencils}
then we blow--up  the resulting variety along the proper transforms of the 
lines of $ |\L_0|$ corresponding to pencils of planes in
$\P^3$ containing two distinct nodes of $S_0$;\\
\item \label{Kum:webs}
finally we blow--up along the proper transforms of the planes  of $
|\L_0|$ corresponding to webs of planes 
in $\P^3$ containing a node of $S_0$.\end{inparaenum}}

The description of these steps parallels the one in \S \ref
{s:4planes-limlin}, so we will be sketchy here.

\subsubsection{Step (\ref{Kum:coplanar}a)}

The $\binom{16}{3}-16\binom{6}{3}=240$ planes in $\P^3$ containing
exactly three distinct  nodes of $S_0$ correspond to the
$0$-dimensional subsystems
\begin{equation}\label{eq:3nodes}
  \bigl|H_0-E_{s'}-E_{s''}-E_{s'''}\bigr|_{\tilde{S}_0} 
\end{equation}
of $|H_0| \cong |\L_0|$, where $\{s',s'',s'''\}$ ranges through all
subsets of cardinality 3 of $\{1,\ldots, 16\}$ 
such that the nodes $p_{s'}, p_{s''}, p_{s'''}$ corresponding to the $(-2)$--curves $E_{s'},E_{s''},E_{s'''}$
do not lie in a plane of the $16_6$ configuration of $S_0$.
We denote by $C_{s's''s'''}$ the unique curve in the system \eqref {eq:3nodes}
and we set $H_{s's''s'''}=C_{s's''s'''}+E_{s'}+E_{s''}+E_{s'''}$, which lies in $\vert H_0\vert$ . 

The exceptional component $\tilde{\fL}_{s's''s'''}$ of the blow-up of
${\fP}$ at the point corresponding to  
$H_{s's''s'''}$ can be identified with the 3--dimensional complete linear system
\begin{equation*}
\fL_{s's''s'''}:=\bigl|\L_0(-\tQ_{s'}-\tQ_{s''}-\tQ_{s'''})\bigr|,
\end{equation*}
which is isomorphic to the projectivization of the kernel of the
surjective map
\begin{multline*}
\mathfrak f:\
(\varsigma',\varsigma'',\varsigma''',\varsigma) \in
\Biggl( \bigoplus_{s\in\{s',s'',s'''\}}
  \H^0\bigl(\tQ_{s}, \O_{\tQ_{s}}(H)\bigr)
\Biggr) 
\oplus \H^0\bigl(\tilde{S}_0,\O_{\tilde{S}_0}(C_{s's''s'''})\bigr)
\\
\longmapsto\quad
(\varsigma'-\varsigma,
\varsigma''-\varsigma,
\varsigma'''-\varsigma) \in
\bigoplus_{s\in\{s',s'',s'''\}}
\H^0\big(\O_{E_{s}}(-E_s)\big)\,
\cong\, \H^ 0\big(\P^ 1, \O_{\P^ 1}(2)\big)^ {\oplus 3}.
\end{multline*}

The typical element of $\fL_{s's''s'''}$ consists of\\
\begin{inparaenum}[(i)]
\item  the curve $C_{s's''s'''}$ on $\tilde S_0$, plus\\
\item one curve in $\vert \O_{\tQ_{s}}(H)\vert$ for each
${s\in\{s',s'',s'''\}}$, matching $C_{s's''s'''}$ along $E_{s}$,
plus \\
\item two rulings in each $W_j$ 
(i.e. a member of $|\O_{W_j}(2F)|=|\L_0\otimes \O_{W_j}|$), 
$1\leqslant j\leqslant 16$, matching along
the divisor $D_j$, while \\
\item the restriction to $\tilde Q_s$ is trivial for every $s\in \{1,\ldots,16\}- \{s',s'',s'''\}$.
\end{inparaenum}

The strict transform of $\fP_0$ is isomorphic to the blow--up of
$|H_0|$ at the point corresponding to $H_{s's''s'''}$. 
By Lemma \ref{l:tg-id}, the
exceptional divisor $\mathcal H_{s's''s'''}\cong \P^ 2$ of this blow-up identifies with the pull--back linear series on  $H_{s's''s'''}$ of the
$2$-dimensional linear system of lines in the plane 
spanned by $p_{s'}, p_{s''}, p_{s'''}$ (note that in this linear series there are three linear subseries corresponding to sections vanishing on 
the curves $E_{s'}, E_{s''}, E_{s'''}$ which are components of $H_{s',s'',s'''}$).

The divisor $\mathcal H_{s's''s'''}$  is cut out on the strict transform of 
$|H_0|$ by $\tilde{\fL}_{s's''s'''}$, along
the plane  $\Pi\subset \fL_{s's''s'''}$ given by the equation $\varsigma=0$ in the above notation.
The identification of $\mathcal H_{s's''s'''}$  with $\Pi$ is not
immediate. It would become more apparent by blowing up  the curves
$C_{s's''s'''}$ in the central fibre; we will not do this here,
because we do not need it, and we leave it to the reader 
(see Step (\ref{Kum:coplanar}b) for a similar argument). 
However, we note
that ${\rm ker} (\mathfrak f)\cap \{\varsigma=0\}$ coincides with the $\C^ 3$
spanned by three non--zero sections 
$(\varsigma_{s'},0,0,0),
(0,\varsigma_{s''},0,0),
(0,0,\varsigma_{s'''},0),$
where $\varsigma_{s}$ vanishes exactly on $E_s$ for each $s\in
\{s',s'',s'''\}$. 
These three sections correspond to three points $\pi_{s'}, \pi_{s''},
\pi_{s'''}$ in $\Pi$. In  the identification of $\Pi$ with $\mathcal
H_{s's''s'''}$ the points $\pi_{s'}, \pi_{s''}, \pi_{s'''}$ are mapped
to the respective pull--backs on $H_{s's''s'''}$  of  the three lines
$\ell_{s''s'''}=\langle p_{s''}, p_{s'''}\rangle,
\ell_{s's'''}=\langle p_{s'}, p_{s'''}\rangle, \ell _{s's''}=\langle
p_{s'}, p_{s''}\rangle$.

\subsubsection{Step (\ref{Kum:coplanar}b)}
The $16$ planes of the $16_6$ configuration correspond to
the $0$-dimensional subsystems
\[
\bigl|H_0-E^1_j-\cdots-E^6_j\bigr|_{\tilde{S}_0}
\subset \bigl|H_0\bigr|\cong \bigl|\L_0\bigr| \quad \quad
(1\leqslant j\leqslant 16),
\]
consisting of the only curve $2D_j$. The
blow-up of $\fP$ at these points introduces $16$ new components 
$\tilde{\fL}^j$, $1\leqslant j\leqslant 16$, in the central fibre,
respectively isomorphic to the linear systems
\[
\fL^ j:=\bigl|\L_0(-2W_j-\tQ_j^1-\cdots-\tQ_j^6)\bigr|.
\]
The corresponding line bundles restrict to the trivial
bundle on all components of $\bar{X}_0$ but $W_j$ and $\tQ_j^s$, for 
$1\leqslant s\leqslant 6$, where the restriction is to 
$\O_{W_j}(2H)$ and to
$\O_{\tQ_j^s}(H-2G_j^s)$,  respectively.

For each $s\in \{1,\ldots, 16\}$, the complete linear system
$|H-2G_j^s|_{\tQ_j^s}$ is $0$-dimensional, its
only divisor is the strict transform in $\tQ_j^s$ of the
unique curve in $|H|_{Q_j^s}$ that is singular at the point $D_j\cap
Q_j^s$.
This is the union of the proper transforms of the two
curves in $\vert {L_j^s}'\vert_{Q_j^s}$ and $\vert
{L_j^s}''\vert_{Q_j^s}$ through $D_j\cap Q_j^s$,  
and it cuts out a  $0$-cycle $Z_j^s$  of degree 2 on $G_j^s$.
We conclude that 
\begin{equation}
\label{W_j}
\fL^j\cong \bigl|\O_{W_j}(2H)\otimes {\cal I}_{Z_j}\bigr|,
\quad \text {for}\quad 1\leqslant j\leqslant 16,
\end{equation}
where ${\cal I}_{Z_j} \subset \O_{W_j}$ is the defining sheaf of
ideals of the $0$-cycle $Z_j:=Z_j^1+\cdots+Z_j^6$ supported on the six 
fibres $G_j^1,\ldots,G_j^6$ of the ruling of $W_j$.
We shall later study the rational map determined by this linear system
on $W_j$ (see Proposition \ref{p:2:1}).

For each $j$, the glueing of $\tilde{\fL}^ j$ with the
strict transform of $|H_0|$ is as follows: the 
exceptional plane $\mathcal H^j$ on the strict transform of $|H_0|$
identifies with 
$|\O_{D_j}(H_0)|\cong \vert \O_{\P^ 1}(2)\vert$ by Lemma
\ref{l:tg-id},
and the latter naturally identifies as the $2$-dimensional linear
subsystem of
$\bigl|\O_{W_j}(2H)\otimes {\cal I}_{Z_j}\bigr|$
consisting of  divisors of the form
\[
2D_j+G_j^1+\cdots+G_j^6+\Phi,\quad 
\Phi \in |\O_{W_j}(2F)|.
\]

\subsubsection{Step (\ref{Kum:pencils})}

Let $\fP'$ be the blow-up of ${\fP}$ at the $240+16$ distinct
points described in the preceding step. The next operation is the
blow-up $\fP'' \to \fP'$ along the $\binom{16}{2}$
pairwise  disjoint respective strict transforms of the pencils
\begin{equation}\label{eq:line}
\bigl|H_0-E_{s'}-E_{s''}\bigr|_{\tilde{S}_0}, \quad
1\leqslant s' < s'' \leqslant 16.
\end{equation}

To describe the exceptional divisor  $\tilde{\fL}_{s's''}$  of $\fP''
\to \fP'$ on the proper transform of \eqref {eq:line},
consider the 3--dimensional linear system 
$\fL_{s's''}:=\bigl|\L_0(-\tQ_{s'} -\tQ_{s''})\bigr| $, 
isomorphic to the projectivization of the kernel of the surjective map
\begin{align}
\label{f_s's''}
\Biggl(\bigoplus_{s\in\{s',s''\}}
\H^0\bigl(\tQ_{s}, \O_{\tQ_{s}}(H)\bigr)
\Biggr) 
\oplus \H^0\bigl(\tilde{S}_0, \O_{\tilde{S}_0}(H_0-E_{s'}-E_{s''})\bigr)
&\rightarrow 
\bigoplus_{s\in\{s',s''\}}
\H^0\bigl(E_s,\O_{E_{s}}(-E_s)\bigr)\\
\notag
(\varsigma',\varsigma'',\varsigma)
&\mapsto
(\varsigma'-\varsigma,
\varsigma''-\varsigma).
\end{align}
Then $\tilde{\fL}_{s's''}$ identifies as the blow--up of $\fL_{s's''}$
along the line defined by  $\varsigma=0$ in the above notation;
in particular it is isomorphic to 
$\P\bigl(\O_{\P^ 1}\oplus \O_{\P^ 1}(1)^ {\oplus 2}\bigr)$,
with $\P^2$--bundle structure
\[
\rho_{s's''}: \tilde{\fL}_{s's''} \to
\bigl|H_0-E_{s'}-E_{s''}\bigr|_{\tilde S_0}
\]
induced by the projection of the left-hand side of \eqref{f_s's''} on
its last summand, as follows from Lemma \ref{l:tg-id}.
The typical element of $\tilde{\fL}_{s's''}$ consists of \\
\begin{inparaenum}[(i)]
\item \label{comp(i)}
a member $C$
of $|H_0-E_{s'}-E_{s''}|_{\tilde S_0}$, plus\\
\item two curves in $|H|_{\tilde Q_s'}$ and $|H|_{\tilde Q_s'}$
  respectively, matching $C$ along $E_{s'}$ and $E_{s''}$,
together with \\
\item rational tails on the $W_j$'s
(two on those $W_j$ meeting neither $\tilde Q_{s'}$ nor $\tilde Q_{s''}$,
one on those $W_j$ meeting exactly one component among
$\tilde Q_{s'}$ and $\tilde Q_{s''}$,
and none on the two $W_j$'s meeting both $\tilde Q_{s'}$ and $\tilde
Q_{s''}$) matching $C$ along $D_j$.
\end{inparaenum}\\
The image by $\rho_{s's''}$ of such a curve is the point corresponding
to its component \eqref{comp(i)}.

\begin{remark} \label{rem:two-nodes}
The image of $\bar{X}$ via the complete linear system
$|\L(-\tilde Q_{s'}-\tilde Q_{s''})|$ provides a model 
$f':S'\to \Delta$ of the initial family $f:S\to \Delta$,
with central fibre the transverse union of two double planes
$\Pi_{s'}$ and $\Pi_{s''}$.
For $s\in \{s',s''\}$,
the plane $\Pi_{s}$ is the projection of $\tilde Q_{s}$ from the
point $p_{\bar s}$ corresponding to the direction of the line
$\ell_{s,s'}$ in 
$|\O_{\tilde Q_{s}}(H)|^\vee \cong |\L_0(-\tilde Q{s})|^\vee$, 
where $\{s,\bar s\}=\{s',s''\}$;
there is a \emph{marked conic} on $\Pi_s$, corresponding to the branch
locus of this projection.
The restriction to $E_s$ of the morphism $\tilde Q_{s}\to \Pi_{s}$ 
is a degree $2$ covering $E_s\to \Pi_{s'}\cap \Pi_{s''}=:L_{s's''}$.
The two marked conics on $\Pi_{s'}$ and $\Pi_{s''}$ intersect at two
points on the line $L_{s's''}$, which are the two branch
points of both the double coverings $E_{s'} \to L_{s's''}$ and 
$E_{s'} \to L_{s's''}$.
These points correspond to the two points cut out on $E_{s'}$
(resp. $E_{s''}$) by the two curves $D_j$ that correspond to the two
double conics of $S_0$ passing through $p_{s'}$ and $p_{s''}$.
There are in addition six distinguished points on $L_{s's''}$,
corresponding to the six pairs of points cut out on $E_{s'}$
(resp. $E_{s''}$) by the six curves $C_{s's''s'''}$ on $\tilde S_0$.

\end{remark}

\subsubsection{Step (\ref{Kum:webs})}

The last operation is the blow-up 
$\fP''' \to \fP''$
along the 16  disjoint surfaces that are the strict transforms of the
$2$-dimensional linear systems
\[
\bigl|H_0-E_s\bigr|_{\tilde{S}_0},\quad
1\leqslant s \leqslant 16.
\]

We want to understand the exceptional divisor $\tilde{\fL}_s$.
Consider the linear system 
$
\fL_s:=\bigl|\L_0(-\tQ_s)\bigr|,
$
which identifies with the 
projectivization of the kernel of the surjective map
\begin{align*}
\mathfrak f_s:
\H^0\bigl(\tQ_{s}, \O_{\tQ_{s}}(H)\bigr)
\oplus \H^0\bigl(\tilde{S}_0,\O_{\tilde{S}_0}(H_0-E_{s})\bigr)
&\rightarrow 
\H^0\bigl(E_{s}, \O_{E_{s}}(-E_s)\bigr)\\
(\varsigma',\varsigma)
&\mapsto
(\varsigma'-\varsigma)
\end{align*}
(itself isomorphic to $\H^0(\tQ_{s}, \O_{\tQ_{s}}(H))$, by the way).
Blow--up $\fL_s$ at the point $\xi$ corresponding to $\varsigma=0$;
one thus gets a $\P^1$--bundle over the plane $|H_0-E_s|_{\tilde S_0}$.
Then $\tilde{\fL}_s$ is obtained by further blowing--up
along the proper transforms  of the lines joining $\xi$ with the 
$6+\left[ \binom{15}{2}-6\binom{5}{2} \right] =51$
points of $|H_0-E_s|$ we blew--up in Step (\ref{Kum:coplanar}).
The typical member of $\tilde\fL_s$ consists of two members of 
$|H_0-E_s|_{\tilde S_0}$ and $|H|_{\tilde Q_s}$ respectively, matching
along $E_s$, together with rational tails on the surfaces $W_j$.

\begin{remark} \label{rem:one-node}
The image of $\bar{X}$ by the complete linear system
$|\L(-\tilde Q_{s})|$ provides a model 
$f':S'\to \Delta$ of the initial family $f:S\to \Delta$,
with central fibre the transverse union of a smooth quadric
$Q$, and a double plane $\Pi$ branched along six lines tangent to the
conic $\Gamma:=\Pi\cap Q$ 
(i.e. the projection of $S_0$ from the node $p_s$).
There are fifteen \emph{marked points} on $\Pi$, namely the
intersection points of the six branch lines of the double covering
$S_0\to \Pi$.
\end{remark}

\subsubsection{Conclusion}

We shall now describe the curves parametrized by the intersections of
the various components of $\fP'''_0$, thus proving:

\begin{proposition}
\label{prop:limlin-kummer}
The central fibre $\fP'''_0$ is the limit linear system of
$|\L_t|=|\O_{\bar X_t}|$ as $t\in \Delta^*$ tends to $0$.
\end{proposition}

\begin{proof}
We analyze step by step
the effect on the central fibre of the birational modifications
operated on $\fP$ in the above construction,
each time using Lemma \ref{l:tg-id} without further notification.

\medskip
\eqref{Kum:coplanar}
At this step, recall (cf. \S\ref{conv}) that $\fP_0 \subset \fP'$
denotes the proper transform of $\fP_0 \subset \fP$ in the blow--up
$\fP'\to \fP$.
For each $\{s',s'',s'''\} \subset \{1,\ldots,16\}$ such that
$\left<p', p'', p'''\right>$ is a plane that does not belong to the
$16_6$ configuration, the intersection $\tilde\fL_{s's''s'''} \cap
\fP_0 \subset \fP'$ is the exceptional $\P^2$ of the blow--up of 
$|\L_0|\cong |\O_{S_0}(1)|$ at the point corresponding to
$H_{s's''s'''}$. Its points, but those lying on one of the three
lines joining two points among $\pi_{s'},\pi_{s''},\pi_{s'''}$ which also have been blown--up 
(the notation is that of Step (\ref{Kum:coplanar}a)), correspond to the trace of
the pull--back of $|\O_{S_0}(1)|$ on $C_{s's''s'''}+E_{s'}+E_{s''}+E_{s'''}$.

For each $j\in \{1,\ldots,16\}$, the intersection $\tilde\fL^j\cap
\fP_0 \subset \fP'$ is a plane, the points of which correspond to
curves 
$2D_j+G_j^1+\cdots+G_j^6+\Phi$ of $\bar X_0$, $\Phi \in
|\O_{W_j}(2F)|$, except for those points on the six lines
corresponding to the cases when $\Phi$ contains one of the six curves
$G_j^1,\ldots,G_j^6$.

\medskip
\eqref{Kum:pencils}
Let $\{s'\neq s''\}\subset \{1,\ldots,16\}$.
The intersection
$\tilde\fL_{s's''}\cap \fP_0\subset \fP''$ is a $\P^1\times \P^1$;
the first factor is isomorphic to the proper transform of the line 
$|H_0-E_{s'}-E_{s''}|_{\tilde S_0}$ in $\fP_0$, 
while the second is isomorphic to the line $\{\varsigma=0\}\subset
\fL_{s's''}$ in the notation of Step \eqref{Kum:pencils} above.
Then the points in $\tilde\fL_{s's''}\cap \fP_0\subset
\fP''$ correspond to curves $C+E_{s'}+E_{s''}$ in $\bar X_0$, with
$C\in |H_0-E_{s'}-E_{s''}|_{\tilde S_0}$, exception made for the points
with second coordinate $[\varsigma_{s'}:0:0]$ or
$[0:\varsigma_{s''}:0]$ in $\fL_{s's''}$, where $\varsigma_{s}\in 
\H^0(\O_{\tilde Q_s}(H))$ vanishes  on $E_s$ for each $s\in
\{s',s''\}$.  

Let $s'''\not\in \{s',s''\}$ be such that $\left<p', p'', p'''\right>$
is a plane outside the $16_6$ configuration.
The intersection $\tilde\fL_{s's''}\cap \tilde\fL_{s's''s'''}\subset
\fP''$ is the $\P^2$ preimage of the point corresponding to
$C_{s's''s'''}$ in $|H_0-E_{s'}-E_{s''}|_{\tilde S_0}$ via $\rho_{s's''}$,
and parametrizes curves
$C_{s's''s'''}+E_{s'''}+C'+C''+ \text{rational tails}$, with 
$C'\in |H|_{\tilde Q_{s'}}$ and $C''\in |H|_{\tilde Q_{s''}}$ matching 
$C_{s's''s'''}$ along $E_{s'}$ and $E_{s''}$ respectively.

On the other hand, for $s'''\not\in \{s',s''\}$ such that 
$\left<p', p'', p'''\right>$ belongs to the $16_6$ configuration, 
let $j\in \{1,\ldots,16\}$ be such that $2D_j$ is cut out on $S_0$ by
$\left<p', p'', p'''\right>$, and set $\tilde Q_{s'}=\tilde Q_j^1$
and $\tilde Q_{s''}=\tilde Q_j^2$;
then $\tilde\fL_{s's''}\cap \tilde\fL_{s's''s'''}\subset \fP''$ is the
preimage by $\rho_{s's''}$ of the point corresponding to $D_j$
in $|H_0-E_{s'}-E_{s''}|_{\tilde S_0}$,
and parametrizes the curves
\[
2D_j+\bigl(G_j^1+C'\bigr)+\bigl(G_j^2+C''\bigr)
+\sum\nolimits_{s=3}^6 \bigl(G_j^s+E_j^s\bigr),
\]
where $C'\in |H-G_j^1|_{\tilde Q_{s'}}$ is the proper
transform by $\tilde Q_{s'}\to Q_{s'}$ of a member of $|H|_{Q_{s'}}$
tangent to $E_{s'}$ at $D_j\cap E_{s'}$, and similarly for $C''$.

\medskip
\eqref{Kum:webs}
Let $s\in \{1,\ldots, 16\}$. 
The intersection $\tilde\fL_s\cap \fP_0\subset \fP'''$ is isomorphic
to the plane $|H_0-E_s|_{\tilde S_0}$ blown--up at the $51$ points
 corresponding to the intersection of at least two lines among the
fifteen $|H_0-E_s-E_{s'}|$, $s'\neq s$.
Each point of the non--exceptional locus of this surface corresponds
to a curve $C+E_s \subset \bar X_0$, with 
$C\in |H_0-E_s|_{\tilde S_0}$.

Let $s'\in \{1,\ldots, 16\}-\{s\}$. 
The intersection $\tilde\fL_s\cap \tilde\fL_{ss'}\subset \fP'''$ is an
$\F_1$, isomorphic to the blow--up at $\xi$ of the plane in $\fL_s$
projectivization of the kernel of the restriction of $\mathfrak f_s$
to $\H^0\bigl(\O_{\tQ_{s}}(H)\bigr)
\oplus \H^0\bigl(\O_{\tilde{S}_0}(H_0-E_{s}-E_{s'})\bigr)$.
It has the structure of a $\P^1$--bundle over $|H_0-E_{s}-E_{s'}|$,
and its points correspond to curves $C+E_{s'}+C_s+ \text{rational
  tails}$, with $C_s\in |H|_{\tilde Q_s}$ matching with $C\in
|H_0-E_{s}-E_{s'}|$ along $E_s$;
note that the points on the exceptional section correspond to the
curves $C+E_{s'}+E_s+ \text{rational tails}$.

Let $s''\in \{1,\ldots, 16\}-\{s,s'\}$, and assume the plane 
$\left<p', p'', p'''\right>$ is outside the $16_6$ configuration.
Then $\tilde\fL_s\cap \tilde\fL_{ss's''}\subset \fP'''$ is a
$\P^1\times \P^1$, the two factors of which are respectively isomorphic
to the 
projectivization of the kernel of the restriction of $\mathfrak f_s$
to $\H^0\bigl(\O_{\tQ_{s}}(H)\bigr)
\oplus \H^0\bigl(\O_{\tilde{S}_0}(H_0-E_{s}-E_{s'}-E_{s''})\bigr)$,
and to the line $\left<\pi_{s'},\pi_{s''}\right>$ in $\fL_{ss's''}$
(with the notations of Step (\ref{Kum:coplanar}b)).
It therefore parametrizes the curves
\[
C_{ss's''}+E_{s'}+E_{s''}+C+\text{rational tails},
\]
where $C\in |H|_{\tilde Q_s}$ matches $C_{ss's''}$ along $E_s$.

Let $j\in \{1,\ldots,16\}$ be such that $W_j$ intersects $\tilde Q_s$,
and set $\tilde Q_j^1=\tilde Q_s$.
Then $\tilde\fL_s\cap \tilde\fL^j\subset \fP'''$ is a $\P^1\times
\P^1$, the two factors of which are respectively isomorphic to the
pencil of pull--backs to $\tilde Q_s$ of members of $|H|_{Q_s}$
tangent to $E_s$ at the point $D_j\cap E_s$, and to the subpencil 
$2D_j+2G_j^1+G_j^2+\cdots+G_j^6+|F|_{W_j}$ of $\fL^j$.
It parametrizes  curves
\[
2D_j+\bigl(G_j^1+C\bigr)
+\sum\nolimits_{s=2}^6 \bigl(G_j^s+E_j^s\bigr),
\]
where $C\in |H-G_j^1|_{\tilde Q_{s}}$ is the proper transform of a
curve on $Q_s$ tangent to $E_{s}$ at $D_j\cap E_{s}$.

\medskip
It follows from the above analysis that the points of $\fP'''_0$ all
correspond in a canonical way to curves on $\bar X_0$,
which implies our assertion by Lemma \ref{l:lim-lin}.
\end{proof}

\subsection{The linear system $\fL^ j$}

In this section, we study the 
rational map $\varphi_j$ (or simply $\varphi$) determined by the
linear system $\fL^ j=\bigl|\O_{W_j}(2H)\otimes {\cal I}_{Z_j}\bigr|$ 
on $W_j$, for $1\leqslant j \leqslant 16$. 

Let $u_j:\bar{W}_j\to W_j$ be the blow--up at the twelve
points in the support of ${Z_j}$.
For $1\leqslant s\leqslant 6$, we denote by $\hat{G}_j^s$ the
strict transform of the ruling $G_j^s$, and by ${I_j^s}',{I_j^s}''$
the two exceptional curves of $u_j$ meeting $\hat{G}_j^s$. 
Then the pull--back via $u_j$ induces a natural isomorphism
 \[
 \bigl|\O_{W_j}(2H)\otimes {\cal I}_{Z_j}\bigr|\cong 
\Bigl|\O_{\bar W_j}\bigl(2H-\sum_{s=1}^ 6 {(I_j^s}'+{I_j^s}'')
\bigr) \Bigr|;
 \]
we denote by $\bar \fL^ j$ the right hand side linear system.
 
\begin{proposition}
\label{p:2:1}
The linear system $\bar \fL^ j$ determines a 
 $2:1$ morphism
\[
\bar \varphi:  \bar{W}_j \to \Sigma \subset \P^3,
\]
where $\Sigma$ is a quadric cone.
The divisor
$\tilde{D}_j:=D_j+\hat{F}_j^1+\cdots+\hat{F}_j^6$ is contracted by
$\bar \varphi$ to the vertex of $\Sigma$.
The branch curve $B$ of $\bar \varphi$ is irreducible, cut out on
$\Sigma$ by a quartic surface; it is rational, with an ordinary six--fold point at the vertex
of $\Sigma$.
\end{proposition}

Before the proof, let us point out the
following corollary, which we will later need.

\begin{corollary}
\label{c:2:1}
The Severi variety of irreducible $\delta$--nodal  curves in 
$\bigl|\O_{W_j}(2H)\otimes {\cal I}_{Z_j} \bigr|$ is isomorphic
to the subvariety of $\check{\P}^3$ parametrizing 
$\delta$--tangent planes to $B$, for $\delta=1, \ldots, 3$.
They have degree $14$, $60$, and $80$, respectively.
\end{corollary}

For the proof of Proposition \ref {p:2:1} we need two preliminary lemmas.

\begin{lemma}
\label{lem:dim-2Sigma}
The linear system $|\O_{W_j}(2H)\otimes\mathcal{I}_{Z_j}|\subset
|2H_{W_j}|$ has dimension $3$.
\end{lemma}

\begin{proof}
The 0-cycle  ${Z_j}$ is cut out on
$G_j^1+\cdots+G_j^6$ by a general curve in $|2H|$.
Let then
\[\sigma\in \bigoplus_{s=1}^ 6 \H^0(G_j^ s, \O_{G_j^ s}(2H))\cong 
\H^ 0(\P^ 1, \O_{\P^ 1}(2))^ {\oplus 6}\]
be a non--sero section vanishing at ${Z_j}$. 
Then 
$ \H^ 0(W_j, \O_{W_j}(2H)\otimes\mathcal{I}_{Z_j})\cong r^ {-1}(\langle \sigma\rangle)$
where 
\[r:\H^0(W_j, \O_{W_j}(2H)) \to \bigoplus_{s=1}^ 6 
\H^0(G_j^ s, \O_{G_j^ s}(2H))\cong \H^ 0(\P^ 1, \O_{\P^ 1}(2))^ {\oplus 6}\]
is the restriction map. The assertion now follows from the restriction
exact sequence, since 
\[  \h^ 0(W_j, \O_{W_j}(2H)\otimes\mathcal{I}_{Z_j})=
1+\h^ 0(W_j, \O_{W_j}(2H-6F))=4.\]
\end{proof}

\begin{lemma}
\label{lem:2Sigma-deg2}
The rational map $\varphi_j$
has degree $2$ onto its image, and
its restriction to any  line of the ruling $|F_{W_j}|$ but the six
$G^s_j$, $1\leqslant s\leqslant 6$, has degree $2$ as 
well.
\end{lemma}

\begin{proof}
Let $x \in W_j$ be a  general point 
and let $F_x$ be the  line of the ruling containing $x$.
One can find a divisor 
$D\in |\O_{W_j}(2H)\otimes \mathcal{I}_{Z_j}|$ containing $x$
but not containing $F_x$. Let $x+x'$ be the length  two scheme cut out by 
$D$ on $F_x$.
By an argument similar to the one in the proof of Lemma
\ref{lem:dim-2Sigma}, one has 
$\dim \bigl( |\O_{W_j}(2H)\otimes \mathcal{I}_{Z_j} \otimes {\cal
  I}_{x+x'}| \bigr)=2$.
This shows that $x$ and $x'$ are mapped to the same
point by $\varphi$.
Then, considering the sublinear system 
\[
2D_j+G_j^1+\cdots+G_j^6+F_x+\Phi,\quad 
\Phi \in |\O_{W_j}(F)|,
\]
of $\fL^ j$, with fixed divisor $2D_j+G_j^1+\cdots+G_j^6+F_x$,
the assertion follows from the base point freeness of
$|\O_{W_j}(F)|$.
\end{proof}

\begin{proof}[Proof of Proposition \ref{p:2:1}]
First we prove that $\fL^j$  has no fixed components, hence that the
same holds for  $\bar \fL^j$.
Suppose $\Phi$ is such  a fixed component. By Lemma \ref{lem:2Sigma-deg2}, $\Phi\cdot F=0$, hence $\Phi$ should consist of 
curves contained in rulings. The argument of the proof of Lemma \ref {lem:dim-2Sigma} shows that no such a curve may  occur in $\Phi$, a contradiction.   

Let $D\in \bar \fL^j$ be a general element. By Lemmas   \ref{lem:dim-2Sigma} and \ref {lem:2Sigma-deg2}, $D$ is irreducible and hyperelliptic, since $D\cdot F=2$. 
Moreover $D^ 2=4$ and $p_a(D)=3$. This implies that $D$ is smooth and that  $\bar \fL^j$ is base point free.  Moreover the image $\Sigma$ of $\varphi$ has degree 2.  Since 
$D\cdot \tilde{D}_j=0$ and $\tilde D_j^ 2=-4$, the connected divisor
$\tilde D_j$ is contracted to a double point $v$ of $\Sigma$, which is
therefore a cone.

Since $D$ is mapped $2:1$ to a general plane section of $\Sigma$,
which is a conic, we see that $\deg (B)=8$.  
Let $\Phi \in |F|_{W_j}$ be general, and $\ell$ its image via
$\varphi$, which is a ruling of $\Sigma$. 
The restriction $\left.\varphi\right|_\Phi: \Phi\to \ell$ is a degree
2 morphism, which is  ramified at the intersection point of $\Phi$
with $D_j$. 
This implies that  $\ell$ meets $B$ at one single point off the vertex
$v$ of $\Sigma$. Hence $B$ has a unique irreducible component $B_0$ which
meets the general ruling $\ell$ in one point off $v$. We claim that
$B=B_0$. If not, $B-B_0$ consists of rulings $\ell_1,\ldots, \ell_n$,
corresponding to rulings $F_1,\ldots, F_n$, clearly all  different
from the $G_j^ s$, with $1\leqslant s\leqslant 6$. Then the
restrictions $\left.\varphi\right|_{F_i}: F_i\to \ell_i$ would be
isomorphisms, for $1\leqslant i\leqslant n$, which is clearly
impossible.  Hence $B$ is irreducible, rational, sits in $\vert
\O_\Sigma(4)\vert$. 
Finally, taking a  plane section of $\Sigma$ consisting of two general rulings, we see that
it has only two intersection points with $B$ off $v$. Hence $B$ has a point of multipilicity 6 at $v$ and the assertion follows. \end{proof}

\begin{remark}\label{rem:singular} Each of the curves
  $\hat{G}_j^s+{I_j^s}'+{I_j^s}'' \in |F|_{\bar W_j}$, for $1\leqslant
  s\leqslant 6$, is mapped by $\bar \varphi$ to a ruling $\ell_s$ of
  $\Sigma$, and this ruling has no intersection point with $B$ off
  $v$. This implies that $v$ is an ordinary 6--tuple point for $B$ and
  that the tangent cone to $B$ at $v$ consists of the  rulings
  $\ell_1,\ldots, \ell_6$ of $\Sigma$. 
\end{remark}

\begin{remark} \label{rem:2-1cone}
Let $S'\to \Delta$ be the image of $\bar{X} \to \Delta$
via the map defined by the linear system 
$|\L(-2W_j-\sum_s \tQ_j^s)|$. One has $S'_t\cong S_{t^2}$ for $t\neq 0$,
and the new central fibre $S'_0$ is  a double quadratic cone
$\Sigma$ in $\P^3$.
\end{remark}

\subsection{The limit Severi varieties} \label {s:kummer-concl}

In this section we describe the regular components of the limit Severi
varieties $\fV_{1,\delta}(\bar X)$ for $1\leqslant \delta\leqslant 3$.
The discussion here parallels the one in \S \ref 
{s:tetra-concl}, therefore we will be sketchy, leaving to the reader most of the straightforward verifications, based on the description of
the limit linear system in \S \ref {s:kummer-ls}.

\begin{proposition} [Limits of $1$-nodal curves]
\label{p:1-kummer}
The regular components of the limit Severi variety \linebreak
$\fV_{1,1}(\bar X)$ are the following 
(they all appear with multiplicity $1$, but the ones in
\eqref{it:double} which appear with multiplicity $2$):\\ 
\begin{inparaenum}[\normalfont (i)]
\item \label{it:kummer}
$V\bigl(\delta_{\tilde{S}_0}=1\bigr)$, which is
isomorphic to the Kummer quartic surface $\check{S}_0 \subset
|\O_{S_0}(1)| \cong \check{\P}^3$;\\
\item \label{it:quadrics}
$V\bigl(\tQ_s,\delta_{\tQ_s}=1\bigr)$, which is
isomorphic to the smooth quadric $\check{Q}_s \subset |\O_{Q_s}(1)|
\cong \check{\P}^3$, for $1\leqslant s\leqslant 16$;\\
\item \label {it:double} $V(\tilde Q_s, \tau_{E_s,2}=1)$, which is
  isomorphic to a quadric cone in 
$|\O_{Q_s}(1)|$, for $1\leqslant s\leqslant 16$;\\
\item
$V\bigl(\tQ_{s'}+\tQ_{s''},\delta_{\tQ_{s'}}=1\bigr)$, which is  isomorphic to $\check{Q}_s \subset |\O_{Q_s}(1)|
\cong \check{\P}^3$, for $1 \leqslant s'< s'' \leqslant 16$;\\
\item
$V\bigl(\tQ_{s'}+\tQ_{s''}+\tQ_{s'''},
\delta_{\tQ_{s'}}=1 \bigr)$, for
$1 \leqslant s',s'',s''' \leqslant 16$
such that $\tQ_{s'},\tQ_{s''},\tQ_{s'''}$ are pairwise distinct and do
not meet a common $W_j$: 
it is again isomorphic to $\check{Q}_s \subset |\O_{Q_s}(1)|
\cong \check{\P}^3$;\\
\item \label{it:doubleconics}
$V\bigl(2W_j+\tQ^1_j+\cdots +\tQ^6_j,
\delta_{W_j}=1 \bigr)$,  which is isomorphic  to the
degree $14$ surface $\check{B} \subset |\O_B(1)|\cong
\check{\P}^3$, for $1\leqslant j\leqslant 16$.\\
\end{inparaenum}
\end{proposition}

\begin {corollary} [Theorem \ref{T:kummer} 
for $\delta=1$]
\label{cor:thmB2i} 
The family $f:S\to \Delta$ of general quartic surfaces
degenerating to a Kummer surface $S_0$ we started with, with smooth
total space $S$, and endowed with the line bundle $\O_S(1)$,
is $1$--well behaved, with good model $\varpi: \bar
X\to \Delta$. 
The limit in $|\O_{S_0}(1)|$ of the dual surfaces
$\check S_t$, $t\in\Delta^*$, consists in the union of the dual
$\check S_0$ of $S_0$ (which is again a Kummer surface),
plus the $16$ planes of the $16_6$ configuration of
$\check S_0$, each counted with multiplicity $2$.
\end{corollary}

\begin{proof}
The push--forward by the morphism $\fP'''_0\to \fP_0 \cong
|\O_{S_0}(1)|$ of the regular 
components of $\fV_{1,1}$ with their respective multiplicities in 
$\reg \fV_{1,1}$
is $\check S_0$ in case \eqref{it:kummer}, 
$2\cdot \check {p}_s$ in case \eqref{it:quadrics}, and $0$ otherwise.
The push-forward of $\reg \fV_{1,1}(\bar X)$ has thus total degree $36$, and
is therefore the crude limit Severi variety
$\cru\fV_{1,1}(S)$ by Proposition
\ref{p:deg-dual}. 
\end{proof}

\begin{remark}\label{rem:dual2} 
\begin{inparaenum}[\bf (a)]
\item
Similar arguments show that $\varpi:\bar X\to \Delta$ is a $1$--good
model for the degenerations of general quartic surfaces obtained 
from $\bar X\to \Delta$ via
the line bundles $\L(-2W_j-\tQ^1_j-\cdots-\tQ^6_j)$ and $\L(-\tilde
Q_s)$ respectively (see Remarks \ref{rem:2-1cone} and
\ref{rem:one-node} for a description of these degenerations). 

To see this in the former case, 
let us consider two general points on a given
$W_j$, and enumerate the regular members of $\fV_{1,1}$ that contain
them.
There are $2$ curves in  \eqref{it:kummer}
(indeed, the two points on $W_j$ project to two general points on
$D_j\cong \Gamma_j\subset S_0 \subset \P^3$, which span a line $\ell
\subset \check{\P}^3$; the limiting curves in ${S}_0$ passing through the
two original  points on $W_j$ correspond to the intersection points of
$\check \ell$ with $\check{S}_0$;
now $\check \ell$ meets $\check{S}_0$ with
multiplicity $2$ at the double point which is the  image of $\Gamma_j$
via the Gauss map, and only the two remaining intersection points are
relevant). 
There are in addition $2$ limiting curves in each of the $10$ components
of type \eqref{it:quadrics} corresponding to the $\tQ_s$'s that do not
meet $W_j$, and $14$ in the relevant component of type
\eqref{it:doubleconics}. 

In this case, the crude limit Severi variety therefore consists, in
the notation of Remark \ref{rem:2-1cone}, of the
degree $14$ surface $\check{B}$, 
plus the plane $\check v$ with multiplicity $22$ 
(this has degree $36$ as required).

For the degeneration given by $\L(-\tilde Q_s)$, 
the crude limit Severi variety consists, in
the notation of Remark \ref{rem:one-node},
of the dual to the smooth quadric $Q$, plus the dual to the conic
$\Gamma$ with multiplicity $2$, plus the fifteen planes $\check p$
with multiplicity $2$, where $p$ ranges among the fifteen marked
points on the double plane $\Pi$.

\item
\label{rem:not-good} One can see that  $\varpi:\bar X\to \Delta$ is  not a $1$--good
model for the degeneration
to a union of two double planes obtained via the line bundle
$\L(-\tilde Q_{s'}-\tilde Q_{s''})$ described in Remark
\ref{rem:two-nodes}.
In addition (see Step (\ref{Kum:coplanar}a))
 the line bundles
$\L(-\tilde Q_{s'}-\tilde Q_{s''}-\tilde Q_{s'''})$, though
corresponding to $3$--dimensional components of the limit linear
system, do not provide suitable degenerations of surfaces.
Despite all this,  it seems plausible that one can obtain a good model by
making further modifications of $\bar X\to \Delta$.
The first thing to do would  be to blow--up the curves
$C_{s's''s'''}$.
\end{inparaenum}
\end{remark}

\begin{proposition} [Limits of $2$-nodal curves]
\label{p:2-kummer}
The regular components of the limit Severi variety \linebreak
$\fV_{1,2}(\bar X)$ are the following
(they all appear with multiplicity $1$, except the ones in  \eqref
{it:2quadrics2} appearing with multiplicity $2$):\\ 
\begin{inparaenum}[\normalfont (i)]
\item \label{it:2quadrics}
$V\bigl(\tQ_{s'}+\tQ_{s''},\delta_{\tQ_{s'}}=\delta_{\tQ_{s''}}=1
\bigr)$ for $s'\neq s''$, proper transform of the intersection of two
smooth quadrics in $\fL_{s's''}$;\\
\item \label{it:2quadrics2}
$V\bigl(\tQ_{s'}+\tQ_{s''},\delta_{\tQ_{s'}}=1, \tau_{E_{s''},2}=1
\bigr)$ for $s'\neq s''$, proper transform of the intersection of a
smooth quadric and a quadric cone in $\fL_{s's''}$;\\
\item \label{it:2quadrics-spe} 
$V\bigl(\tQ_{s'}+\tQ_{s''}+\tQ_{s'''},
\delta_{\tQ_{s'}}=\delta_{\tQ_{s''}}=1 \bigr)$ for
$1 \leqslant s',s'',s''' \leqslant 16$
such that $\tQ_{s'},\tQ_{s''},\tQ_{s'''}$ are pairwise distinct and do
not meet a common $W_j$,
proper transform of the intersection of two
smooth quadrics in $\fL_{s's''s'''}$;\\ 
\item \label{it:2conic-2}
$V\bigl(2W_j+\tQ^1_j+\cdots +\tQ^6_j,\delta_{W_j}=2 \bigr)$
for each $j\in\{1,\ldots, 16\}$,
proper transform of a degree $60$ curve in $\fL^j$.\\
\end{inparaenum}
\end{proposition}

\begin{proof}
Again, one checks that the components listed in the above statement
are the 
only ones provided by Proposition \ref{p:zrr},
taking the following points into account:\\
\begin{inparaenum}[(a)]
\item 
the condition $\delta_{\tilde{S}_0}=2$ is impossible to fulfil,
because there is no plane of $\P^3$ tangent to $S_0$ at exactly
two points (see Proposition \ref{p:dual-kummer});\\
\item 
the condition $\delta_{\tilde{S}_0}=\delta_{\tQ_i}=1$ is also impossible to
fulfil, because there is no plane in $\P^3$ tangent to
$S_0$ at exactly one point, and passing through one of its double
points.  Indeeed, let $p_i$ be a  double point of $S_0$,  the dual plane  $\check p_i$  is everywhere tangent to
$\check{S}_0$ along the contact conic Gauss image of $E_i$;\\
\item
the condition $\delta_{\tilde Q_s}=\tau_{E_s,2}=1$ imposes to a member
of $|H|_{\tilde Q_s}$ to be the sum of two rulings intersecting at a
point on $E_s$, and such a curve does not belong to the limit Severi
variety:\\
\item
the condition 
$\tau_{E_{s'},2}=\tau_{E_{s''},2}=1$
imposes to contain one of the two curves $D_j$ intersecting both
$E_{s'}$ and $E_{s''}$,
which violates condition \eqref{no-double-curve} of Definition
\ref{d:zrr}.
\end{inparaenum}
\end{proof}

\begin{remark}\label{rem:480_2} As in Remark \ref {rem:480},
we can enumerate the $480$ limits of $2$--nodal curves
passing through a general point in certain irreducible components of
$\bar{X}_0$:\\
\begin{inparaenum}[(a)]
\item \label{case:kummer-2}
for a general point on $\tilde{S}_0$, we find $4$
  limit curves in each of the $\binom{16}{2}=120$ 
components  in  \eqref{it:2quadrics} of
Proposition \ref{p:2-kummer};\\
\item \label{case:2conic-2}
 for a general point on a given $W_j$, we find
 $60$ limit curves in the appropriate component in  \eqref{it:2conic-2},
and $4$ in each of the $\binom{16}{2}-\binom{6}{2}=105$ different 
components of type \eqref{it:2quadrics} such that $\tilde Q_{s'}$ and
$\tilde Q_{s''}$ do not both meet $W_j$.
\end{inparaenum}

This shows that $\bar X\to \Delta$ is a $2$--good model for the
degenerations of quartics corresponding to the line bundles $\L$ and 
$\L(-2W_j-\tQ^1_j-\cdots-\tQ^6_j)$.
In particular, it implies Corollary \ref{cor:thmB2ii} below.
\end{remark}

\begin {corollary} [Theorem \ref{T:kummer} for $\delta=2$]
\label{cor:thmB2ii}
Same setting as in Corollary \ref {cor:thmB2i}. 
The crude limit Severi variety $\cru \fV_{1,2}(S)$ consists of the
images in $\vert \O_{S_0}(1)\vert$  
of the $120$ irreducible curves  listed in case \eqref{case:kummer-2} of Remark \ref {rem:480_2}.
Each of them projects $4:1$ onto a pencil of planes containing two
double points of $S_0$.\end{corollary}

\begin{proposition} [Limits of $3$-nodal curves]
\label{p:3-4planes_Kummer}
The family $\bar X\to \Delta$ is absolutely $3$--good, and the limit
Severi variety $\fV_{1,3}$ is reduced, consisting of:\\
\begin{inparaenum}[\normalfont (i)]
\item \label {item:aa}
$8$ distinct points in each 
$V\bigl(-\tQ_{s'}-\tQ_{s''}-\tQ_{s'''},
\delta_{\tQ_{s'}}=\delta_{\tQ_{s''}}=\delta_{\tQ_{s'''}}=1 \bigr)$, 
where $1 \leqslant s',s'',s''' \leqslant 16$ are such that 
$\left<p_{s'},p_{s''},p_{s'''} \right>$ is a plane that does not
belong to the $16_6$ configuration of $S_0$;\\
\item
the $80$ distinct points in each
$V\bigl(2W_j+\tQ^1_j+\cdots +\tQ^6_j,\delta_{W_j}=3 \bigr)$, 
 corresponding to the triple points of the double curve of $\check{B}
\subset |\O_B(1)|\cong \check{\P}^3$ that are also triple points of
$\check{B}$.
\end{inparaenum}
\end{proposition}

\begin{proof}
There are $240$ unordered triples $\{s',s'',s'''\}$
such that the corresponding double points of $S_0$ do not lie on a
common $D_j$,
so $\reg \fV_{1,3}$ has degree $3200$, which fits  with Proposition
\ref {p:deg-dual}.
\end{proof}

\begin{corollary} [Theorem \ref{T:kummer} for $\delta=3$]
\label{cor:thmB2iii} 
Same setting as in Corollary \ref {cor:thmB2i}. 
The crude limit Severi variety $\cru \fV_{1,3}(S) \subset |\O_{S_0}(1)|$
consists of:\\
\begin{inparaenum}[\normalfont (i)]
\item
the $240$ points corresponding to a plane through three nodes of
$S_0$, but not member of the $16_6$ configuration, each counted with
multiplicity $8$;\\
\item
the $16$ points corresponding to a member of the $16_6$ configuration,
each counted with multiplicity $80$.
\end{inparaenum}
\end{corollary}

\section{Plane quartics curves through 
 points in special position}
\label{S:triangle}

In this section we prove the key result needed for the proof of
Theorem \ref{T:triangle}, itself given in \S\ref{s:tetra-concl}. 
We believe this result, independently predicted
with tropical methods by E.~Brugall\'e and G.~Mikhalkin (private
communication), is interesting on its own.
Its proof shows once again the usefulness of constructing (relative)
good models. 

The general framework is the same as that of \S\ref{S:4planes} and
\S\ref{S:Kummer-degen}, and we are going to be sketchy here.

\subsection{The degeneration and its good model} 

We start with the trivial family $f:S:=\P^2\times\Delta \to\Delta$,
together with flatly varying data for $t\in\Delta$
of three independent lines
$a_t,b_t,c_t$ lying in $S_t$, and of a $0$-dimensional scheme $Z_t$ of
degree $12$ cut out on $a_t+b_t+c_t$ by a quartic curve $\Gamma_t$ in
$S_t$, which is general for $t\in \Delta^*$.
We denote by $\O_S(1)$ the pull--back line bundle of 
$\O_{\P^ 2}(1)$ via the projection $S\to \P^ 2$.

We  blow--up $S$ along the line $c_0$. This produces 
a new family $Y \to \Delta$, the central fibre $Y_0$ of which
is the transverse union of
a plane $P$ (the proper transform of $S_0$, which we may identify with
$P$) and of an $\F_1$ surface $W$ (the exceptional divisor).
The curve $E:=P\cap W$ is the line $c_0$ in $P$, and
the $(-1)$-- section in $W$.
The limit on $Y_0$ of the three lines 
$a_t,b_t, c_t$ on the fibre $Y_t\cong \P^ 2$, for $t\in \Delta^ *$,
consists of:\\
\begin{inparaenum}[(i)]
\item two general lines $a,b$ in $P$ plus  the curves 
$a',b'\in |F|_W$ matching them on $E$;\\
\item a curve $c\in |H|_W=|F+E|_W$  on $W$.
\end{inparaenum}

We denote by $\O_Y(1)$ the pull--back of $\O_S(1)$ and we set
$\L^\natural=\O_Y(4)\otimes \O_Y(-W)$. 
One has $\L^\natural_t \cong \O_{\P^ 2}(4)$ for $t\in \Delta^ *$,
whereas $\L^\natural_0$ restrict to $\O_{P}(3H)$ and 
$\O_{W}(4F+E)\cong \O_{W}(4H-3E)$ respectively.
We may assume that the quartic curve 
$\Gamma_t\in \vert \L^\natural_t\vert$ cutting $Z_t$ on $a_t+b_t+c_t$
for $t\in \Delta^ *$ tends, for $t\to 0$,  to a general curve $\Gamma_0\in \vert \L^\natural_0\vert$. 
Then $\Gamma_0=\Gamma_P+\Gamma_W$, where $\Gamma_P$ is a general cubic
in $P$ and $\Gamma_Q\in  \vert 4H-3E \vert_W$, with $\Gamma_P$ and
$\Gamma_W$ matching along $E$. 
Accordigly  $Z_0=Z_P+Z_W$, where $Z_P$ has length 6 consisting of 
$3$ points on $a$ and 3 on $b$, and $Z_W$ consists of $1$ point on both
$a'$ and $b'$, and $4$ points on $c$ (see Figure \ref{f:triangle}).

\begin{figure}
\begin{center}
\includegraphics[height=3cm]{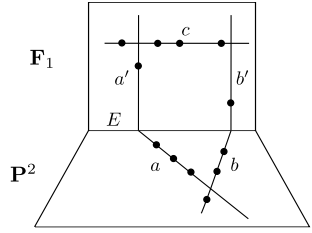}
\end{center}
\caption{Degeneration of base points on a triangle}
\label{f:triangle}
\end{figure}

Next we consider the blow--up 
$\epsilon: X\to Y$ along the curve $Z$ in $Y$ described by $Z_t$, 
for $t\in \Delta$, and
thus obtain a new family $\pi: X\to \Delta$, 
where each $X_t$ is the blow--up
of $Y_t$ along $Z_t$.
We call $E_Z$ the exceptional divisor of $\epsilon$.
The fibre of $\left.\epsilon\right|_{E_Z }: E_Z\to \Delta$ at 
$t\in \Delta$ consists of the twelve $(-1)$--curves of the 
blow--up of $Y_t$ at $Z_t$. The central fibre $X_0$ is the transverse 
union of $\tilde P$ and $\tilde W$,  respectively the blow--ups of $P$
and $W$ along $Z_P$ and $Z_W$; we denote by $E_P$ and $E_W$ the 
corresponding exceptional divisors. 

We let $\L:= \epsilon^* \L^\natural \otimes \O_X(-E_Z)$.
Recall from \S\ref{s:planes} that the fibre of
$\P(\pi_*(\L))$ over $t\in
\Delta^*$ has dimension $3$.
We will see that $X\to \Delta$, endowed with $\L$, is well behaved and
we will describe 
the crude limit Severi variety $\cru \fV_\delta$
for $1\leqslant \delta\leqslant 3$. This analysis will prove Theorem 
\ref{T:triangle}.

\begin{remark}\label{rem:central}
We shall need a detailed description of the linear system 
$\vert\L_0\vert$.  
The vector space  
$\H^ 0(X_0, \L_0)$ is the subspace of
$ \H^ 0(\tilde W, \O_{\tilde W}
(4H-3E-E_W))\times  \H^ 0(\tilde P, \O_{\tilde P}(3H-E_P))$ which is the
fibred product corresponding to the Cartesian diagram
\begin{equation}\label{eq:squarebis}
\xymatrix@=15pt{
\H^ 0(X_0, \L_0)   \ar[d]  \ar[rr] 
&& \H^ 0(\tilde P, \O_{\tilde P}(3H-E_P)) \ar[d]^ {r_P} \\
 \H^ 0(\tilde W, \O_{\tilde W} (4H-3E-E_W)) \ar[rr]_{r_W} 
&& \H^ 0(E,\left .\L \right|_{E})\cong \H^ 0(\P^ 1, \O_{\P^ 1}(3))
}\end{equation}
where $r_P, r_W$ are restriction maps.  The map $r_W$ is injective,
whereas $r_P$ has  a 1-dimensional kernel generated by a section $s$
vanishing on the proper transforms of $a+b+c$. Since 
$\h^ 0(X_0,\L_0)\geqslant 4$ by semicontinuity, one has
${\rm Im}(r_P)={\rm Im}(r_W)$, and therefore 
$\H^ 0(X_0, \L_0)\cong \H^ 0(\tilde P, \O_{\tilde P}(3H-E_P)) $ has
also dimension 4. Geometrically, for a general curve $C_P\in \vert
3H-E_P\vert$, there is a unique curve $C_W\in \vert 4H-3E-E_W\vert$
matching it along $E$ and $C_P+C_W\in \vert \L_0\vert$. On the other
hand  $(0,s)\in \H^ 0(X_0, \L_0)$ is the only non--trivial section (up
to a constant)  identically vanishing on a component of the central
fibre (namely $\tilde W$),  
and $\H^ 0(X_0, \L_0) /(s)\cong \H^ 0(\tilde W, \O_{\tilde W} (4H-3E-E_W))$. 
Therefore, if we denote by $D$ the  point corresponding to $(0,s)$  in $\vert \L_0\vert$, a line through $D$ parametrizes the pencil
consisting of a fixed divisor in $\vert 4H-3E-E_W\vert$ on $\tilde W$ plus all divisors in $\vert 3H-E_P\vert$ matching it on $E$. 

We will denote by $\mathfrak R$  the $g^ 2_3$ on $E$ given by $\vert {\rm Im}(r_P)\vert =\vert {\rm Im}(r_W)\vert$.
\end{remark}

To get a good model, we first  blow--up the proper transform of $a$ in
$\tilde P$, and then we blow--up the proper transform of $b$ on the
strict transform of $\tilde P$. We thus obtain a new family 
$\varpi: \bar X\to \Delta$. 
The general fibre $\bar X_t$, $t\in \Delta^ *$, is
isomorphic to $X_t$. The central fibre $\bar X_0$ has four
components (see Figure \ref{f:triangle2}):\\ 
\begin{inparaenum}[(i)]
\item the proper transform of $\tilde P$, which is isomorphic to $\tilde P$;\\
\item the proper transform $\bar W$ of $\tilde W$, which is isomorphic to the blow--up of $\tilde W$ at the two points $a\cap E, b\cap E$, with exceptional divisors $E_a$ and $E_b$; \\
\item the exceptional divisor $W_b$ of the last blow--up, which is isomorphic to $\F_0$;\\
\item the proper transform $W_a$ of the exceptional divisor over $a$, which is the blow--up of an $\F_0$--surface, at the point corresponding to $a\cap b$ (which is a general point of $\F_0$) with exceptional divisor $E_{ab}$.
\end{inparaenum}

As usual, we go on calling $\L$ the pull--back to $\bar X$ of the line
bundle $\L$ on $X$.

\begin{figure}
\begin{center}
\includegraphics[height=5cm]{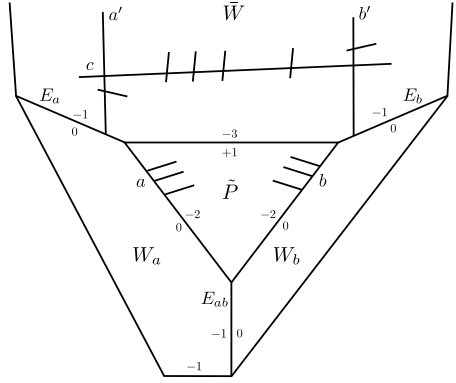}
\end{center}
\caption{Good model for plane quartics through twelve points}
\label{f:triangle2}
\end{figure}

\subsection{The limit linear system}

We shall now describe the limit linear system $\fL$ associated to $\L$. 
As usual,  we start with  ${\fP} := \P(\varpi_*(\L))$, and we
consider the blow--up ${\fP}'\to {\fP}$ at the point $D\in \fP_0\cong
|\L_0|$. The central fibre  of ${\fP}'\to \Delta$ is, as we 
will see, the limit linear system $\fL$. It consists of only two
components: 
the proper transform $\fL_1$ of $\vert \L_0\vert$ and the exceptional
divisor $\fL_2\cong \P^ 3$. Let us describe these two components in
terms of twisted linear systems on the central fibre.

Since the  map $r_W$ in \eqref {eq:squarebis} is injective, it is
clear that 
$\fL_2 \cong \vert \L_0(-\bar W-W_a-W_b)\vert$.
The line bundle $ \L_0(-\bar W-W_a-W_b)$ is
 trivial on $\tilde P$ and restricts  to  
$\O_{\bar W}(4H-2E-3E_a-3E_b-E_W),\O_{W_a}(H-E_{ab}),\O_{W_b}(H)$ on $\bar W,W_a,W_b$
respectively. Once chosen $C_W\in \bigl|\O_{\bar W}(4H-2E-3E_a-3E_b-E_W)\bigr|$, there is only one possible choice of two curves $C_a$ and $C_b$ in 
$\bigl|\O_{W_a}(H-E_{ab})\bigr|$ and $\bigl|\O_{W_b}(H)\bigr|$ respectively,
that match with $C_W$ along $E_a, E_b$ respectively. They
automatically match along $E_{ab}$.

In conclusion, by mapping $\bar W$ to $\P^ 2$ (via $|H|_{\bar W}$),
we have:

\begin{proposition}\label{propo:descrlim} 
The component $\fL_2\cong \vert \L_0(-\bar W-W_a-W_b)\vert$ 
of $\fP'_0$ is
isomorphic to a $3$-dimensional linear system of plane
quartics with an imposed double point $x$,
prescribed tangent lines $t_1,t_2$ at $x$, and
six further base points, two of which general on $t_1,t_2$ and the
remaining four on a general line. 
\end{proposition}

To identify $\fL_1$ as the blow--up of $\vert \L_0\vert$ at $D$, we
take into account Lemma \ref {l:tg-id}, which tells us that the
exceptional divisor 
${\mathfrak E}\subset \fL_1$ identifies with $\mathfrak R$.
Since ${\mathfrak E}=\fL_1\cap \fL_2$, the linear system
${\mathfrak E}$ identifies with a sublinear system of codimension $1$ in
$\fL_2$, namely that of curves
\[
a+b+E+C,\quad 
C\in \bigl| 4H-3E-3E_a-3E_b-E_W \bigr|_W
\]
(in the setting of Proposition \ref{propo:descrlim}, $C$ corresponds
to a quartic plane curve with a triple point at $x$ passing through
the six simple base points).

It follows from this analysis that $\fL$ is the limit as $t\to 0$ of
the linear systems $|\L_t|$, $t\in \Delta^*$, in the sense of
\S\ref{s:resolving}.

\subsection{The limit Severi varieties}
\label{s:lim-triangle}

We will use the notion of ${\bf n}$--degree introduced in Definition
\ref {def:enne}. However we will restrict our attention to the case
in which we fix 1 or 2 points only on $\tilde P$.  Hence, if we agree
to set $\tilde P=Q_1$, then we call $P$--degree
of a component $V$ of $\fV_\delta$ its ${\bf n}$--degree with
${\bf n}=(3-\delta,0,0,0)$; we denote it by $\deg_P( V)$.

\begin{proposition} [Limits of $1$-nodal curves]
\label{p:1-triangle} 
The regular components of the limit Severi variety \linebreak
$\fV_1(\bar X,\L)$
are the following, all appearing with multiplicity $1$, except
\eqref {item:3}, which has multiplicity $2$:\\ 
\begin{inparaenum}[\normalfont (i)]
\item  \label{item:1}
$V(\delta_{\tilde P}=1)$, with $P$--degree $9$;\\
\item  \label{item:2}
$V(\delta_{\bar W}=1)$, with $P$--degree $4$;\\
\item \label{item:3}
$V(\tau_{E,2}=1)$, with $P$--degree $4$;\\
\item  \label{item:4}
$V(\bar W+W_a+W_b,\delta_{\bar W}=1)$, with $P$--degree $0$.
\end{inparaenum}
\end{proposition}

\begin{proof}
The list is an application of Proposition \ref{p:zrr}.
The only things to prove are the degree assertions. 
Since $\fL_2$ is trivial on $\tilde P$, case  \eqref{item:4} is trivial.
Case \eqref{item:1} follows from Proposition \ref{p:deg-dual-sing},
because the $P$--degree of $V(\delta_{\tilde P}=1)$
is the degree $9$ of the dual surface of the image $X_P$ of $\tilde P$
via the 
linear system $\vert 3H-E_P\vert$, which is a cubic surface with an
$A_2$ double point  (see Proposition \ref {p:deg-dual-sing}). 

As for \eqref {item:2}, note that nodal curves in $\vert
4H-3E-E_W\vert$ on $\bar W$ 
consist of a ruling in $\vert F\vert$ plus a curve $C$ in  $\vert
4H-F-3E-E_W\vert$. 
If $F$ does not intersect one of the 4 exceptional curves in $E_W$
meeting $c_0$,  
then $C=c_0+a'+b'$ and the matching curve on $\tilde P$ contains the
proper transform of $a$ and $b$, which is not allowed. 
So $F$ has to contain one of the 4 exceptional curves 
in $E_W$ meeting $c_0$. This gives rise to four pencils of singular
curves in  
$\vert 4H-3E-E_W\vert$, which produce (see Remark \ref{rem:central})
four $2$--dimensional 
linear subsystems in  $\vert \L_0\vert$, and this  implies the degree assertion. 

The degree assertion in \eqref{item:3} follows from the fact that a
$g^ 1_3$ on $E$ has 4 ramification points.
\end{proof}

\begin{proposition} [Limits of $2$-nodal curves]
\label{p:2-triangle}
The regular components of the limit Severi variety \linebreak
$\fV_2(\bar X,\L)$ are the
following, all appearing with multiplicity $1$, except \eqref
{item:4a} and \eqref {item:5a}, which have multiplicity $2$, and
\eqref {item:6a}, which has multiplicity $3$:\\ 
\begin{inparaenum}[\normalfont (i)]
\item \label {item:1a}
$V(\delta_{\tilde P}=2)$, with $P$-degree $9$;\\
\item  \label {item:2a}
$V(\delta_{\bar W}=2)$, with $P$-degree $6$;\\
\item  \label {item:3a}
$V(\delta_{\tilde P}=\delta_{\bar W}=1)$, with $P$-degree $36$;\\ 
\item  \label {item:4a}
$V(\delta_{\tilde P}=\tau_{E,2}=1)$, with $P$-degree $28$;\\
\item  \label {item:5a}
$V(\delta_{\bar W}=\tau_{E,2}=1)$, with $P$-degree $8$;\\
\item  \label {item:6a}
$V(\tau_{E,3}=1)$, with $P$-degree $3$;\\
\item  \label {item:7a}
$V(\bar W+W_a+W_b,\delta_{\bar W}=2)$, with $P$--degree $0$.
\end{inparaenum}
\end{proposition}

\begin{proof}
Again, the list is an immediate application of Proposition
\ref{p:zrr},
and the only things to prove are the degree assertions. Once more case \eqref{item:7a} is clear.

In case \eqref  {item:1a} the degree equals the number of lines on
$X_P$ (the cubic surface image of $\tilde P$), that do not contain the
double point; this is $9$. 

In case \eqref  {item:2a}, we have to
consider the  binodal curves in $\vert 4H-3E-E_W\vert$ not containing
$E$. 
Such curves split into a sum 
$\Phi_1+\Phi_2+C$, where $\Phi_1$ and $\Phi_2$ are the strict
transforms of two curves in $|F|_W$. 
They are uniquely determined by the choice of two  curves in $E_W$ meeting $c_0$: these
fix the two rulings in $\vert F\vert$ containing them, and there is a unique curve in
$|2H-E|$ containing the remaining curves in $E_W$. This shows that the degree is $6$. 

Next, the limit curves of type \eqref{item:3a} consist of  a
nodal cubic in $\vert 3H-E_P\vert_{\tilde P}$
and a nodal curve in $\vert 4H-3E-E_W\vert_{\bar W}$; a ruling
necessarily splits from the latter curve.
 Again, the splitting rulings $F$ are the ones containing one of the four curves in $E_W$ meeting $c_0$.
The curves in $|3H-2E|_{\bar W}$ containing the remaining curves in
$E_W$, fill up a pencil. 
Let $F_0$ be one of these four rulings.
The number of nodal curves  in $\vert 3H-E_P\vert_{\tilde P}$ passing
through the base point $F_0 \cap E$ and through a fixed general point on
$\tilde P$ equals   
the degree of the dual surface of $X_P$, which is 9. 
For each such curve, there is a unique curve in the aforementioned
pencil on $\tilde W$ matching it. This shows that the degree is 36. 

The general limit curve of type \eqref{item:4a}
can be identified with the general plane of 
$\P^ 3=|3H-E_P|_{\tilde P}^\vee$
which is tangent to both $X_P$ and the curve $C_E$ (image of $E$ in
$X_P$), at different points. 
The required degree is the number of such planes passing through a
general point $p$ of $X_P$.  
The planes in question 
are parametrized in $\check \P^ 3$  by a component $\Gamma_1$ of 
$\check{X}_P\cap \check{C}_E$: one needs to remove from 
$\check{X}_P\cap \check{C}_E$ the component $\Gamma_2$, the general
point of which corresponds to a plane which is tangent to
$X_P$ at a general point of $C_E$.
The latter appears with multiplicity $2$ in $\check{X}_P\cap
\check{C}_E$ by Lemma \ref{l:tangency}. Moreover, $\check{X}_P$ and $\check{C}_E$ have respective degrees $9$ and
$4$ by Proposition \ref{p:deg-dual-sing}.
Thus we have 
\[
\deg_P\bigl(V(\delta_{\tilde P}=\tau_{E,2}=1)\bigr)
=36-2\deg(\Gamma_2).
\]
To compute $\deg(\Gamma_2)$, take a general point $q=(q_0:\ldots: q_3)
\in \P^ 3$, and
let $P_q(X_P)$ be the \emph{first polar} of $X_P$ with respect to $q$,
i.e. the surface of homogeneous equation 
\[
q_0 \frac {\partial f}{\partial x_0}+\cdots+q_0 \frac {\partial f}{\partial x_3}=0,
\]
where $f=0$ is the homogeneous equation of $X_P$.
The number of planes containing $q$ and tangent to $X_P$ at a point of
$C_E$ is then equal to the number of points of $P_q(X_P)\cap C_E$,
distinct from the singular point $v$ of $X_P$.
A local computation, which can be left to the reader, shows that
$v$ appears with multiplicity $2$ in $P_p(X_P)\cap C_E$, which shows
that $\deg (\Gamma_2)=4$, whence 
$\deg_P\bigl(V(\delta_{\tilde P}=\tau_{E,2}=1)\bigr)=28$. 

In case \eqref {item:5a}, we have to determine the curves in
$\vert 4H-3E-E_W\vert$ with one node (so that some  
ruling splits) that are also tangent to $E$. As usual, the splitting
rulings are the one containing one of the four curves in $E_W$ meeting
$c_0$. Inside the residual pencil there are  2 tangent curves at
$E$. This yields the degree 8 assertion.

Finally, in case \eqref {item:6a}, the degree  equals the number of
flexes of $C_E$, which is a nodal plane cubic: this is $3$.
\end{proof}

\begin{proposition} [Limits of $3$-nodal curves]
\label{p:3-triangle}  
The regular components of the limit Severi variety \linebreak
$\fV_3(\bar X,\L)$ are the
following $0$--dimensional varieties, all appearing with multiplicity
$1$, except the ones in \eqref {bitem4} and \eqref {bitem5} appearing
with multiplicity $2$, and \eqref {bitem6}  with multiplicity $3$:\\ 
\begin{inparaenum}[\normalfont (i)]
\item \label{bitem1}
$V(\delta_{\tilde P}=3)$, which consists of
$6$ points;\\
\item  \label{bitem2}
$V(\delta_{\tilde P}=2,\delta_{\bar W}=1)$,
which consists of $36$ points;\\
\item  \label{bitem3}
$V(\delta_{\tilde P}=1,\delta_{\bar W}=2)$,
which consists of $54$ points;\\
\item  \label{bitem4}
$V(\delta_{\tilde P}=2,\tau_{E,2}=1)$,
which consists of $18$ points;\\
\item  \label{bitem5}
$V(\delta_{\tilde P}=\delta_{\bar W}=\tau_{E,2}=1)$,
which consists of $56$ points;\\
\item  \label{bitem6}
$V(\delta_{\tilde P}=\tau_{E,3}=1)$,
which consists of $18$ points;\\
\item  \label{bitem7}
$V(\bar W+W_a+W_b,\delta_{\bar W}=3)$,
which consists of $6$ points.
\end{inparaenum}
\end{proposition}

In the course of the proof, we will need the following lemma.
\begin{lemma}
\label{l:spec-pencil}
Let $p,q$ be general points on $E$.\\
\begin{inparaenum}[\normalfont (i)]
\item \label {uno} The pencil ${\mathfrak l}\subset \vert 3H-E_P\vert$ of curves containing $q$, and 
tangent to $E$ at $p$, contains exactly $7$ irreducible
nodal curves not singular at $p$.\\
\item \label{due} The pencil ${\mathfrak m}\subset \vert 3H-E_P\vert$ of curves with a contact of order $3$
with $E$ at  $p$ contains exactly $6$ irreducible nodal curves not singular at $p$. 
\end{inparaenum}\end{lemma}

\begin{proof}
First note that $\mathfrak l$ and $\mathfrak m$ are indeed pencils by Remark
\ref{rem:central}.
Let $P_{pq} \to \tilde P$ be the blow-up at  $p$
and $q$, with exceptional curves $E_p$ and $E_q$ above $p$ and $q$
respectively. Let $P'_{pq} \to P_{pq}$ be the blow-up at the
point $E\cap E_p$, with exceptional divisor $E'_p$. 
Then $\mathfrak l$ pulls back to the  linear system 
$\bigl|3H -E_P-E_p-E_q-2E'_p\bigr|$,  which induces an
elliptic fibration $P'_{pq} \to \P^1$, with
 singular fibres in number of $12$ (each counted with its multiplicity) by Lemma \ref{l:pencil}. Among them are: 
\begin{inparaenum}[(i)]
\item the proper transform of $a+b+E$, which has $3$ nodes, hence
multiplicity $3$ as a singular fibre; 
\item the unique curve of $\mathfrak l$ containing 
 the $(-2)$-curve $E_p$, which has $2$
nodes along $E_p$, hence multiplicity $2$ as a singular fibre.
\end{inparaenum}
The remaining 7 singular fibres are  the ones we want to
count.

The proof of \eqref{due}  is similar and can be left to the reader.
\end{proof}

\begin{proof}[Proof of Proposition \ref{p:3-triangle}]
There is no member of $\fL_1$ with 3 nodes on $\bar W$, because every
such curve contains one of the curves $a',b',c_0$. 

There is no member of $\fL_1$ with two nodes on $\bar W$ and a tacnode
on $E$ either. Indeed, the
 component on  $\bar W$ of such a curve
would be the proper transform of a curve of $W$ consisting of two rulings plus a curve 
in $|2H-E|$, altogether containing $Z_W$. Each of the two lines passes through one of the points
of $Z_W$ on $c_0$. The curve in $|2H-E|$  must contain  the remaining 
points of $Z_W$, hence it is uniquely determined and cannot be tangent to $E$. 

Then the list covers all remaining possible cases, and 
we only have to prove the assertion about the cardinality of the
various sets.

The limiting curves of type \eqref{bitem1}
are in one-to-one
correspondence with the unordered triples of lines distinct from $a$
and $b$ in $P$, the union of which contains the six points of $Z_P$.
There are $6$ such triples.

The limiting curves of type \eqref{bitem2}
consist of the proper transform $C_P$ in $\tilde P$
of the union of a conic and a line on $P$ containing $Z_P$, 
plus the union $C_W$ of the proper transforms in $\bar W$ 
of a curve in  $|F|$ and one in $|3H-2E|$ altogether containing $Z_W$, 
with $C_P$ and $C_W$ matching along $E$.
We have $9$ possible pencils for $C_P$, corresponding
to the choice of two points on $Z_P$, one on $a$ and one on $b$;
each such pencil determines by restriction on $E$ a line $\mathfrak
l\subset \mathfrak R$. 
There are $4$ possible pencils for $C_W$, corresponding
to the choice of one of the points of $Z_W$ on $c_0$: there is a unique
ruling containing this point, and a pencil of curves in $|3H-2E|$
containing the five remaining points in $Z_W$;
each such pencil defines a line
$\mathfrak m\subset \mathfrak R$. 
For each of the above choices, the lines $\mathfrak l$ and $\mathfrak m$
intersect at one point, whence the order $36$.

We know from the proof of Proposition \ref{p:2-triangle} that there
are six  $2$--nodal curves in $\vert 4H-3E-E_W\vert$. 
For each such curve, there is a
pencil of matching curves in $\vert 3H-E_P\vert$. This pencil contains 
$\deg (\check{X}_P)=9$ nodal curves, whence the number $54$ of
limiting curves of type \eqref{bitem3}. 

The component on $\tilde P$ of a limiting curve
of type \eqref{bitem4}
is the proper transform of the union of a conic and a
line on $P$, containing $E_P$. 
As above, there are $9$ possible choices for the line. For each such
choice, there is a pencil of conics containing the $4$ 
points of $Z_P$ not on the line. 
This pencil cuts out a $g^1_2$ on $E$, and therefore contains
$2$ curves tangent to $E$. It follows that there are $18$ limiting
curves of type \eqref{bitem4}. 

The component on $\bar W$ of a limiting curve of type \eqref{bitem5}
is the proper transform of a  ruling of $W$ plus a curve in $|3H-2E|$
tangent to $E$, altogether passing through $Z_W$. 
The line necessarily contains one of the four points of $Z_W$ on
$c_0$. There is then 
a pencil of curves in $|3H-2E|$ containing the five remaining 
points of $Z_W$. It cuts out a $g^1_2$ on $E$, hence contains $2$ curves
tangent to $E$.
For any such curve $C_W$ on $\bar W$, there is a pencil  of curves on the
$\tilde P$--side  matching it. 
By Lemma \ref{l:spec-pencil}, this pencil contains $7$ curves,
the union of which with $C_W$ is a limiting curve of type \eqref{bitem5}.
This proves that there are $56$ such limiting curves.

As for \eqref{bitem6}, 
there are $3$ members of $\mathfrak R$ that are triple points
(see the proof of Proposition \ref{p:2-triangle}). Each of them
determines a pencil of curves on the $\tilde P$--side, which contains
six $1$-nodal curves by Lemma \ref{l:spec-pencil}. 
This implies that there are $18$ limiting curves of type \eqref{bitem6}.  

Finally we have to count the members of 
$V(\bar W+W_a+W_b,\delta_{\bar W}=3)$. 
They are in one--to--one correspondence with their components on $\bar
W$, 
which decompose into the proper transform of unions $C_a\cup C_b$ of
two curves $C_a\in |2H-E-2E_a-E_b|$ and $C_b\in|2H-E-E_a-2E_b|$,
altogether containing $Z_W$. 
The curves $C_a,C_b$ must contain the two base points on $b',a'$
respectively. We conclude that each limiting curve of type
\eqref{bitem7}
corresponds to a partition of the 4 points of $Z_W$ on $c_0$ in two
disjoint sets of two points, and the assertion follows.
\end{proof}

In conclusion, the following is an immediate consequence of
Propositions \ref{p:1-triangle}, \ref{p:2-triangle}, and
\ref{p:3-triangle}, together with the formula \eqref{degree}.

\begin{corollary} [preliminary version of Theorem \ref{T:triangle}]
\label{coro:D}
Let $a,b,c$ be three independent  lines in the projective
plane, and $Z$ be a degree $12$ divisor on $a+b+c$ cut out by a
general quartic curve. 
We consider the $3$--dimensional sub--linear system $\mathcal V$  of
$|\O_{\P^2}(4)|$ parametrizing curves containing $Z$,
and we let, for $1\leqslant \delta\leqslant 3$,
$\mathcal V_\delta$ be the Zariski closure in $\mathcal V$ of the
codimension $\delta$ locally closed subset parametrizing irreducible
$\delta$--nodal curves.
One has
\begin{equation}
\label{ineq-ThmD}
\deg(\mathcal{V}_1)\geqslant 21,\quad
\deg(\mathcal{V}_2)\geqslant 132,\quad
\text{and}\quad
\deg(\mathcal{V}_3)\geqslant 304.
\end{equation}
\end{corollary}

\begin{remark}
\label{rem:D}
\begin{inparaenum}[\bf (a)]
\item
{\bf (Theorem \ref{T:triangle})}
The three inequalities in \eqref{ineq-ThmD} above are actually
equalities. 
This is proved in \S\ref{s:tetra-concl},
by using both \eqref{ineq-ThmD} and the degrees of the Severi
varieties of a general quartic surface, given by Proposition
\ref{p:deg-dual}.

Incidentally, this proves that $\varpi: \bar X\to \Delta$ is a good
model for the family $\hat f: \hat S\to \Delta$ obtained by
blowing--up $S=\P^2\times \Delta$ along $Z$, and endowed with the
appropriate subline bundle of $\O_{\hat S}(1)$.

\item
\label{rem:D-red}
In particular, we have $\fV_3=\reg\fV_3$.
It then follows from Remark \ref{r:smoothness} that the relative Severi
variety $V_3(\bar X,\L)$ is smooth at the points of $\fV_3$.
This implies 
that the general fibre of $V_3(\bar X,\L)$ is reduced.
Therefore, in the setting of Corollary \ref{coro:D}, if $a+b+c$ and
$Z$ are sufficiently general, then $\mathcal{V}_3$ consists of
$304$ distinct points.
\end{inparaenum}
\end{remark}

\section{Application to the irreducibility of
Severi varieties and to the monodromy action}
\label{S:irreducibility}

Set $\hil B= \vert \O_{\P^ 3}(4)\vert$.
We have the \emph{universal family} $p: \mathcal P\to \hil B$, such
that the fibre of $p$ over $S\in \mathcal B$ is the linear 
system $\vert \O_S(1)\vert$. The variety $\mathcal P$ is a component
of the \emph{flag Hilbert scheme}, namely the one
parametrizing pairs $(C,S)$, where $C$ is a plane quartic curve in $\P^ 3$ and $S\in \mathcal B$ contains $C$.
So $\mathcal P\subset \mathcal B\times \mathcal W$, where $\mathcal W$
is the component of the Hilbert scheme of curves in $\P^ 3$ whose
general point corresponds to a plane quartic. The map $p$ is the
projection to the first factor; we let $q$ be the projection to the second factor.

Denote by $\mathcal U\subset \mathcal B$ the open subset parametrizing
smooth surfaces, and set 
$\mathcal P_\mathcal U=p^ {-1}(\mathcal U)$.
Inside $\mathcal P_\mathcal U$ we have the \emph{universal Severi
  varieties} $\ring{\hil V}_\delta$,
$1\leqslant \delta\leqslant 3$, such that for all 
$S\in \mathcal U$, the fibre of $\ring{\hil V}_\delta$ over $S$ is the
Severi variety $V_\delta(S, \O_S(1))$.
Since $S$ is a $K3$ surface, we know that for all irreducible
components $V$ of $V_\delta(S, \O_S(1))$,
we have $\dim (V)=3-\delta$, so that all
components of $\ring{\hil V}_\delta$ have codimension $\delta$ in
$\mathcal P_\mathcal U$.
We then let $\hil V_\delta$ be the Zariski closure of $\ring{\hil
  V}_\delta$ in $\hil P$; we will call it universal Severi variety as
well.

The following is immediate (and it is a special case of a more general result, see \cite {cd-universal}):

\begin{proposition}\label {prop:irr}  The universal Severi varieties $\mathcal V_\delta$ are irreducible for $1\leqslant \delta\leqslant 3$.
\end{proposition}

\begin{proof} It suffices to consider the projection $q: \mathcal
  V_\delta\to \mathcal W$, and notice that its image is the
  irreducible variety whose general point corresponds to a quartic
  curve with $\delta$ nodes
(cf. \cite{harris, harris-morrison}),
and that the fibres are all irreducible of the same dimension 20. 
\end{proof}

Note that the irreducibility of $\mathcal V_1$ also follows from the
fact that for all $S\in \mathcal U$, we have $V_1(S,\O_S(1))
\cong \check S$. To the other extreme, $p: \ring{\mathcal V}_3\to
\mathcal U$ is a finite cover of degree 3200.  
We will denote by $G_{4,3}\sg \mathfrak S_{3200}$ the \emph{monodromy group of this covering}, which acts transitively because
$\mathcal V_3$ is irreducible.

\subsection{The irreducibility of the family of binodal plane 
sections of a general quartic surface} \label {ssec:irred}

In the middle we have $p: \ring{\mathcal V}_2\to \mathcal U$. Though
$\mathcal V_2$ is irreducible, we cannot deduce from this that for the
general $S\in \mathcal U$, the Severi variety $V_2(S, \O_S(1))$ (i.e.,
the  curve of binodal plane sections of $S$) is irreducible. Though
commonly accepted as a known fact, we have not been able to find
any proof of this in the current literature. It is the purpose
of this paragraph to provide a proof of this fact. 

In any event, we have a commutative diagram similar to the one in \eqref {eq:stein}
\[\xymatrix{
{ \hil V}_2' \ar[r]^\nu \ar[dr] \ar[d]_{p'}
& \ring{\mathcal V}_2  \ar[d]^p  \\
\mathcal U'  \ar[r]_{f} &  \hil U
}\]
where $\nu$ is the normalization of $\ring{\hil V}_2$, and $f\circ p'$ is the
Stein factorization of $p\circ \nu: {\hil V}_2'\to \hil U$.
The morphism
$f: \mathcal U'\to \mathcal U$ is finite, of degree $h$ equal to the
number of irreducible components of $V_2(S, \O_S(1))$ for 
general $S\in \mathcal U$.
The monodromy group of this covering acts
transitively. This ensures that, for general $S\in \mathcal U$,
all irreducible components of $V_2(S,\O_S(1))$ have the same degree,
which we denote by $n$. 
By Proposition \ref {prop:irrcomp}, we have $n \geqslant
36$.

\begin{theorem}\label{thm:V2} If $S\subset \P^ 3$ is a general quartic
  surface,
then the curve $V_2(S, \O_S(1))$ is irreducible. 
\end{theorem}

\begin{proof}
Let $S_0$ be a general quartic Kummer surface, and
$f:\hil S\to \Delta$ a family of surfaces induced as in Example
\ref{ex:pi3} by a pencil generated by $S_0$ and a general quartic
$S_\infty$.
Given two distinct nodes $p$ and $q$ of $S_0$, we denote by $\mathfrak l_ {pq}$
the pencil of plane sections of $S_0$ passing through $p$ and $q$.
Corollary \ref {cor:thmB2ii} asserts that the union of these lines,
each counted with multiplicity 4, is the crude limit Severi variety
$\cru \fV_2(\hil S,\O_{\hil S}(1))$.

Let $\Gamma_t$ be an irreducible component  of
$V_1(S_t,\O_{S_t}(1))$, for $t\in \Delta^*$, and let $\Gamma_0$ be its (crude)
limit as $t$ tends to $0$, which consists of
a certain number of (quadruple) curves $\mathfrak l_{pq}$. Note that, by Proposition \ref {p:2-kummer}, the pull--back of the lines
$\mathfrak l_ {pq}$ to the good limit constructed in  \S \ref {S:Kummer-degen} all appear
with multiplicity 1 in the limit Severi variety. This yields that,
 if $\mathfrak l$ is an irreducible component of $\Gamma_0$,
then it cannot be in the limit of an irreducible component
$\Gamma'_t$ of $V_1(S_t,\O_{S_t}(1))$ other than $\Gamma_t$.

We shall prove successively the following claims,
the last one of which proves the
theorem:\\
\begin{inparaenum}[(i)]
\item \label{auno}
$\Gamma_0$ contains two curves  $\mathfrak l_{pq}$, $\mathfrak l_{pq'}$,
with $q\neq q'$;\\ 
\item \label {adue}
$\Gamma_0$  contains two curves  $\mathfrak l_{pq}$, $\mathfrak l_{pq'}$,
with $q\neq q'$, and  $p,q,q'$ on a contact conic $D$ of $S_0$;\\
\item \label{atre}
there is a contact conic $D$ of $S_0$, such that 
$\Gamma_0$ contains all curves ${\mathfrak l}_{pq}$ with $p,q \in D$; \\
\item \label{aquattro}
 property \eqref{atre} holds for every contact conic of $S_0$;\\
\item \label{acinque}
$\Gamma_0$ contains all curves $\mathfrak l_{pq}$.
\end{inparaenum}

\smallskip
If $\Gamma_0$ does not verify \eqref {auno}, then it contains at most
8 curves of type $\mathfrak l_{pq}$, a contradiction
to $n \geqslant 36$. 
To prove \eqref {adue}, we consider two curves  $\mathfrak l_{pq}$ and
$\mathfrak l_{pq'}$ contained in $\Gamma_0$, and assume that $p,q,q'$ do
not  lie on a contact conic, otherwise there is nothing to prove.
Consider a degeneration of $S_0$ to a product Kummer surface
$\mathsf{S}$, and let $\mathsf{p,q,q'}$ be the limits on $\mathsf{S}$
of $p,q,q'$ respectively:
they are necessarily in one of the three configurations depicted in 
Figure \ref{f:triples-exc}.
In all three cases, we can exchange two horizontal lines in $\sf S$
(as indicated in Figure \ref{f:triples-exc}),
thus moving $\sf q'$ to $\sf q''$,
in such a way that $\sf p$ and $\sf q$ remain fixed,
and there is a limit in $\sf S$ of contact conics that contains
the three points $\sf p$, $\sf q'$, and $\sf q''$.
Accordingly, there is an element $\gamma\in G_{16,6}$ 
mapping $p,q,q'$ to $p,q,q''$ respectively, such that $p,q',q''$ lie
on a contact conic $D$ of $S_0$.
Then $\gamma(\Gamma_0)$ contains $\gamma({\mathfrak l}_{pq})=
{\mathfrak l}_{pq}$. By the remark  preceding the statement of \eqref {auno}--\eqref {acinque},  we have $\gamma(\Gamma_0)=\Gamma_0$.
It follows that $\Gamma_0$ contains ${\mathfrak l}_{pq'}$ and ${\mathfrak
  l}_{pq''}$, and therefore satisfies \eqref{adue}.

\begin{figure}
\begin{center}
\includegraphics[height=2.5cm]{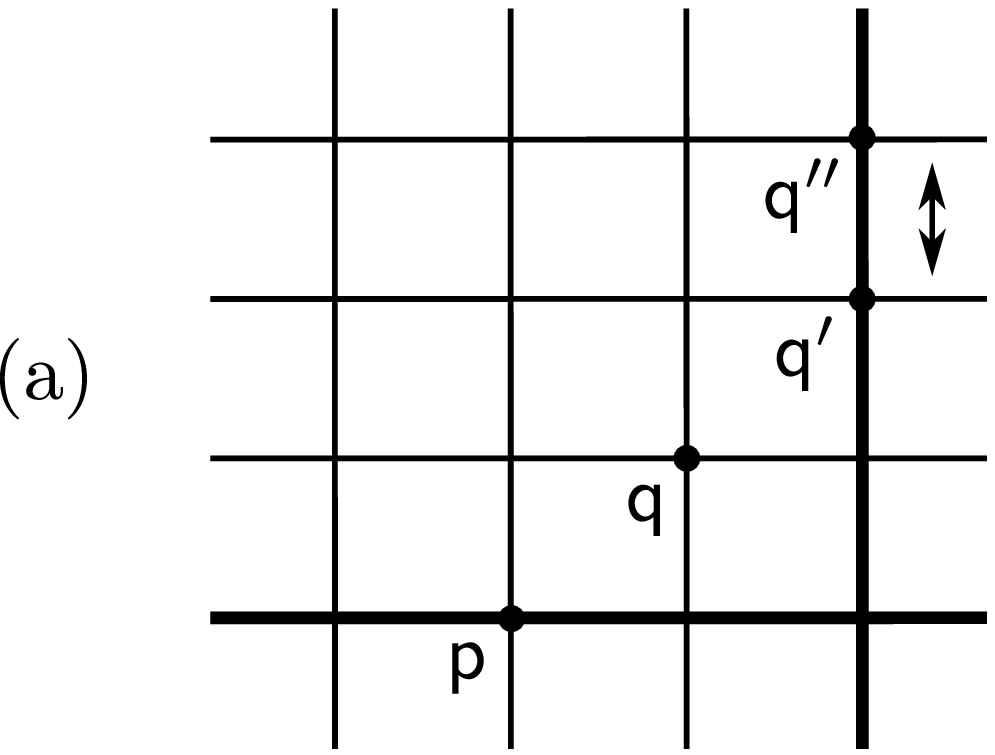}
\hspace{1cm}
\includegraphics[height=2.5cm]{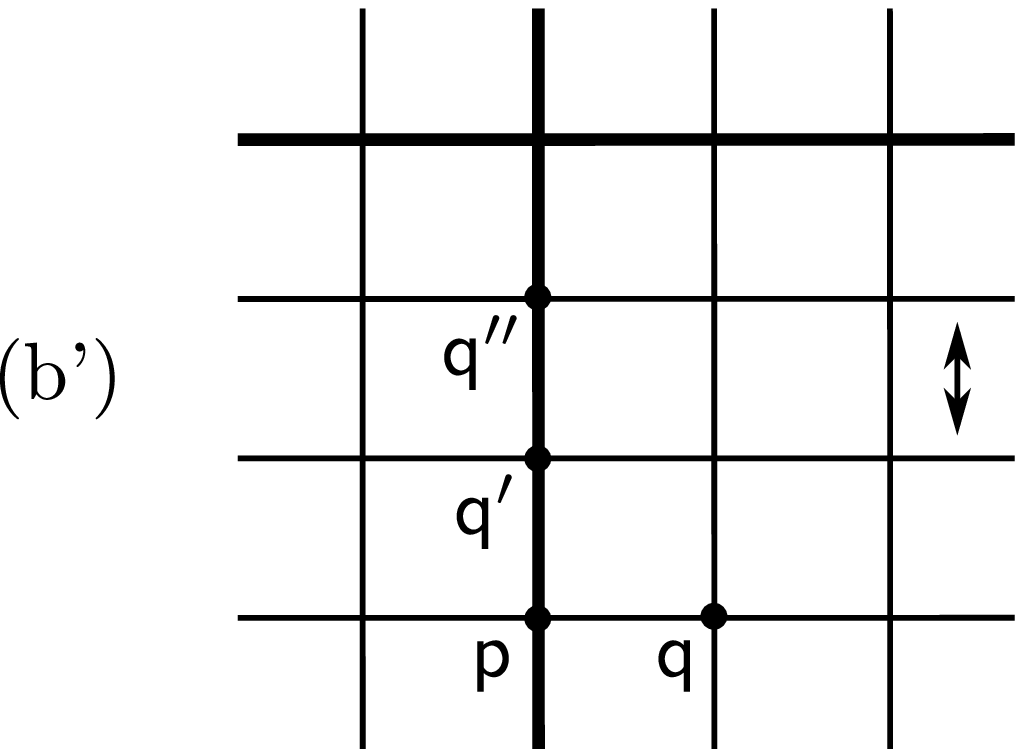}
\hspace{1cm}
\includegraphics[height=2.5cm]{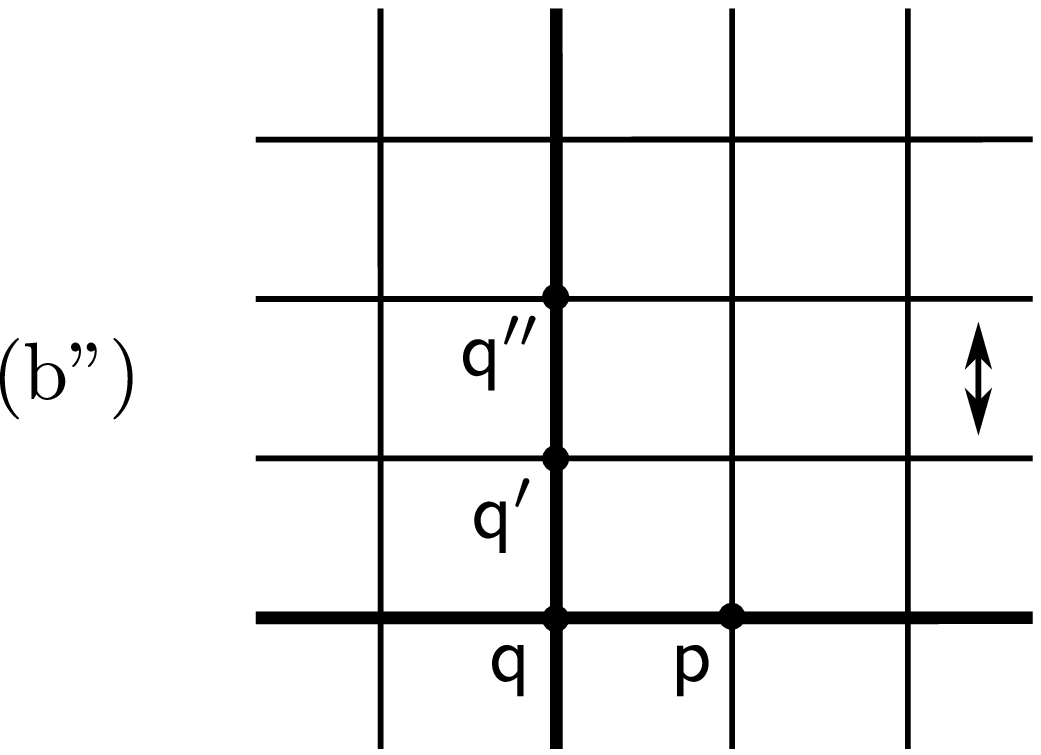}
\end{center}
\caption{How to obtain three double points on a double conic}
\label{f:triples-exc}
\end{figure}

Claim \eqref{atre} follows from \eqref{adue} and the fact
that the monodromy acts as $\mathfrak S_6$ on the set of nodes lying
on $D$
(see Lemma \ref {l:S_6}).
As for \eqref{aquattro}, let $D'$ be any other contact conic of
$S_0$. There exists $\gamma\in G_{16,6}$ interchanging $D$ and $D'$
(again by Lemma \ref {l:S_6}).
The action of $\gamma$ preserves $D\cap D'=\{x,y\}$. 
We know that
$\Gamma_0$ contains $\mathfrak l_{xy}+\mathfrak l_{xy'}$ with $y'\in D$
different from $y$. Then the same argument as above yields that
$\Gamma_0$ contains $\mathfrak l_{\gamma(x)\gamma(y')}$, where
$\gamma(x)\in \{x,y\}$ and $\gamma(y')\in D'-\{x,y\}$.
This implies that $\Gamma_0$ satisfies \eqref{adue} for $D'$, and 
therefore \eqref{atre} holds for $D'$.
Finally \eqref{aquattro} implies \eqref{acinque}.
\end{proof} 

It is natural to conjecture that Theorem \ref {thm:V2}  is a particular case of the following  general statement:

\begin {conjecture}\label{cong:irrednod} Let $S\subset \P^ 3$ be a general surface of degree $d\ge 4$. Then the following curves are irreducible:\\
\begin{inparaenum}[(i)]
\item $V_2(S, \O_S(1))$, the curve of binodal plane sections of $S$;\\
\item $V_\kappa(S, \O_S(1))$, the curve of cuspidal plane sections of $S$.\\
\end{inparaenum}
\end{conjecture}

We hope to come back to this in a future work. 

\subsection{Some noteworthy subgroups of  $G_{4,3}
\sg \sym_{3200}$}
\label{ssec:monod}

In this section we use the degenerations we studied in \S\S \ref
{S:4planes} and \ref {S:Kummer-degen} 
to give some information on the monodromy group $G_{4,3}$ of
$p:\ring{\hil V}_3\to \hil U$.
We will use the following:

\begin{remark}\label{rem:restrict} Let $f: X\to Y$ be a dominant,
  generically finite morphism of degree $n$  between projective irreducible varieties, with monodromy group $G\sg \sym_n$.
Let $V\subset Y$ be an irreducible codimension $1$ subvariety,
the generic point of which is a smooth point of $Y$.
Then $f_V:=\left. f \right |_{f^ {-1}(V)}  : f^ {-1}(V)\to V$
is still generically finite, with monodromy group $G_V$.
If $V$ is not contained in the branch locus of $f$,
then  $G_V\sg G$. 

Suppose to the contrary that $V$ is contained in a component of the
branch locus of $f$. Then
$G_V\sg \sym_{n_V}$, with $n_V:=\deg f_V < n$, and $G_V$ is no longer
a subgroup of $G$.
We can however consider the  \emph{local monodromy group 
$\loc G_V$ of $f$ around $V$}, i.e. the subgroup of $G\sg \sym_n$
generated by  permutations associated to non--trivial loops 
\emph{turning around} $V$.
Precisely: 
let $U_V$ be a tubular neighbourhood  of $V$ in $Y$;
then $\loc G_V$ is the image in $G$ of the subgroup
$\pi_1\bigl( U_V - V \bigr)$ of  $\pi_1(Y-V)$.

There is an  epimorphism $\loc G_V \to G_V$,
obtained by deforming loops
in $U_V - V$ to loops in $V$.
We let $\loc H_V$ be the kernel of this epimorphism, so that one has
the exact sequence of groups
\begin{equation}
\label{loc-monodromy-seq}
1\to \loc H_V\to \loc G_V \to G_V \to 1.
\end{equation}
\end{remark}

\medskip
We first apply this  to the degeneration studied in \S \ref
{S:4planes}. To this end, we consider the 12--dimensional subvariety
$\hil T$ of  $\hil B$ which is the Zariski closure of the
set of fourtuples of distinct planes. 
Let $f: \tetra{\tilde{\hil B}}\to \hil B$ be the blow--up of 
${\hil B}$
along $\hil T$, with exceptional divisor 
$\tilde {\hil T}$.
The proof of the following lemma  (similar to Lemma \ref
{l:tg-id}) can be left to the reader: 

\begin{lemma}\label{lem:blow}
Let $X$ be a general point of $\hil T$. Then the fibre of $f$ over $X$
can be identified with 
$\vert \O_\Lambda(4)\vert$, where $\Lambda={\rm Sing}(X)$.
\end{lemma} 

Thus, for general $X\in \hil T$, a general point of the fibre of $f$
over $X$ can be identified with a pair $(X,D)$, with
$D\in  \vert \O_\Lambda(4)\vert$ general,  where 
$\Lambda={\rm Sing}(X)$.
Consider a family $f:S\to \Delta$ of surfaces in $\P^3$, induced as in
Example \ref{ex:pi3} by a pencil $\mathfrak{l}$
generated by $X$ and a general quartic; 
then the singular locus of $S$ is a member of
$\vert \O_\Lambda(4)\vert$,
which corresponds to the tangent direction normal to $\hil T$ defined
by $\mathfrak l$ in $\hil B$.

\medskip
Now the universal family $p: \mathcal P \to \mathcal B$ can be pulled
back to $\tilde p: \tilde {\mathcal P}\to \tetra{\hil{\tilde B}}$,
and the analysis of \S \ref {S:4planes} tells us that we have
a generically finite map
$\tilde p: \tilde{\hil V}_3 \to \tetra{\tilde{\hil B}}$, which
restricts to $p: \ring{\hil V}_3\to \hil U$ over $\hil U$,
and such that $\tilde {\mathcal T}$ is in the
branch locus of $\tilde p$. 
We let $\tetra G$ be the 
monodromy group of 
$\tilde p: \tilde{\hil V}_3 \to \tetra{\tilde{\hil B}}$
on $\tilde{\hil T}$,
and $\loc{\tetra G}$, resp. $\loc{\tetra H}$, be
as in \eqref{loc-monodromy-seq}.

\begin{proposition}\label{prop:monodrom1}
Consider a general $(X,D)\in \tilde {\mathcal B_t}$.
One has:\\
\begin{inparaenum}[\normalfont (a)]
\item \label {it:111} $\tetra G\cong \prod _{i=1}^ 4 G_i$, where:\\
\begin{inparaenum}[(i)]
\item  \label {it:222}  $G_1\cong \mathfrak S_{1024}$  is the monodromy
  group of planes containing three points in $D$, but no
  edge of $X$; \\
\item  \label {it:333}  $G_2\cong \sym_4\times \sym_3\times
(\sym_{4})^ 2$  is
  the monodromy group of planes containing a vertex of $X$ and two
  points in $D$, but no edge of $X$;\\ 
\item  \label {it:444}  $G_3\cong \mathfrak S_6\times \mathfrak S_4$  is the
  monodromy group of planes containing an edge of $X$, and a point in
  $D$ on the opposite edge of $X$;\\ 
\item \label {it:555}   $G_4\cong  \mathfrak S_4$  is the monodromy group
  of faces of $X$;\\ 
\end{inparaenum}
\item  \label {it:666}  
$\loc{\tetra H}\cong \mathfrak S_3\times G \times H$,
where $G\sg \sym_{16}$ is the monodromy group  of 
 bitangent lines to $1$--nodal plane quartics as in Proposition
\ref{p:projectW},
 and $H\sg \sym_{304}$ is the monodromy group of irreducible  trinodal
curves in the linear system of quartic curves with 12 base points at a
general divisor of $\vert \O_{a+b+c}(4)\vert$, with $a,b,c$ three
lines not in a pencil (see \S \ref {S:triangle}).  
\end{inparaenum}
\end{proposition}

\begin{proof} The proof follows from Corollary \ref {cor:thmBiii}.
  Recall  that a group $G\sg \sym_n$ is equal to $\mathfrak S_n$,
  if and only if it contains a transposition and it is doubly
  transitive. Using this, it is easy to verify  the assertions in
  \eqref  {it:222}--\eqref  {it:555}
(see \cite[p.698]{harris-gal}).
As for \eqref  {it:666},  the
  factor $\mathfrak S_3$ comes from the fact that the monodromy acts as
  the full symmetric group on a  general line section of the
  irreducible cubic surface $T$ as in Proposition \ref
  {p:cubics}. \end{proof}

Analogous considerations can be made for the degeneration studied in
\S  \ref {S:Kummer-degen}. In that case, we consider the
18--dimensional subvariety $\hil K$ of  $\hil B$ which is
the Zariski closure of the set $\ring{\hil K}$ of Kummer surfaces. 
Let  $g: \kum{\tilde{\hil B}} \to \hil B$ be the blow--up 
along $\hil K$, with exceptional divisor $\tilde{\hil K}$.
In this case we have:

\begin{lemma}\label{lem:blow2} Let $X\in \hil K$ be a general point,
  with singular locus $N$. Then the fibre of $g$ over $X$ can be
  identified with $\vert \O_N(4)\vert\cong \P^ {15}$.\end{lemma} 

The universal family $p: \mathcal P \to \mathcal B$ can be
pulled back to $\tilde p: \tilde {\mathcal P}\to 
\kum{\tilde{\hil B}}$.
The analysis of \S \ref  {S:Kummer-degen} tells us that
we have a map
$\tilde p: \tilde{\hil V}_3 \to \kum{\tilde{\hil B}}$, 
generically finite over $\tilde {\hil K}$, which
is  in the branch locus of $\tilde p$. 
We let $\kum G$ be the 
monodromy group of $\tilde p$ on $\tilde {\mathcal K}$,
and set $\loc{\kum G}$ and $\loc{\kum H}$
as in \eqref{loc-monodromy-seq}.

\begin{proposition}\label{prop:monodrom2} One has:\\
\begin{inparaenum}[\normalfont (a)]
\item \label {it:sic}
$\kum G\cong G_{16,6}\times G'$, where $G'$ is the monodromy group of unordered
triples of distinct nodes of a general Kummer surface which do not lie on a
contact conic
(see \S\ref{s:monodromy} for the definition of $G_{16,6}$);\\ 
\item \label{it:tic}
$\loc{\kum H}\cong \sym_8 \times G''$,
  where $G''$ is the monodromy group of the tritangent planes to a
  rational curve $B$ of degree $8$ as in the statement of Proposition 
\ref{p:2:1}. 
\end{inparaenum}
\end{proposition}

\begin{proof} Part \eqref {it:sic} follows right away from Proposition
  \ref {p:3-4planes_Kummer}. Part \eqref   {it:tic} also follows,
  since the monodromy on complete intersections of three general
  quadrics in $\P^ 3$ (which gives the multiplicity 8 in \eqref
  {item:aa} of  
Proposition \ref {p:3-4planes_Kummer}) is clearly the full symmetric
group.  \end{proof} 

Concerning the group $G'$ appearing in Proposition \ref {prop:monodrom2}, \eqref {it:sic},
remember that it acts with at most two orbits on the set of of unordered
triples of distinct of nodes of a general Kummer surface which do not lie on a
contact conic (see Proposition \ref {p:3-transitivity}, \eqref {it:beta}).

\begin{closing}

\vskip .2cm \noindent
\textsc{%
Dipartimento di Matematica, 
Universit\`a degli Studi di Roma Tor Vergata,
Via della Ricerca Scientifica,
 00133 Roma, Italy} \\
\texttt{cilibert@mat.uniroma2.it}

\vskip .4cm \noindent
\textsc{%
Institut de Math\'ematiques de Toulouse (CNRS UMR 5219),
Universit\'e Paul Sabatier,
 31062 Tou\-louse Cedex 9, France} \\
\texttt{thomas.dedieu@m4x.org}

\end{closing}

\end{document}